\documentclass[11pt]{article}

\title{\textbf{Seiberg--Witten and Gromov invariants for self-dual harmonic 2-forms}}

\author{\Large Chris Gerig}

\AtEndDocument{\bigskip{\footnotesize
  \textsc{Department of Mathematics, UC Berkeley, CA 94720, USA} \par  
  \textit{E-mail address:} \texttt{cgerig@berkeley.edu}
}}

\usepackage{hyperref}
\usepackage[initials]{amsrefs}

\usepackage{fullpage}
\usepackage{amssymb,latexsym,amsmath,amsthm,amscd,dsfont,graphicx}
\usepackage{comment}
\usepackage{enumerate}
\usepackage[usenames,dvipsnames,svgnames,table]{xcolor}

\usepackage{epigraph}
\usepackage{etoolbox}

\makeatletter
\patchcmd{\epigraph}{\@epitext{#1}}{\itshape\@epitext{#1}}{}{}
\makeatother

\numberwithin{equation}{section}

\newtheorem{theorem}{Theorem}[section]
\newtheorem{prop}[theorem]{Proposition}
\newtheorem{cor}[theorem]{Corollary}
\newtheorem{lemma}[theorem]{Lemma}
\newtheorem{lemma-definition}[theorem]{Lemma-Definition}

\theoremstyle{definition}
\newtheorem{definition}[theorem]{Definition}
\newtheorem{remark}[theorem]{Remark}

\newtheorem{notation}[theorem]{Notation}

\renewcommand{\SS}{{\mathbb S}}
\renewcommand{\AA}{{\mathbb A}}

\newcommand{\CC}{{\mathbb C}}

\newcommand{\RR}{{\mathbb R}}
\newcommand{\MM}{{\mathbb M}}
\newcommand{\NN}{{\mathbb N}}
\newcommand{\ZZ}{{\mathbb Z}}
\newcommand{\TT}{{\mathbb T}}
\newcommand{\FF}{{\mathbb F}}
\newcommand{\HH}{{\mathbb H}}

\newcommand{\nN}{{\mathcal N}}

\newcommand{\cC}{{\mathcal C}}
\newcommand{\eE}{{\mathcal E}}
\newcommand{\mM}{{\mathcal M}}
\newcommand{\uU}{{\mathcal U}}
\newcommand{\vV}{{\mathcal V}}
\newcommand{\bB}{{\mathcal B}}
\newcommand{\tT}{{\mathcal T}}
\newcommand{\iI}{{\mathcal I}}

\newcommand{\kK}{{\mathcal K}}

\newcommand{\lL}{{\mathcal L}}

\newcommand{\zZ}{{\mathcal Z}}
\newcommand{\rR}{{\mathcal R}}

\newcommand{\yY}{{\mathcal Y}}

\newcommand{\op}{\operatorname}

\newcommand{\Spinc}{\op{Spin}^c}

\newcommand{\Hom}{\op{Hom}}
\newcommand{\End}{\op{End}}
\newcommand{\Conn}{\op{Conn}}
\renewcommand{\ker}{\op{Ker}}
\newcommand{\coker}{\op{Coker}}
\newcommand{\im}{\op{Im}}
\renewcommand{\dbar}{\overline{\partial}}

\newcommand{\cl}{\op{cl}}
\newcommand{\PD}{\op{PD}}
\newcommand{\gr}{\op{gr}}

\newcommand{\dv}{d\text{vol}}
\newcommand{\1}{\mathds{1}}
\newcommand{\e}{\varepsilon}

\newcommand{\ind}{\op{ind}}
\newcommand{\dist}{\op{dist}}
\newcommand{\crit}{\op{crit}}
\newcommand{\grad}{\op{grad}}

\newcommand{\bigtimes}{\raisebox{-1pt}{\text{\LARGE$\times$}}}

\newcommand{\fo}{\mathfrak o}
\newcommand{\fu}{\mathfrak u}
\newcommand{\fv}{\mathfrak v}

\newcommand{\s}{\mathfrak s}
\newcommand{\h}{\mathfrak h}
\renewcommand{\a}{\mathfrak a}
\renewcommand{\b}{\mathfrak b}
\renewcommand{\t}{\mathfrak t}

\newcommand{\p}{\mathfrak p}
\renewcommand{\q}{\mathfrak q}
\newcommand{\C}{\mathfrak C}

\newcommand{\D}{\mathfrak D}
\renewcommand{\d}{\mathfrak d}
\newcommand{\M}{\mathfrak M}
\renewcommand{\c}{\mathfrak c}
\newcommand{\bc}{\boldsymbol\c}

\newcommand{\bA}{\mathbf A}

\newcommand{\Hfrom}{\widehat{\mathit{HM}}}
\newcommand{\Cfrom}{\widehat{\mathit{CM}}}

    \RequirePackage{rotating}                   % Case (2)
    \def\Hto{%
       \setbox0=\hbox{$\widehat{\mathit{HM}}$}
       \setbox1=\hbox{$\mathit{HM}$}
       \dimen0=1.1\ht0
       \advance\dimen0 by 1.17\ht1
       \smash{\mskip2mu\raise\dimen0\rlap{%
          \begin{turn}{180}
              {$\widehat{\phantom{\mathit{HM}}}$}
           \end{turn}} \mskip-2mu    
                \mathit{HM}
    }{\vphantom{\widehat{\mathit{HM}}}}{}}
    \RequirePackage{rotating}                   % Case (2)
    \def\Cto{%
       \setbox0=\hbox{$\widehat{\mathit{CM}}$}
       \setbox1=\hbox{$\mathit{CM}$}
       \dimen0=1.1\ht0
       \advance\dimen0 by 1.17\ht1
       \smash{\mskip2mu\raise\dimen0\rlap{%
          \begin{turn}{180}
              {$\widehat{\phantom{\mathit{CM}}}$}
           \end{turn}} \mskip-2mu    
                \mathit{CM}
    }{\vphantom{\widehat{\mathit{CM}}}}{}}

\definecolor{blue}{rgb}{0,0,1}
\definecolor{red}{rgb}{1,0,0}
\definecolor{green}{rgb}{0,.7,0}

\begin{document}
\maketitle

\begin{abstract}
This is the sequel to \cite{Gerig:taming} which gives an extension of Taubes' ``SW=Gr'' theorem to non-symplectic 4-manifolds. The main result of this paper asserts the following. Whenever the Seiberg--Witten invariants are defined over a closed minimal 4-manifold $X$, they are equivalent modulo 2 to ``near-symplectic'' Gromov invariants in the presence of certain self-dual harmonic 2-forms on $X$. A version for non-minimal 4-manifolds is also proved. A corollary to Morse theory on 3-manifolds is also announced, recovering a result of Hutchings--Lee--Turaev about the 3-dimensional Seiberg--Witten invariants.
\end{abstract}

%%%%%%%%%%%%%%%%%%%%%%%%%%%%%%%%%%%%%%%%%%%%%%%%%%%%%%%%%%
%%%%%%%%%%%%%%%%%%%%%%%%%%%%%%%%%%%%%%%%%%%%%%%%%%%%%%%%%%
\section{Introduction}

\indent\indent
The purpose of this paper is to present a proof of the assertion that a compact 4-manifold has its Seiberg--Witten invariants
equal to its near-symplectic Gromov invariants. The latter are suitable counts of closed and punctured Riemann surfaces (which may be disconnected and multiply covered) in the complement of certain smoothly embedded circles in the 4-manifold. When the 4-manifold is equipped with a symplectic 2-form, this was already known by Taubes \cite{Taubes:SWGrBook}, in which case there are only closed surfaces and no smoothly embedded circles. When the 4-manifold $X$ is not symplectic, we still assume $b^2_+(X)>0$. Then given a generic Riemannian metric on $X$ there are nontrivial closed self-dual (hence harmonic) 2-forms which vanish transversally along a disjoint union of circles in $X$ and are symplectic elsewhere \cite{LuttingerSimpson, Honda:transversality, LeBrun:Yamabe}. Such a 2-form $\omega$ is a \textit{near-symplectic form}, and it is through the use of them that a version of ``$SW=Gr$'' is demonstrated below.

\begin{remark}
We assume the reader is familiar with the previous paper \cite{Gerig:taming}, and when necessary we point the reader to specific locations in it. See its introduction for further motivation of the main theorem of the current paper.
\end{remark}

Throughout this paper, $(X,g)$ denotes a closed connected oriented smooth Riemannian 4-manifold with $b^2_+(X)\ge1$, and $\omega$ denotes a self-dual near-symplectic form on $X$ whose zero set
$$Z:=\omega^{-1}(0)$$
has $N\ge0$ components, all of which are untwisted\footnote{We may also allow components of $Z$ to be twisted zero-circles which are non-contractible in $X$ (see \cite{Gerig:taming}*{Appendix}), but we restrict our attention to untwisted zero-circles for simplicity of notation.} zero-circles (see \cite{Gerig:taming}*{\S1.2} for the notion of (un)twisted zero-circles). Such harmonic 2-forms always exist, but the parity of $N$ must be the same as that of $1-b^1(X)+b^2_+(X)$.

Fix an ordering of the zero-circles of $\omega$. In \cite{Gerig:taming} we defined the near-symplectic Gromov invariants
$$Gr_{X,\omega}:\Spinc(X)\to\Lambda^* H^1(X;\ZZ)$$
in terms of counts of pseudoholomorphic curves in a certain completion of $X-Z$, for which they a priori depend on an almost complex structure on $X-Z$. Here is the main theorem of this paper for $X$ minimal,\footnote{A manifold is \textit{minimal} if there are no \textit{exceptional spheres}, smoothly embedded 2-spheres of self-intersection $-1$.}  from which it follows that $Gr_{X,\omega}$ over $\ZZ/2\ZZ$ are indeed smooth invariants of $X$.

\begin{theorem}
\label{thm:SWGr1}
Given a minimal 4-manifold $(X,\omega)$ as above and any $\s\in\Spinc(X)$,
$$Gr_{X,\omega}(\s)\equiv SW_X(\s){\mod2}\indent\in\Lambda^* H^1(X;\ZZ)\otimes\ZZ/2\ZZ$$
where $SW_X(\s)$ is the Seiberg--Witten invariant defined in Definition~\ref{defn:SW} and $\omega$ determines the chamber for defining the Seiberg--Witten invariants when $b^2_+(X)=1$.
\end{theorem}

When $X$ is not minimal, there is a version of this theorem (see Theorem~\ref{thm:SWGr2}) but it is more involved. In particular, it cannot be demonstrated for all spin-c structures on $X$. This issue also appears in Taubes' Gromov invariants of a closed symplectic 4-manifold, but thanks to the blow-up formula for the Seiberg--Witten invariants (\cite{OS:Thom}*{Theorem 2.2} and \cite{LiLiu:SW=Gr}*{Proposition 4.3}) we may as well suppose that $X$ is minimal.

\begin{remark}
The definition of the Seiberg--Witten invariants over $\ZZ$ requires a choice of homology orientation of $X$. As we will explain in Section~\ref{Homology orientations}, this choice is equivalent to a choice of ordering of the zero-circles of $\omega$ plus a choice of homology orientation of the cobordism obtained from $X$ by removing tubular neighborhoods of the zero-circles. We expect that there is a canonical homology orientation of this cobordism determined by $\omega$, and that Theorem~\ref{thm:SWGr1} can be lifted to $\ZZ$ coefficients.
\end{remark}

%%%%%%%%%%%%%%%%%%%%%%%%%%%%%%%%%%%%%%%%%%%%%%%%%%%%%%
\subsection{Outline of remainder of paper}

\indent\indent
What follows is a brief outline of the remainder of the paper. Sections~\ref{Taubes' map} and~\ref{Construction of the near-symplectic Gromov invariants} recall the near-symplectic Gromov invariants and explain how they depend on spin-c structures on $X$. Section~\ref{$S^1$-valued Morse theory} introduces a small application of the main theorem in dimension three, concerning Morse theory, which may help the reader see what near-symplectic forms look like and what the near-symplectic Gromov invariants count.

Section~\ref{Review of gauge theory} consists of a review of those aspects of Seiberg--Witten theory that are relevant to this paper. This concerns the Seiberg--Witten equations and corresponding Floer homologies on 3- and 4-manifolds, either with symplectic/contact forms (following the work of Taubes) or without that additional structure (following the work of Kronheimer--Mrowka). Next, in Section~\ref{Review of Taubes' isomorphisms} we review Taubes' isomorphisms between embedded contact homology and monopole Floer homology. They will be mimicked and generalized in Section~\ref{Equating ECH and HM cobordism counts} to relate pseudoholomorphic curves with Seiberg--Witten solutions on the symplectic cobordism $(X-\nN,\omega)$, where $\nN$ is a certain tubular neighborhood of $\omega^{-1}(0)$. The story here is complicated by the existence of multiply covered tori and planes, but resolved using the methodology of Taubes' proof of ``$SW=Gr$'' for closed symplectic 4-manifolds.

Finally, in Section~\ref{Relation of Seiberg--Witten counts} we complete the proof of Theorem~\ref{thm:SWGr1} by relating Seiberg--Witten solutions on $X-\nN$ with Seiberg--Witten solutions on $X$, via a familiar ``stretching the neck'' procedure.

%%%%%%%%%%%%%%%%%%%%%%%%%%%%%%%%%%%%%%%%%%%%%%%%%%%%%%%%%%
\subsection{Taubes' map}
\label{Taubes' map}

\indent\indent
We are about to explain the construction of $Gr_{X,\omega}$, as it will nicely set the stage for the rest of the this paper. Let $\nN$ denote the union of arbitrarily small tubular neighborhoods of the components of $Z\subset X$, chosen in such a way that the complement
$$(X_0,\omega):=(X-\nN,\omega|_{X-\nN})$$
is a symplectic manifold with contact-type boundary, where each (negative) boundary component is a copy of $(S^1\times S^2,\xi_0)$. Here, $\xi_0$ is an overtwisted contact structure whose contact form $\lambda_0$ is specified in \cite{Gerig:taming}*{\S3.1}.

Now, $\omega$ determines the canonical bundle $K\to X_0$ and induces an $H_2(X;\ZZ)$-equivariant map
\begin{equation}
\label{eqn:map}
\tau_\omega:\Spinc(X)\to H_2(X_0,\partial X_0;\ZZ),\indent\s\mapsto\PD(c_1(E))
\end{equation}
where $E\to X_0$ is the complex line bundle that defines the decomposition of the positive spinor bundle associated with the restricted spin-c structure $\s|_{X_0}$,
$$\SS_+(\s|_{X_0})=E\oplus K^{-1}E$$
See the upcoming Section~\ref{Symplectic cobordisms} for an elaboration. The following lemma shows that $\tau_\omega$ gives a canonical identification between $\Spinc(X)$ and the set
$$\Big\{ A\in H_2(X_0,\partial X_0;\ZZ)\;\big\rvert\;\partial A=\1\in H_1(\partial X_0;\ZZ)\Big\}$$
where $\1$ is the oriented generator on each component (see \cite{Gerig:taming}*{\S3.1} for orientation conventions).

\begin{lemma}
\label{lem:map}
The map $\tau_\omega$ is injective, and its image consists of the subset of relative homology classes whose boundary is the oriented generator of $H_1(\partial X_0;\ZZ)$, i.e. for each $\s$ on $X$
$$\partial\tau_\omega(\s)=-(1,\ldots,1)\in-\bigoplus_{k=1}^NH_1(S^1\times S^2;\ZZ)$$
\end{lemma}

\begin{proof}
The restriction map $\Spinc(X)\to\Spinc(X_0)$ is injective, or in terms of the cohomology actions, the restriction map $H^2(X;\ZZ)\to H^2(X_0;\ZZ)$ is injective. This follows from the cohomological long exact sequence applied to the pair $(X,X_0)$ because
$$H^2(X,X_0;\ZZ)\cong H^2(\operatorname{cl}\nN,\partial\nN;\ZZ)\cong H_2(\operatorname{cl}\nN;\ZZ)= 0$$
using excision and Poincar\'e--Lefschetz duality. Then $\tau_\omega$ is injective, since
$$\tau_\omega(\s\otimes E)-\tau_\omega(\s)= \PD c_1(E|_{X_0})$$
for any complex line bundle $E\to X$.

Likewise, the determinant line bundle $\det\SS_+(\s|_{\partial X_0})$ is trivial. In terms of the cohomology actions, the restriction map $H^2(X;\ZZ)\to H^2(\partial X_0;\ZZ)$ is trivial because it factors through $H^2(\operatorname{cl}\nN;\ZZ)=0$. On each boundary component, this constraint
$$0=c_1\big(\det\SS_+(\s|_{S^1\times S^2})\big)=2c_1(E|_{S^1\times S^2})+c_1(K^{-1}|_{S^1\times S^2})$$
implies
$$c_1(E|_{S^1\times S^2})=1\in\ZZ\cong H^2(S^1\times S^2;\ZZ)$$
because $K^{-1}|_{S^1\times S^2}=\xi_0$ with Euler class $-2$ (see \cite{Gerig:taming}*{\S3.1}).
\end{proof}

%%%%%%%%%%%%%%%%%%%%%%%%%%%%%%%%%%%%%%%%%%%%%%%%%%%%%%%%%%
\subsection{Construction of the near-symplectic Gromov invariants}
\label{Construction of the near-symplectic Gromov invariants}

\indent\indent
We now fix a spin-c structure $\s\in\Spinc(X)$ and choose $\nN$ so that $-\partial X_0=\partial\nN$ is a contact 3-manifold whose contact form is $\lambda_\s$ on each component, an $\s$-dependent rescaling of $\lambda_0$ as provided in \cite{Gerig:taming}*{Lemma 3.9} such that $\ker\lambda_\s=\xi_0$. A key property of $\lambda_\s$, spelled out in \cite{Gerig:taming}*{\S3.2}, is that there is a positive real number\footnote{The definition of this real number is given by \cite{Gerig:taming}*{Equation 3-6} with the notation $\rho(\tau_\omega(\s))$ there.}
$$\rho_\s\in\RR$$
for which all of its Reeb orbits of symplectic action less than $\rho_\s$ are $\rho_\s$-flat (we also say that $\lambda_\s$ itself is $\rho_\s$-flat). The notion of ``$L$-flatness'' for a given positive real number $L$ is defined in \cite{Taubes:ECH=SWF1}*{\S2.d} and reviewed in \cite{Gerig:taming}*{\S2.5}, while its importance is revealed later in Section~\ref{Maps between ECH and HM} (see Theorem~\ref{thm:generators}, in particular).

Define the integer
\begin{equation}
\label{eqn:dim}
d(\s):=\frac{1}{4}\left(c_1(\s)^2-2\chi(X)-3\sigma(X)\right)
\end{equation}
where $c_1(\s)$ denotes the first Chern class of the spin-c structure's positive spinor bundle, $\chi(X)$ denotes the Euler characteristic of $X$, and $\sigma(X)$ denotes the signature of $X$. Introduce the set
$$\eE_\omega\subset H_2(X_0,\partial X_0;\ZZ)$$
of classes represented by symplectically embedded 2-spheres of self-intersection $-1$ in $X_0$ (which is empty if $X$ is minimal).

\bigskip
We now summarize the definition of $Gr_{X,\omega}(\s)$ given by \cite{Gerig:taming}*{Definition 1.8}. The component of the element $Gr_{X,\omega}(\s)$ in $\Lambda^pH^1(X;\ZZ)$ is defined to be zero whenever $d(\s)-p$ is odd or negative, or whenever $E\cdot\tau_\omega(\s)<-1$ for some $E\in\eE_\omega$ (this latter condition is independent of $p$). In the remaining cases it is determined by its evaluation on $[\eta_1]\wedge\cdots\wedge[\eta_p]$ for a given ordered set of classes $[\bar\eta]:=\big\{[\eta_i]\big\}_{i=1}^p\subset H_1(X;\ZZ)/\operatorname{Torsion}$. To specify this evaluation, fix the following data:
\bigskip

	$\bullet$ an ordered set of $p$ disjoint oriented loops $\bar\eta\subset X_0$ which represent $[\bar\eta]$,
	
	$\bullet$ a set of $\frac{1}{2}\big(d(\s)-p\big)$ disjoint points $\bar z:=\lbrace z_k\rbrace\subset X_0-\bar\eta$,
		
	$\bullet$ a cobordism-admissible almost complex structure $J$ on the completion $(\overline X_0,\omega)$ of $(X_0,\omega)$.

\bigskip\noindent
See \cite{Gerig:taming}*{\S2.2} for the terminology in the third bullet; the completion $\overline X_0$ is obtained by attaching certain cylindrical ends to the strong symplectic cobordism $(X_0,\omega):(\varnothing,0)\to(-\partial X_0,\lambda_\s)$. Given an admissible orbit set $\Theta$ on $(-\partial X_0,\lambda_\s)$ (in the sense of \cite{Gerig:taming}*{\S2.1}) such that its total homology class is $[\Theta]=-\partial\tau_\omega(\s)$, we form the moduli space
$$\mM_{d(\s)}(\varnothing,\Theta;\tau_\omega(\s),\bar z,\bar\eta)$$
of $J$-holomorphic currents in $\overline X_0$ that satisfy the following properties: they are asymptotic to $\Theta$ in the sense of \cite{Gerig:taming}*{\S2.2}; they have ECH index $d(\s)$ in the sense of \cite{Gerig:taming}*{\S2.3}; they represent $\tau_\omega(\s)$; they intersect every point and loop in $\bar z\cup\bar\eta$. Thanks to \cite{Gerig:taming}*{Proposition 3.16}, a component of a given current in this moduli space can only be multiply covered if it is a special curve, defined as follows.

\begin{definition}
A $J$-holomorphic curve in $\overline X_0$ is \textit{special} if it has Fredholm/ECH index zero, and is either an embedded torus or an embedded plane whose negative end is asymptotic to an embedded elliptic orbit with multiplicity one.
\end{definition}

It turns out that $\mM_{d(\s)}(\varnothing,\Theta;\tau_\omega(\s),\bar z,\bar\eta)$ is a finite set for generic $J$ (see \cite{Gerig:taming}*{Proposition 3.13}) and the symplectic action of $\Theta$ is bounded by $\rho_\s$ if this moduli space is nonempty (see \cite{Gerig:taming}*{\S3.2} or \cite{Hutchings:fieldtheory}). Since there can only be finitely many (nondegenerate) orbit sets with uniformly bounded action, there are only finitely many nonempty moduli spaces $\mM_{d(\s)}(\varnothing,\Theta;\tau_\omega(\s),\bar z,\bar\eta)$ indexed by $\Theta$. Moreover, any orbit set $\Theta$ for which the corresponding moduli space is nonempty has an absolute grading as generators of the ECH chain complex $ECC_*(-\partial X_0,\lambda_\s,1)$ (see \cite{Gerig:taming}*{\S2.4}) and they are in fact all of the same grading (see \cite{Gerig:taming}*{\S3.7}). Denote this grading by $g(\s)$.

As a further reminder, each orbit set $\Theta$ comes equipped with a choice of orientation so that the moduli spaces are all coherently oriented (see \cite{Gerig:taming}*{\S2.1, \S3.5}). The result of \cite{Gerig:taming}*{\S3.6}\footnote{This is an opportune moment to point out a negligible falsity in the proof of \cite{Gerig:taming}*{Proposition 3.23}; delete the paranthetical phrase ``(respectively cylinder)''.} is the following element in the ECH chain complex, whose homology class does not depend on the choice of points $\bar z$ nor representatives $\bar\eta$ of $[\bar\eta]$.

\begin{definition}
\label{def:GromovCycle}
With the above setup, the \textit{Gromov cycle} is
\begin{equation}
\Phi_{Gr}:=\sum_{\Theta}\mM_\Theta\Theta\in ECC_{g(\s)}(-\partial X_0,\lambda_\s,1)
\end{equation}
where the sum is over admissible orbit sets $\Theta$ in the grading $g(\s)$ with $[\Theta]=-\partial\tau_\omega(\s)$, and the integer coefficient $\mM_\Theta\in\ZZ$ is the sum over $\cC\in\mM_{d(\s)}(\varnothing,\Theta;\tau_\omega(\s),\bar z,\bar\eta)$ of a certain weight $q(\cC)\in\ZZ$ defined in \cite{Gerig:taming}*{\S3.5}. 
\end{definition}

Now, $ECH_*(-\partial X_0,\xi_0,1)$ is the tensor product $\bigotimes^N_{k=1}ECH_*(S^1\times S^2,\xi_0,1)$ and this homology is given by \cite{Gerig:taming}*{Proposition 3.2}. In terms of the absolute grading on $ECH_*(S^1\times S^2,\xi_0,1)$ by homotopy classes of oriented 2-plane fields on $S^1\times S^2$ (see \cite{Gerig:taming}*{\S2.4}), there is a unique grading $[\xi_*]$ such that
\begin{equation*}
    ECH_{[\xi_*]+n}(S^1\times S^2,\xi_0,1)\cong
    \begin{cases}
      0, & \text{if}\ n<0 \\
      \ZZ, & \text{otherwise}
    \end{cases}
\end{equation*}
In the proof of Theorem~\ref{thm:stretchNeck} it will be shown that $g(\s)=N[\xi_*]$. That said, $Gr_{X,\omega}(\s)\big([\eta_1]\wedge\cdots\wedge[\eta_p]\big)$ is by definition the coefficient of the class
$$[\Phi_{Gr}]\in ECH_{g(\s)}(-\partial X_0,\xi_0,1)\cong \bigotimes^N_{k=1} ECH_{[\xi_*]}(S^1\times S^2,\xi_0,1)\cong\ZZ$$
as a multiple of the positive generator $\1\in \bigotimes^N_{k=1} ECH_{[\xi_*]}(S^1\times S^2,\xi_0,1)$.

%%%%%%%%%%%%%%%%%%%%%%%%%%%%%%%%%%%%%%%%%%%%%%%%%%%%%%
\subsection{Exceptional spheres}
\label{Exceptional spheres}

\indent\indent
If $E\cdot \tau_\omega(\s)\ge-1$ for all $E\in\eE_\omega$ then the moduli spaces used to define $Gr_{X,\omega}(\s)$ do not contain any elements with multiply covered exceptional sphere components (see \cite{Gerig:taming}*{Proposition 3.16}). However, if $E\cdot \tau_\omega(\s)<-1$ for some $E\in\eE_\omega$ then we cannot rule out their existence, and we define the near-symplectic Gromov invariant to be zero, as is done with Taubes' Gromov invariants of non-minimal symplectic 4-manifolds. What follows is a more general version of the main Theorem~\ref{thm:SWGr1} which does not assume $X$ to be minimal; its proof will subsume the proof of Theorem~\ref{thm:SWGr1}.

\begin{theorem}
\label{thm:SWGr2}
Given $(X,\omega,J)$ and $\s\in\Spinc(X)$ satisfying $E\cdot\tau_\omega(\s)\ge-1$ for all $E\in\eE_\omega$,
$$Gr_{X,\omega}(\s)\equiv SW_X(\s){\mod2}\indent\in\Lambda^* H^1(X;\ZZ)\otimes\ZZ/2\ZZ$$
where $\omega$ determines the chamber for defining the Seiberg--Witten invariants when $b^2_+(X)=1$.
\end{theorem}

%%%%%%%%%%%%%%%%%%%%%%%%%%%%%%%%%%%%%%%%%%%%%%%%%%%%%%
\subsection{$S^1$-valued Morse theory}
\label{$S^1$-valued Morse theory}

\indent\indent
A basic example of a near-symplectic manifold is $(S^1\times M,\omega_f)$, where $M$ is a closed oriented Riemannian 3-manifold with $b^1(M)>0$, the metric on $S^1\times M$ is the product metric $dt^2+g_M$, and $\omega_f$ is defined momentarily. A result of Honda \cite{Honda:Morse} and of Calabi \cite{Calabi:intrinsic} says that for $g_M$ suitably generic, any nonzero class in $H^1(M;\ZZ)$ is represented by a harmonic map $f:M\to S^1$ (i.e. $d^* df=0$) with nondegenerate critical points $\crit(f)$ of index 1 or 2, hence a harmonic 1-form $df$ with transversal zeros. Then
$$\omega_f:=dt\wedge df+*_3df$$
is a closed self-dual 2-form which vanishes transversally on
$$Z_f:=S^1\times\crit(f)$$
All zero-circles are untwisted, as can be seen by writing out $\omega_f$ in local coordinates and comparing to the standard model on $\RR\times\RR^3$. There are an even number of zero-circles, i.e. $1-b^1(X)+b^2_+(X)$ is even, because $b^1(X)=b^1(M)+1$ and $b^2(X)=2b^1(M)$ and $b^2_+(X)=b^1(M)$. Here, we note that
$$H^1(M;\RR) \to H^2_+(X;\RR)\,,\indent a \mapsto [dt \wedge a]^+=\frac{1}{2}(dt\wedge a+*_3a)$$
is an isomorphism.

After equipping $(S^1\times M)-Z_f=S^1\times\big(M-\crit(f)\big)$ with the compatible almost complex structure $J$ determined by $\omega_f$ and $dt^2+g_M$, the $S^1$-invariant connected $J$-holomorphic submanifolds are of the form $C=S^1\times \gamma$, where $\gamma$ is a single gradient flowline of $\nabla f$. Since the Morse trajectories in $M$ are either periodic orbits (with periodicity) or paths between critical points, $C$ is either a torus (with multiplicity) or a cylinder which bounds two zero-circles in $S^1\times M$ (see Figure~\ref{fig:3dim}).

\begin{figure}
    \centering
    \includegraphics[width=5cm]{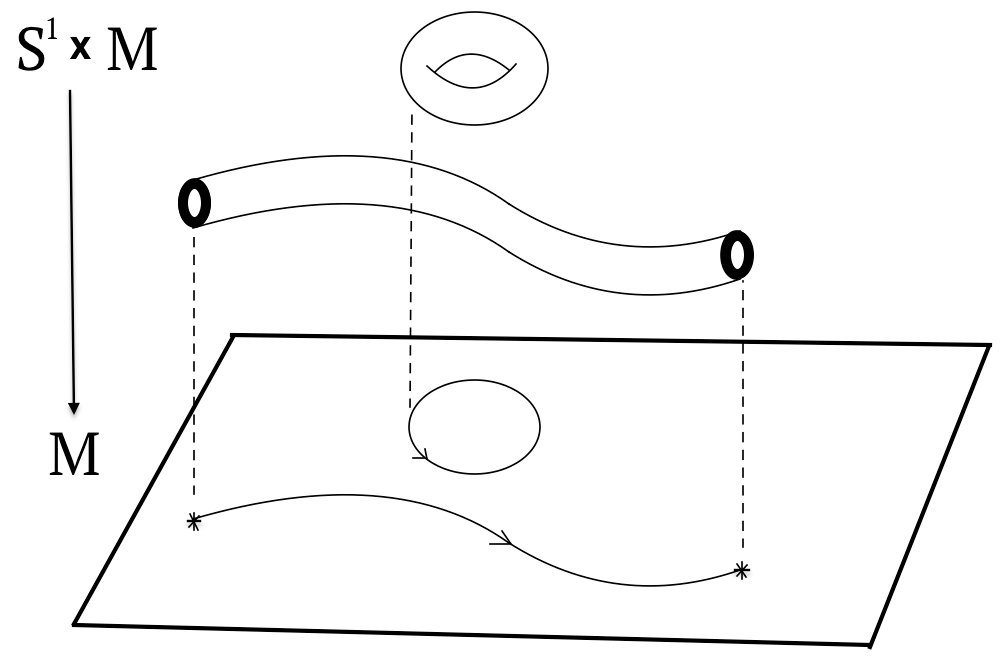}
    \caption{Zero set of near-symplectic form in bold}
    \label{fig:3dim}
\end{figure}

In their PhD theses, Hutchings and Lee built a 3-dimensional invariant $I_{M,f}$ of $M$ which suitably counts the gradient flowlines (see \cite{HutchingsLee1,HutchingsLee2,Hutchings:thesis}), and they further showed that it equals a version of topological (Reidemeister) torsion defined by Turaev \cite{Turaev:torsion}. It was subsequently shown by Turaev that this Reidemeister torsion equals the 3-dimensional SW invariant $SW_M$ of $M$ (see \cite{Turaev:SW}). Strictly speaking, there is a required choice of ``chamber'' (determined by $f$) with which to define $SW_M$ when $b^1(M)=1$, and an ordering of the set $\crit(f)$ with which to define $I_{M,f}$. While $SW_M$ is a function of the set $\Spinc(M)$ of spin-c structures, $I_{M,f}$ is a function of the set
$$\big\{ \eta\in H_1(M,\crit(f);\ZZ)\;|\;\partial\eta=[\crit(f)]\big\}$$
and there exists an $H_1(M)$-equivariant isomorphism $\tau_f$ between them (see \cite{HutchingsLee1}*{Lemma 4.3}). In summary,

\begin{theorem}[Hutchings--Lee--Turaev]
\label{thm:HLT}
Let $(M,f)$ be as above. Then for all $\s\in\Spinc(M)$,
$$I_{M,f}\big(\tau_f(\s)\big)=\pm SW_M(\s)\in\ZZ$$
where $f$ determines the chamber for defining $SW_M$ when $b^1(M)=1$. The global $\pm$ sign is pinned down by a suitable choice of ordering of $\crit(f)$.
\end{theorem}

Now, we can recover Theorem~\ref{thm:HLT} (over $\ZZ/2\ZZ$) without having to pass through Reidemeister torsion, via a dimensional reduction of Theorem~\ref{thm:SWGr1}.\footnote{This was in fact an expectation and also the motivation, as explained in \cite{HutchingsLee1}*{\S4.2.1}.} It was shown in \cite{OkonekTeleman:3SW}*{Theorem 3.5} that the 4-dimensional SW invariant\footnote{There is a canonical homology orientation on 4-manifolds of the form $S^1\times M$.} recovers the 3-dimensional SW invariant: all solutions to the SW equations on $S^1\times M$ associated with product spin-c structures are $S^1$-invariant. It will be shown elsewhere \cite{Gerig:Morse} that the near-symplectic Gromov invariant recovers Hutchings--Lee's flowline invariant.

\begin{cor}[\cite{Gerig:Morse}]
\label{cor:Morse}
Let $(M,f)$ be as above, and let $\pi:S^1\times M\to M$ be the projection map onto the second factor. Fix an ordering of the critical points of $f$ (hence of the zero-circles of $\omega_f)$. Then
$$I_{M,f}(\s)=Gr_{S^1\times M,\omega_f}(\pi^*\s)\equiv_{(2)} SW_{S^1\times M}(\pi^*\s)=SW_M(\s)$$
for all $\s\in\Spinc(M)$. When $b^1(M)=b^2_+(S^1\times M)=1$, the chamber is determined by $f$ (hence $\omega_f$).
\end{cor}

\begin{remark}
In fact, when $f$ has no critical points (hence $\omega_f$ is symplectic) the proof of this corollary was already known \cite{HutchingsLee2}*{Remark 1.10}. The first instance appeared in \cite{Salamon:mappingTori} for the special case of a mapping torus of a symplectomorphism of a Riemann surface, in which Salamon showed that the 3-dimensional SW invariants recover the Lefschetz invariants of the symplectomorphism.
\end{remark}

We end this discussion with a brief sketch of the first equality in Corollary~\ref{cor:Morse}. While the gradient flowlines for $I_{M,f}$ sit inside $M$, the $J$-holomorphic curves for $Gr_{S^1\times M,\omega_f}$ do not sit inside $S^1\times M$ but rather inside the completion of the complement of $Z_f$. Explicitly, we choose a 3-ball neighborhood $\bigsqcup_kB^3$ of the critical points of $f$, hence a tubular neighborhood $\nN=S^1\times\bigsqcup_k B^3$ of the zero-circles in $Z_f$. Then
$$X_0:=(S^1\times M)-\nN=S^1\times(M-\bigsqcup_k B^3)$$
is a symplectic manifold and $H_2(X_0,\partial X_0;\ZZ)$ is isomorphic to a direct sum of $|\crit(f)|$ copies of
$$H_2(S^1\times(M-B^3),S^1\times S^2;\ZZ)\cong  H_2(M-B^3,S^2;\ZZ)\oplus H_1(M-B^3,S^2;\ZZ)$$
using the relative K\"unneth formula. Given the relative 1st homology class $\tau_f(\s)\in H_1(M,\crit(f);\ZZ)$, the corresponding relative 2nd homology class is
$$[S^1]\times\tau_f(\s)=\tau_{\omega_f}(\pi^*\s)\in H_2(X_0,\partial X_0;\ZZ)$$

Since the contact form $\lambda_0$ of Section~\ref{Taubes' map} is actually $S^1$-invariant, it follows from \cite{Gerig:taming} that we can choose the 3-balls in such a way that $-\partial X_0$ is a contact boundary of $X$ with $\omega_f$ equal to a scalar multiple of $d\lambda_0$ on each component. Although the playground $\big(X_0,\omega_f,J,\tau_{\omega_f}(\pi^*\s)\big)$ is $S^1$-invariant, the calculation of $Gr_{S^1\times M,\omega_f}\big(\tau_{\omega_f}(\pi^*\s)\big)$ requires us to modify $\lambda_0$ (and $\partial X_0$) into a non-$S^1$-invariant nondegenerate contact form $\lambda_\s$, so $(-\partial X_0,\lambda_\s)$ does not arise from any choice of the 3-balls. This is a complication because we would like to lift flowlines $\gamma\subset M$ to $S^1$-invariant curves $S^1\times\gamma\subset X_0$ and count them. Nonetheless, a limiting argument will show that we can perturb the $S^1$-invariant setup to the non-$S^1$-invariant setup and relate the corresponding pseudoholomorphic curves.

%%%%%%%%%%%%%%%%%%%%%%%%%%%%%%%%%%%%%%%%%%%%%%%%%%%%%%%%%%
%%%%%%%%%%%%%%%%%%%%%%%%%%%%%%%%%%%%%%%%%%%%%%%%%%%%%%%%%%
\section{Review of gauge theory}
\label{Review of gauge theory}

\indent\indent
The point of this section is to introduce most of the terminology and notations that appear in the later sections. Further information and more complete details are found in \cite{KM:book, HutchingsTaubes:Arnold2}.

%%%%%%%%%%%%%%%%%%%%%%%%%%%%%%%%%%%%%%%%%%%%%%%%%%%%%%%%%%
\subsection{Closed 3-manifolds}
\label{Closed 3-manifolds}

\indent\indent
Let $(Y,\lambda)$ be a closed oriented connected contact 3-manifold, and choose an almost complex structure $J$ on $\xi$ that induces a symplectization-admissible almost complex structure on $\RR\times Y$. There is a compatible metric $g$ on $Y$ such that\footnote{The factors of $\frac{1}{2}$ can be dropped or changed to any other nonzero real number by a particular rescaling of the metric, but they will be left in to be consistent with the papers of Taubes and Hutchings.} $|\lambda|=1$ and $*\lambda=\frac{1}{2}d\lambda$, with $g(v,w)=\frac{1}{2}d\lambda(v,Jw)$ for $v,w\in\xi$.

View a spin-c structure $\s$ on $Y$ as an isomorphism class of a pair $(\SS,\cl)$ consisting of a rank 2 Hermitian vector bundle $\SS\to Y$ and Clifford multiplication $\cl:TY\to\End(\SS)$. We refer to $\SS$ as the \textit{spinor bundle} and its sections as \textit{spinors}. The set $\Spinc(Y)$ of spin-c structures is an affine space over $H^2(Y;\ZZ)$, defined by
$$(\SS,\cl)+x=(\SS\otimes E^x,\cl\otimes\1)$$
where $E^x\to Y$ is the complex line bundle satisfying $c_1(E^x)=x\in H^2(Y;\ZZ)$. Denote by $c_1(\s)$ the first Chern class of $\det\SS$; it satisfies $c_1(\s+x)=c_1(\s)+2x$.

The contact structure $\xi$ (and more generally, any oriented 2-plane field on $Y$) picks out a canonical spin-c structure $\s_\xi=(\SS_\xi,\cl)$ with $\SS_\xi=\underline{\CC}\oplus\xi$, where $\underline{\CC}\to Y$ denotes the trivial line bundle, and Clifford multiplication is defined as follows. Given an oriented orthonormal frame $\lbrace e_1,e_2,e_3\rbrace$ for $T_yY$ such that $\lbrace e_2,e_3\rbrace$ is an oriented orthonormal frame for $\xi_y$, then in terms of the basis $(1,e_2)$ for $\SS_\xi$,
$$\cl(e_1)=\bigl( \begin{smallmatrix}
i&0\\ 0&-i
\end{smallmatrix} \bigr),\indent \cl(e_2)=\bigl( \begin{smallmatrix}
0&-1\\ 1&0
\end{smallmatrix} \bigr),\indent \cl(e_3)=\bigl( \begin{smallmatrix}
0&i\\ i&0
\end{smallmatrix} \bigr)$$
There is then a canonical isomorphism
$$H^2(Y;\ZZ)\to\Spinc(Y),\indent x\mapsto\left(E^x\oplus(\xi\otimes E^x),\cl\right)$$
where the 0 class corresponds to $\s_\xi$. In other words, there is a canonical decomposition $\SS=E\oplus \xi E$ into $\pm i$ eigenbundles of $\cl(\lambda)$. Here and in what follows, the tensor product notation is suppressed.

\bigskip
A \textit{spin-c connection} is a connection $\bA$ on $\SS$ which is compatible with Clifford multiplication in the sense that
$$\nabla_\bA(\cl(v)\psi)=\cl(\nabla v)\psi+\cl(v)\nabla_\bA\psi$$
where $\nabla v$ denotes the covariant derivative of $v\in TY$ with respect to the Levi-Civita connection. Such a connection is equivalent to a Hermitian connection (also denoted by $\bA$) on $\det\SS$, and determines a \textit{Dirac operator}
$$D_\bA:\Gamma(\SS)\stackrel{\nabla_\bA}{\longrightarrow}\Gamma(T^* Y\otimes\SS)\stackrel{\cl}{\longrightarrow}\Gamma(\SS)$$
With respect to the decomposition $\SS=E\oplus\xi E$, the determinant line bundle is $\det\SS=\xi E^2$ and any spinor can be written as
$$\psi=(\alpha,\beta)$$
There is a unique connection $A_\xi$ on $\xi$ such that its Dirac operator kills the spinor $(1,0)\in\Gamma(\SS_\xi)$, and on $\det\SS$ there is a canonical decomposition
$$\bA=A_\xi+2A$$
with Hermitian connection $A$ on $E$. We henceforth refer to a spin-c connection as a Hermitian connection on $E$, and denote its Dirac operator by $D_A$.

The gauge group $C^\infty(Y,S^1)$ acts on a given pair $(A,\psi)$ by
$$u\cdot(A,\psi)=(A-u^{-1}du,u\psi)$$
In this paper, a \textit{configuration} $\c$ refers to a gauge-equivalence class of such a pair, and the set of configurations is denoted by
$$\bB(Y,\s):=(\Conn(E)\times\Gamma(\SS))/C^\infty(Y,S^1)$$

\bigskip
Fix a suitably generic exact 2-form $\mu\in\Omega^2(Y)$ as described in \cite{HutchingsTaubes:Arnold2}*{\S 2.2}, and a positive real number $r\in\RR$. A configuration $\c$ solves \textit{Taubes' perturbed Seiberg--Witten equations} when
\begin{equation}
\label{SW3}
D_\bA\psi=0,\indent\indent* F_A=r(\tau(\psi)-i\lambda)-\frac{1}{2}* F_{A_\xi}+i*\mu
\end{equation}
where $F_{A_\xi}$ is the curvature of $A_\xi$ and $\tau:\SS\to iT^* Y$ is the quadratic bundle map
$$\tau(\psi)(\cdot)=\langle\cl(\cdot)\psi,\psi\rangle$$
An appropriate change of variables recovers the usual Seiberg--Witten equations (with perturbations) that appear in \cite{KM:book}. 

\begin{remark}
There are additional ``abstract tame perturbations'' to these equations required to obtain transversality of the moduli spaces of its solutions (see \cite{KM:book}*{\S 10}), but they have been suppressed because they do not interfere with the analysis presented in this paper. Further clarification on this matter can be found in \cite{HutchingsTaubes:Arnold2}*{\S 2.1} and \cite{Taubes:ECH=SWF1}*{\S3.h.5}, where the same suppression occurs.
\end{remark}

Denote by $\M(Y,\s)$ the set of solutions to~\eqref{SW3}, called \textit{(SW) monopoles}. A solution is \textit{reducible} if its spinor component vanishes, and is otherwise \textit{irreducible}. After attaching orientations (this being a $\ZZ/2\ZZ$ choice for each monopole, see Section~\ref{Homology orientations}), the monopoles freely generate the monopole Floer chain complex $\Cfrom^*(Y,\lambda,\s,J,r)$. The chain complex differential will not be reviewed here. Of importance to this paper are irreducible monopoles with certain bounds on their energy
$$E(\c):=i\int_Y\lambda\wedge F_A$$
Denote by $\Cfrom^*_L(Y,\lambda,\s,J,r)$ the submodule generated by irreducible monopoles $\c$ with energy $E(\c)<2\pi L$. When $r$ is sufficiently large, $\Cfrom^*_L(Y,\lambda,\s,J,r)$ is a subcomplex of $\Cfrom^*(Y,\lambda,\s,J,r)$ and the homology $\Hfrom^*_L(Y,\lambda,\s,J,r)$ is well-defined and independent of $r$ and $\mu$ (see \cite{HutchingsTaubes:Arnold2}*{\S 2.3}). Taking the direct limit over $L>0$, we recover the ordinary $\Hfrom^*(Y,\s)$ in \cite{KM:book} which is independent of $\lambda$ and $J$. It is sometimes convenient to consider the group
$$\Hfrom^*(Y):=\bigoplus_{\s\in\Spinc(Y)}\Hfrom^*(Y,\s)$$
over all spin-c structures at once.

%%%%%%%%%%%%%%%%%%%%%%%%%%%%%%%%%%%%%%%%%%%%%%%%%%%%%%%%%%
\subsection{Symplectic cobordisms}
\label{Symplectic cobordisms}

\indent\indent
Let $(X,\omega)$ be a strong symplectic cobordism between (possibly disconnected or empty) closed oriented contact 3-manifolds $(Y_\pm,\lambda_\pm)$. Due to the choice of metric $g_\pm$ on $Y_\pm$ in Section~\ref{Closed 3-manifolds} (and following \cite{HutchingsTaubes:Arnold2}*{\S4.2}), we do not extend $\omega$ over $\overline X$ using $d(e^s\lambda_\pm)$ on the ends $(-\infty,0]\times Y_-$ and $[0,\infty)\times Y_+$. Instead, we extend $\omega$ using $d(e^{2s}\lambda_\pm)$ as follows. Fix a smooth increasing function $\phi_-:(-\infty,\e]\to(-\infty,\e]$ with $\phi_-(s)=2s$ for $s\le\frac{\e}{10}$ and $\phi_-(s)=s$ for $s>\frac{\e}{2}$, and fix a smooth increasing function $\phi_+:[-\e,\infty)\to[-\e,\infty)$ with $\phi_+(s)=2s$ for $s\ge-\frac{\e}{10}$ and $\phi_+(s)=s$ for $s\le-\frac{\e}{2}$, where $\e>0$ is such that $\omega=d(e^s\lambda_\pm)$ on the $\e$-collars of $Y_\pm$. Then the desired extension is
\begin{equation*}
\tilde\omega:=\begin{cases}
    d(e^{\phi_-}\lambda_-) & \text{on}\;\; (-\infty,\e]\times Y_-\\
    \omega & \text{on}\;\; X\setminus\Big(\big([0,\e]\times Y_-\big)\cup\big([-\e,0]\times Y_+\big)\Big)\\
    d(e^{\phi_+}\lambda_+) & \text{on}\;\; [-\e,\infty)\times Y_+
  \end{cases}
\end{equation*}
Now choose a cobordism-admissible almost complex structure $J$ on $(\overline X,\tilde\omega)$. Following \cite{HutchingsTaubes:Arnold2}*{\S4.2}, we equip $\overline X$ with a particular metric $g$ so that it agrees with the product metric with $g_\pm$ on the ends $(-\infty,0]\times Y_-$ and $[0,\infty)\times Y_+$ and so that $\tilde\omega$ is self-dual. Finally, define
$$\widehat\omega:=\sqrt{2}\tilde\omega/|\tilde\omega|_g$$
and note that $J$ is still cobordism-admissible.

The 4-dimensional gauge-theoretic scenario is analogous to the 3-dimensional scenario. View a spin-c structure $\s$ on $X$ as an isomorphism class of a pair $(\SS,\cl)$ consisting of a Hermitian vector bundle $\SS=\SS_+\oplus\SS_-$, where $\SS_\pm$ have rank 2, and Clifford multiplication $\cl:TX\to\End(\SS)$ such that $\cl(v)$ exchanges $\SS_+$ and $\SS_-$ for each $v\in TX$. We refer to $\SS_+$ as the \textit{positive spinor bundle} and its sections as \textit{(positive) spinors}. The set $\Spinc(X)$ of spin-c structures is an affine space over $H^2(X;\ZZ)$, and we denote by $c_1(\s)$ the first Chern class of $\det\SS_+=\det\SS_-$. A spin-c connection on $\SS$ is equivalent to a Hermitian connection $\bA$ on $\det\SS_+$ and defines a Dirac operator $D_\bA:\Gamma(\SS_\pm)\to\Gamma(\SS_\mp)$.

A spin-c structure $\s$ on $X$ restricts to a spin-c structure $\s|_{Y_\pm}$ on $Y_\pm$ with spinor bundle $\SS_{Y_\pm}:=\SS_+|_{Y_\pm}$ and Clifford multiplication $\cl_{Y_\pm}(\cdot):=\cl(v)^{-1}\cl(\cdot)$, where $v$ denotes the outward-pointing unit normal vector to $Y_+$ and the inward-pointing unit normal vector to $Y_-$. There is a canonical way to extend $\s$ over $\overline X$, and the resulting spin-c structure is also denoted by $\s$. There is a canonical decomposition $\SS_+=E\oplus K^{-1}E$ into $\mp2i$ eigenbundles of $\cl_+(\widehat\omega)$, where $K$ is the canonical bundle of $(\overline X,J)$ and $\cl_+:\bigwedge^2_+T^*\overline X\to\End(\SS_+)$ is the projection of Clifford multiplication onto $\End(\SS_+)$. This agrees with the decomposition of $\SS_{Y_\pm}$ on the ends of $\overline X$.

The symplectic form $\omega$ picks out the canonical spin-c structure $\s_\omega=(\SS_\omega,\cl)$, namely that for which $E$ is trivial, and the $H^2(X;\ZZ)$-action on $\Spinc(X)$ becomes a canonical isomorphism. There is a unique connection $A_{K^{-1}}$ on $K^{-1}$ such that its Dirac operator annihilates the spinor $(1,0)\in\Gamma((\SS_\omega)_+)$, and we henceforth identify a spin-c connection with a Hermitian connection $A$ on $E$ and denote its Dirac operator $D_A$.

\bigskip
In this paper, a \textit{configuration} $\d$ refers to a gauge-equivalence class of a pair $(\bA,\Psi)$ under the gauge group $C^\infty(X,S^1)$-action. A connection $\bA$ on $\det\SS_+$ is in \textit{temporal gauge} on the ends of $\overline X$ if
$$\nabla_\bA=\frac{\partial}{\partial s}+\nabla_{\bA(s)}$$
on $(-\infty,0]\times Y_-$ and $Y_+\times[0,\infty)$, where $\bA(s)$ is a connection on $\det\SS_{Y_\pm}$ depending on $s$. Any connection can be placed into temporal gauge by an appropriate gauge transformation. Given monopoles $\c_\pm$ on $Y_\pm$, the set of configurations which are asymptotic to $\c_\pm$ (in temporal gauge on the ends of $\overline X$) is denoted by
$$\bB(\c_-,\c_+;\s)\subset(\Conn(E)\times\Gamma(\SS_+))/C^\infty(X,S^1)$$

\bigskip
Fix suitably generic exact 2-forms $\mu_\pm\in\Omega^2(Y_\pm)$, a suitably generic exact 2-form $\mu\in\Omega^2(\overline X)$ that agrees with $\mu_\pm$ on the ends of $\overline X$ (with $\mu_*$ denoting its self-dual part), and a positive real number $r\in\RR$. \textit{Taubes' perturbed Seiberg--Witten equations} for a configuration $\d$ are
\begin{equation}
\label{SW4}
D_\bA\Psi=0,\;\;F^+_A=\frac{r}{2}(\rho(\Psi)-i\widehat{\omega})-\frac{1}{2}F^+_{A_{K^{-1}}}+i\mu_*
\end{equation}
where $F_A^+$ is the self-dual part of the curvature of $A$ and $\rho:\SS_+\to i\bigwedge^2_+T^* X$ is the quadratic bundle map
$$\rho(\Psi)(\cdot,\cdot)=-\frac{1}{2}\big\langle[\cl(\cdot),\cl(\cdot)]\Psi,\Psi\big\rangle$$
Similarly to the 3-dimensional equations, there are additional ``abstract tame perturbations'' which have been suppressed in this paper (see \cite{KM:book}*{\S24.1}). Denote by $\M(\c_-,\c_+;\s)$ the set of solutions to~\eqref{SW4} in $\bB(\c_-,\c_+;\s)$, called \textit{(SW) instantons}.

Similarly to ECH, an ``index'' is associated with each SW instanton, namely the local expected dimension of the moduli space of SW instantons. Denote by $\M_k(\c_-,\c_+;\s)$ the subset of elements in $\M(\c_-,\c_+;\s)$ that have index $k$.

%%%%%%%%%%%%%%%%%%%%%%%%%%%%%%%%%%%%%%%%%%%%%%%%%%%%%%%%%%
\subsection{Closed 4-manifolds}
\label{Closed 4-manifolds}

\indent\indent
The case $(Y_\pm,\lambda_\pm)=(\varnothing,0)$ recovers Seiberg--Witten theory on closed oriented symplectic 4-manifolds. In general, for closed oriented Riemannian 4-manifolds $(X,g)$, we can recover Seiberg--Witten theory from the above setup by ignoring the appearance of $\omega$ and thus ignoring the canonical decomposition of $\SS_+$. The set of spin-c structures is then only an $H^2(X;\ZZ)$-torsor. A configuration $\d=[\bA,\Psi]\in\bB(X,\s)$ solves \textit{the (perturbed) Seiberg--Witten equations} when
\begin{equation}
\label{SWclosed}
D_\bA\Psi=0,\;\;F^+_\bA=\frac{1}{4}\rho(\Psi)+i\mu
\end{equation}
where $\mu\in\Omega^2_+(X)$ is now a self-dual 2-form. Denote the space of solutions to~\eqref{SWclosed} by $\M(\s)$. 

When $b^2_+(X)>0$, a generic choice of $\mu$ makes $\M(\s)$ a finite-dimensional compact orientable smooth manifold, where the orientation is determined by a \textit{homology orientation} of $X$ (see also Section~\ref{Homology orientations}), this being an orientation of
$$\operatorname{det}^+(X):=\det H^1(X;\RR)\otimes\det H^2_+(X;\RR)$$
As explained in \cite{Taubes:Gr=SW}*{\S 1.c}, if $X$ is equipped with a symplectic form then there is a canonical homology orientation.

The dimension of $\M(\s)$ is equal to the integer $d(\s)$ given by~\eqref{eqn:dim}, and its parity is equal to the parity of
$$1-b^1(X)+b^2_+(X)$$
If $\dim\M(\s)<0$, then $\M(\s)$ is empty and the \textit{Seiberg--Witten invariant} $SW_X(\s)$ is defined to be zero. In the remaining cases, the \textit{Seiberg--Witten invariant} $SW_X(\s)$ is an element of $\Lambda^* H^1(X;\ZZ)$ and given by suitable counts of points in $\M(\s)$ (see the upcoming Definition~\ref{defn:SW}). For example, if $\dim\M(\s)=0$ then $SW_X(\s)\in\ZZ$ is the signed count of the finite number of oriented points in $\M(\s)$.

%%%%%%%%%%%%%%%%%%%%
\subsubsection{Choice of ``chamber''}

\indent\indent
When $b^2_+(X)>1$, the value of the Seiberg--Witten invariant is a diffeomorphism invariant of $X$ independent of the choice of generic pairs $(g,\mu)\in\operatorname{Met}(X)\times\Omega^2_+(X)$, where $\operatorname{Met}(X)$ denotes the Fr\'echet space of smooth Riemannian metrics on $X$. When $b^2_+(X)=1$, there is a ``wall-crossing phenomenon'' as follows. Denote by $\omega_g$ the unique (up to scalar multiplication) nontrivial near-symplectic form with respect to $g$. The set of pairs $(g,\mu)$ satisfying the constraint
\begin{equation}
\label{chamber}
2\pi[\omega_g]\cdot c_1(\s)+\int_X\omega_g\wedge\mu=0
\end{equation}
defines a ``wall'' which separates $\operatorname{Met}(X)\times\Omega^2_+(X)$ into two open sets, called \textit{$c_1(\s)$-chambers}. The Seiberg--Witten invariant is constant on any $c_1(\s)$-chamber, and the difference between chambers is computable.

The near-symplectic form $\omega$ on $X$ picks out a canonical $c_1(\s)$-chamber, namely those pairs $(g,\mu)$ for which the left hand side of~\eqref{chamber} is negative. This is the chamber that pertains to the large $r$ version of Taubes' perturbed Seiberg--Witten equations~\eqref{SW4} over the completion of $(X_0,\omega)$ and subsequently used to state our main Theorem~\ref{thm:SWGr1}.

%%%%%%%%%%%%%%%%%%%%%%%%%%%%%%%%%%%%%%%%%%%%%%%%%%%%%%%%%%
\subsection{Kronheimer--Mrowka's formalism}
\label{Kronheimer--Mrowka's formalism}

\indent\indent
The previous sections concerned the setup of Seiberg--Witten theory from the point of view of symplectic geometry, using Taubes' large perturbations. We now briefly review some relevant aspects of Seiberg--Witten theory from the point of view of Kronheimer--Mrowka's monopole Floer homology, following closely the notations from \cite{KM:book} and deferring to \cite{KM:book} for their precise definitions.

Let $\bB(Y,\s)$ denote the space of configurations $[\bA,\psi]$. Since we are not taking large perturbations to the Seiberg--Witten equations, we have to deal with the reducible locus $\bB^\text{red}(Y,\s)$ which prevents $\bB(Y,\s)$ from being a Banach manifold. This is done by forming the \textit{blow-up} $\bB^\sigma(Y,\s)$, the space of configurations $[\bA,s,\psi]$ such that $s\in\RR^{\ge0}$ and $\|\psi\|_2=1$, equipped with the blow-down map
$$\bB^\sigma(Y,\s)\to\bB(Y,\s),\indent[\bA,s,\psi]\mapsto[\bA,s\psi]$$
This is a Banach manifold whose boundary $\partial\bB^\sigma(Y,\s)$ consists of reducible configurations (where $s=0$). The same setup applies to the case that $X$ is a closed 4-manifold. The integral cohomology ring $H^*(\bB^\sigma(M,\s);\ZZ)$, for $M$ either $Y$ or $X$, is isomorphic to the graded algebra
$$\AA(M):=\big(\Lambda^* H_1(M;\ZZ)/\text{Torsion}\big)\otimes\ZZ[U]$$
where $U$ is a 2-dimensional generator (see \cite{KM:book}*{Proposition 9.7.1}).

We can construct a certain vector field $\vV^\sigma$ on $\bB^\sigma(Y,\s)$ using the pull-back of the gradient of the Chern--Simons--Dirac functional $\lL_\text{CSD}:\bB(Y,\s)\to\RR$ (see \cite{KM:book}*{\S4.1}). Strictly speaking, the Chern--Simons--Dirac functional is not well-defined on $\bB(Y,\s)$ unless $c_1(\s)$ is torsion, but such spin-c structures are the only ones relevant to this paper. Likewise, the perturbed gradient $\grad\lL_\text{CSD}+\q$ gives rise to a vector field $\vV^\sigma+\q^\sigma$, where $\q$ is an ``abstract tame perturbation'' (see \cite{KM:book}*{\S 10}). We always assume that $\q$ is chosen from a residual subset of the Banach space of `large' tame perturbations introduced in \cite{KM:book}*{Theorem 11.6.1, Definition 11.6.3} so that all stationary points of $\vV^\sigma+\q^\sigma$ are nondegenerate (by \cite{KM:book}*{Theorem 12.1.2}).

The critical points (i.e. stationary points) of $\vV^\sigma+\q^\sigma$ are either \textit{irreducibles} of the form $[\bA,s,\psi]$ with $s>0$ and $[\bA,s\psi]\in\crit(\grad\lL_\text{CSD}+\q)$, or \textit{reducibles} of the form $[\bA,0,\psi]$ with $\psi$ an eigenvector of $D_\bA$. A reducible is \textit{boundary-stable} (respectively, \textit{boundary-unstable}) if the corresponding eigenvalue is positive (respectively, negative). Denote by
$$\C(Y,\s)=\C^o(Y,\s)\sqcup\C^u(Y,\s)\sqcup\C^s(Y,\s)$$
the decomposition of the set of critical points into the respective sets of irreducibles and boundary-(un)stable reducibles. We can package these critical points together in various ways to form the monople Floer (co)homologies, such as $\Hto^*(Y,\s)$ and $\Hfrom^*(Y,\s)$ -- the former cochain complex is generated by $\C^o(Y,\s)\sqcup\C^s(Y,\s)$ while the latter complex is generated by $\C^o(Y,\s)\sqcup\C^u(Y,\s)$, both equipped with coherent choices of orientations (see Section~\ref{Homology orientations}). The differentials will not be reviewed here, but we do assume in this paper that all perturbations $\q$ are chosen so that the differentials are well-defined.

\begin{remark}
If $\q$ is one of Taubes' sufficiently large perturbations associated with a contact form (given in Section~\ref{Closed 3-manifolds}), then the image of $\C(Y,\s)$ under the blow-down map is $\M(Y,\s)$. In fact, we no longer need to use the blow-up model thanks to the direct limit description in Section~\ref{Closed 3-manifolds}.
\end{remark}

Let $X$ either be a closed 4-manifold or have boundary $Y$. There is a partially-defined restriction map $r:\bB^\sigma(X,\s)\dashrightarrow\bB^\sigma(Y,\s)$ whose domain consists of those configurations $[\bA,s,\Psi]$ satisfying $\Psi_Y:=\Psi|_Y\ne0$, such that
$$r([\bA,s,\Psi])=\Big[\bA|_Y,s\|\Psi_Y\|_2,\Psi_Y/\|\Psi_Y\|_2\Big]$$
Similarly, if $X=[0,1]\times Y$ then there is a family of restriction maps $r_t:\bB^\sigma(X,\s)\dashrightarrow\bB^\sigma(Y,\s)$ for $t\in[0,1]$. If we instead work over $\RR\times Y$ or cylindrical ends such as $(-\infty,0]\times Y$, then we need to use $L^2_{k,loc}$-norms (see \cite{KM:book}*{\S 13}). 

With respect to a cylindrical completion $\overline X$ of $X$, the unperturbed Seiberg--Witten equations~\eqref{SWclosed} on $\bB(\overline X,\s)$ now take the form
\begin{equation}
\label{SWblowup}
D_\bA\Psi=0,\;\;F^+_\bA=s^2\frac14\rho(\Psi)
\end{equation} 
on $\bB^\sigma(\overline X,\s)$. It is explained in \cite{KM:book}*{\S24.1} how to define abstract perturbations $\p^\sigma$ to~\eqref{SWblowup} from abstract perturbations $\p$ to~\eqref{SWclosed}. In the ``cylindrical'' case $\overline X=\RR\times Y$ with spin-c structure induced from $\s$ on $Y$ and $t$-independent abstract perturbation $\p$ (see \cite{KM:book}*{Definition 10.1.1}) induced from an abstract tame perturbation on $Y$, any solution $\d$ to the $\p^\sigma$-perturbed version of~\eqref{SWblowup} on $\RR\times Y$ determines a path
$$\check\d(t):=r_t(\d)\in\bB^\sigma(Y,\s)$$
because there is a unique continuation theorem which ensures that $r_t$ is defined on each slice $\d|_{\lbrace t\rbrace\times Y}$ (see \cite{KM:book}*{\S10.8}).

In the general case of a cobordism $(X,\s):(Y_+,\s_+)\to(Y_-,\s_-)$, we fix abstract tame perturbations $\q_\pm$ on $Y_\pm$ and extend them to a suitable abstract perturbation $\p$ on $\overline X$. To fix notation, if $X$ is a symplectic cobordism with data $(\omega,\lambda_\pm$) then we denote by $\q_{\lambda_\pm}$ and $\p_\omega$ the abstract perturbations which are used in Section~\ref{Symplectic cobordisms} to define Taubes' perturbed Seiberg--Witten equations (it is explained in \cite{Taubes:ECH=SWF1}*{\S3.d} that they belong to our residual subset of abstract perturbations). With the notation of Section~\ref{Symplectic cobordisms} and suppressing additional abstract perturbations, the $iT^* Y$-component of $\q_{\lambda_\pm}$ (and of $\q^\sigma_{\lambda_\pm}$) is $-ird\lambda+2i\mu_\pm$ and the $i\bigwedge^2_+T^* X$-component of $\p_\omega$ (and of $\p^\sigma_\omega$) is $-ir\widehat{\omega}+2i\mu_*$. The $\p^\sigma_\omega$-perturbed version of~\eqref{SWblowup} is identified with~\eqref{SW4} by blowing-down and rescaling $\Psi$ by $\sqrt{2r}$.

Given $\c_\pm\in\C(Y_\pm,\s_\pm)$, we denote by $M(\c_-,\c_+;\s)$ the subset of $\p^\sigma$-perturbed Seiberg--Witten solutions $\d\in\bB^\sigma(\overline X,\s)$ for which $\check\d$ (on the ends of $\overline X$) is asymptotic to $\c_\pm$ as $t\to\pm\infty$. Depending on the context, we may alternatively write $M(\c_-,X,\c_+;\s)$ to make the manifold explicit. In the ``cylindrical'' case $\overline X=\RR\times Y$ with $t$-independent perturbation there is an $\RR$-action by translation on $M(\c_-,\c_+;\s)$. The resulting equivalence class of unparametrized nontrivial trajectories, where a trajectory is \textit{nontrivial} if it is not $\RR$-invariant, is denoted by $\breve M(\c_-,\c_+;\s)$. The moduli space of broken trajectories in the sense of \cite{KM:book}*{Definition 16.1.2} is denoted by $\breve M^+(\c_-,\c_+;\s)$.

\bigskip
We now revisit Section~\ref{Closed 4-manifolds}, where $\M(\s)\subset\bB(X,\s)$ for a closed 4-manifold $X$. As explained in \cite{KM:book}*{\S27}, for generic perturbations to the Seiberg--Witten equations~\eqref{SWblowup} on $\bB^\sigma(X,\s)$ the resulting moduli space of Seiberg--Witten solutions is diffeomorphic to $\M(\s)$ via the blow-down map. Thus, we will define the Seiberg--Witten invariants using the blown-up configuration space, and $\M(\s)$ will also denote the moduli space of Seiberg--Witten solutions in $\bB^\sigma(X,\s)$. The moduli space gives a well-defined element $[\M(\s)]\in H_*(\bB^\sigma(X,\s);\ZZ)$. 

\begin{definition}
\label{defn:SW}
For a given choice of homology orientation of $X$, and a given choice of $c_1(\s)$-chamber when $b^2_+(X)=1$, the \textit{Seiberg--Witten invariant} $SW_X(\s)\in\Lambda^* H^1(X;\ZZ)$ is defined as follows. Its value on $a\in\Lambda^p H_1(X;\ZZ)/\text{Torsion}$, for $p\le d(\s)$ such that $d(\s)-p$ is even, is
$$SW_X(\s)(a):=\left\langle U^{\frac{1}{2}(d(\s)-p)}a,[\M(\s)]\right\rangle\in\ZZ$$
and it is defined to be zero for all other integers $p$.
\end{definition}

%%%%%%%%%%%%%%%%%%%%%%%%%%%%%%%%%%%%%%%%%%%%%%%%%%%%%%%%%%
\subsection{Homology orientations}
\label{Homology orientations}

\indent\indent
To coherently orient the moduli spaces $M(\c_-,\c_+;\s)$, as explained in \cite{KM:book}*{\S 20, \S28.4}, we must make a $\ZZ/2\ZZ$ choice for each generator $\c_\pm$ and we must choose a \textit{(cobordism) homology orientation} of $X$. The latter is an orientation of
$$\operatorname{det}^+(X):=\det H^1(X;\RR)\otimes\det I^+(X;\RR)\otimes\det H^1(Y_+;\RR)$$
where $I^+(X;\RR)$ is defined as follows (see also \cite{KM:book}*{\S3.4}): The relative cap-product pairing
$$H^2(X,\partial X;\RR)\times H^2(X;\RR)\to H^4(X,\partial X;\RR)\cong\RR$$
induces a nondegenerate quadratic form on the kernel of the restriction map $H^2(X;\RR)\to H^2(\partial X;\RR)$, and $I^+(X;\RR)\subset H^2(X;\RR)$ is a maximal nonnegative subspace for this quadratic form. The set of homology orientations is denoted by $\Lambda(X)$.  In the case that $X=[0,1]\times Y$ there is a canonical homology orientation $\fo(X)\in\Lambda(X)$, and it is implicitly used when coherently orienting the moduli spaces of trajectories on $\RR\times Y$ to define the monopole Floer differentials. In the case that $Y_\pm=\varnothing$, we recover the notion of homology orientation of a closed 4-manifold in Section~\ref{Closed 4-manifolds}.

Likewise, the $\ZZ/2\ZZ$ set of orientations for a configuration $\c\in\bB^\sigma(Y)$ is denoted by $\Lambda(\c)$ and defined in \cite{KM:book}*{\S 20.3}. These sets are defined so that, when $\c_0$ is a reducible critical point of the unperturbed Chern--Simons--Dirac functional, there is a canonical choice $\fo(\c_0)\in\Lambda(\c)$.

\bigskip
We now explain these choices in a bit more detail, for the case that $(X,\omega)$ is a symplectic cobordism and the moduli spaces are defined using Taubes' large perturbations.

Any $\c\in\bB(Y,\s)$ determines a self-adjoint operator $\lL_\c$ which, roughly speaking, is the linearization of Taubes' perturbed Seiberg--Witten equations and the gauge group action (see \cite{Taubes:ECH=SWF1}*{\S3.d}). A monopole $\c\in\M(Y,\s)$ is \textit{nondegenerate} if the kernel of $\lL_\c$ is trivial. Similarly, the linearization of Taubes' perturbed Seiberg--Witten equations and the gauge group action at a given configuration $\d\in\bB(\c_-,\c_+;\s)$ between monopoles $\c_\pm$ determines its \textit{deformation operator}
\begin{center}
$\D_\d:L^2_1(iT^*\overline X\oplus\SS_+)\to L^2(i\bigwedge^2_+T^*\overline X\oplus\SS_-\oplus i\RR)$
\end{center}
When $\c_\pm$ are irreducible and nondegenerate, this operator is Fredholm.

Fix spin-c structures $\s_\pm$ and nondegenerate monopoles $\c_\pm$ on $Y_\pm$. Let $\bB(\c_-,\c_+)$ denote the union of $\bB(\c_-,\c_+;\s)$ over all spin-c structures on $X$ which restrict to $\s_\pm$ on $Y_\pm$, and let $\Lambda(\c_-,\c_+)$ denote the orientation sheaf of the determinant line bundle $\det\D\to\bB(\c_-,\c_+)$. The collection $\lbrace\Lambda(\c_-,\c_+)\rbrace$ over all nondegenerate monopoles $\c_\pm\in\M(Y_\pm,\s_\pm)$ satisfies the following property: Each nondegenerate monopole $\c$ has an associated $\ZZ/2\ZZ$-module $\Lambda(\c)$ such that there is a canonical isomorphism
$$\Lambda(\c_-,\c_+)\cong\Lambda(\c_-)\otimes_{\ZZ/2\ZZ}\Lambda(X)\otimes_{\ZZ/2\ZZ}\Lambda(\c_+)$$
and the orientations $\lbrace\fo(\c_-,\c_+)\in\Lambda(\c_-,\c_+)\rbrace_{\c_\pm\in\M(Y_\pm,\s_\pm)}$ are \textit{coherent} if, after fixing a homology orientation $\fo(X)\in\Lambda(X)$, there exists a corresponding set of choices $\lbrace\fo(\c)\in\Lambda(\c)\rbrace_{\c\in\M(Y_\pm,\s_\pm)}$ such that $\fo(\c_-,\c_+)=\fo(\c_-)\fo(X)\fo(\c_+)$.

If $\d$ is nondegenerate, i.e. $\coker(\D_\d)=0$, then the restriction of $\Lambda(\c_-,\c_+)$ to $\d$'s component of $\M(\c_-,\c_+;\s)$ is canonically isomorphic to the set of orientations of $\ker(\D_\d)$.

\begin{remark}
It is currently unknown whether there is a canonical homology orientation of $(X,\omega)$, except in the case of a closed symplectic 4-manifold \cite{Taubes:Gr=SW}*{\S 1.c}. But the search for a canonical choice can be ``pushed to the boundary'' $\partial X$, as follows. Fix the canonical spin-c structures $\s_{\xi_\pm}$ on $Y_\pm$ and the canonical spin-c structure $\s_\omega$ on $X$. Consider the canonical configurations
$$\c_{\xi_\pm}:=[A_{\xi_\pm},(1,0)]\in\bB(Y_\pm,\s_{\xi_\pm})$$
and the canonical configuration
$$\d_\omega:=[A_{K^{-1}},(1,0)]\in\bB(\c_{\xi_-},\c_{\xi_+};\s_\omega)$$
There are perturbations to these configurations, still denoted $\c_{\xi_\pm}$ and $\d_\omega$, which are nondegenerate solutions to Taubes' perturbed Seiberg--Witten equations for $r$ sufficiently large, and the deformation operator $\D_{\d_\omega}$ has trivial kernel and cokernel. Thus, there is a canonical orientation of $\det\D_{\d_\omega}$, i.e. a canonical choice in $\Lambda(\c_{\xi_-},\c_{\xi_+})$. If it can be shown that there are canonical choices in $\Lambda(\c_{\xi_\pm})$, then there is a canonical choice in $\Lambda(X)$.
\end{remark}

%%%%%%%%%%%%%%%%%%%%%%%%%%%%
\subsubsection{Choice of near-symplectic homology orientation}

\indent\indent
The case relevant to this paper is a (closed) near-symplectic manifold $(X,\omega)$ and the induced symplectic cobordism $(X_0,\omega)$. As explained in \cite{KM:book}*{\S3.4, \S26.1}, there is a composition law for (cobordism) homology orientations. Namely, we view $X$ as the composition of cobordisms
$$\varnothing\overset{X_0}{\xrightarrow{\hspace*{.75cm}}}\bigsqcup_{i=1}^NS^1\times S^2\overset{\nN}{\xrightarrow{\hspace*{.75cm}}}\varnothing$$
and then there is a specification
$$\Lambda(X)=\Lambda(\nN)\otimes_{\ZZ/2\ZZ}\Lambda(X_0)$$
so that a choice of homology orientation for any two objects in $\lbrace X,\nN,X_0\rbrace$ determines a homology orientation of the third object.

Now, $\nN$ is the disjoint union of $N$ copies of $S^1\times B^3$ (the tubular neighborhoods of the zero-circles of $\omega$). Since each $S^1\times B^3$ has a canonical homology orientation, a homology orientation of $\nN$ is equivalent to a choice of ordering of the zero-circles of $\omega$. Thus, once an ordering of the zero-circles and a homology orientation of $X$ have been fixed, there is an induced homology orientation of $X_0$. (Likewise, if it turns out that $\omega$ determines a canonical homology orientation of $X_0$, then a homology orientation of $X$ is determined by a choice of ordering of the zero-circles.)

%%%%%%%%%%%%%%%%%%%%%%%%%%%%%%%%%%%%%%%%%%%%%%%%%%%%%%%%%%
\subsection{Gradings and $U$-maps}

\indent\indent
The group $\Hfrom^{-*}(Y)$ has an absolute grading by homotopy classes of oriented 2-plane fields on $Y$ (see \cite{KM:book}*{\S 28} or \cite{Hutchings:revisited}*{\S 3}), the set of which is denoted by $J(Y)$. This grading of a critical point $\c\in\C(Y,\s)$ is denoted by $|\c|\in J(Y)$.

As described in \cite{KM:book}*{\S 28} and \cite{Gompf:handlebody}*{\S 4}, there is a well-defined map $J(Y)\to\Spinc(Y)$ with the following properties. If $H^2(Y;\ZZ)$ has no 2-torsion then the Euler class of the given 2-plane field uniquely determines the corresponding spin-c structure. There is a transitive $\ZZ$-action on $J(Y)$ whose orbits correspond to the spin-c structures: If $[\xi]\in J(Y)$ then $[\xi]+n$ is the homotopy class of a 2-plane field which agrees with $\xi$ outside a small ball $B^3\subset Y$ and disagrees with $\xi$ on $B^3$ by a map $(B^3,\partial B^3)\to(SO(3),\lbrace\1\rbrace)$ of degree $2n$.\footnote{This convention is opposite to that used in \cite{KM:book}.} A given orbit $J(Y,\s)$ is freely acted on by $\ZZ$ if and only if the corresponding Euler class is torsion. In particular, there is an induced relative $\ZZ/d\ZZ$ grading on $\Hfrom^{-*}(Y,\s)$, where $d$ denotes the divisibility of $c_1(\s)$ in $H^2(Y;\ZZ)/\text{Torsion}$.

It is useful to write out the relative $\ZZ$ grading on $\Hfrom^{-*}(Y,\s)$ when $\s$ is torsion, as follows. Given $\c_\pm\in\C(Y,\s)$ and cylindrical metric on $\RR\times Y$ and $t$-independent abstract perturbation, each trajectory $\d\in M(\c_-,\c_+;\s)$ over $\RR\times Y$ has a Fredholm operator $Q_\d$ which, roughly speaking, is the linearization of the perturbed version of~\eqref{SWblowup} and the gauge group action (see \cite{KM:book}*{\S14.4}). The \textit{relative grading} $\gr(\c_-,\c_+)$ between $\c_-$ and $\c_+$ is defined to be the Fredholm index of $Q_\d$ for any $\d\in M(\c_-,\c_+;\s)$, and
$$|\c_+|=|\c_-|+\gr(\c_-,\c_+)$$
as expected. The fact that this index does not depend on the choice of $\d$ (for $\s$ torsion) follows immediately from \cite{KM:book}*{Proposition 14.4.5, Lemma 14.4.6}.

\bigskip
As described in \cite{KM:book}*{\S 23.2, \S25.3}, there is an $\AA(Y)$-module structure on $\Hto_*(Y)$ and $\Hfrom^*(Y)$. We now clarify the action of $U\in\AA(Y)$ for the moduli spaces of solutions to the large $r$ version of Taubes' perturbed Seiberg--Witten equations over a connected contact 3-manifold $Y$. Fix a base point $y\in Y$ and consider SW instantons on $\RR\times Y$ for a given spin-c structure $\s$ on $Y$. Denote by $\M_2(\c_-,\c_+;\s,y)$ the subset of SW instantons $[A,(\alpha,\beta)]\in\M_2(\c_-,\c_+;\s)$ for which $\alpha\in\Gamma(E)$ vanishes at $(0,y)\in\RR\times Y$. The action of $U$ then becomes a degree $-2$ chain map\footnote{This definition appears in \cite{Taubes:ECH=SWF5}*{\S1.b} and agrees with that in \cite{KM:book}*{\S25.3} and \cite{KMOS:monopoles}*{\S4.11}. See \cite{HFHM5}*{\S2.5} for details.}
$$U_y:\Hfrom^{-*}(Y,\s)\to\Hfrom^{2-*}(Y,\s)$$
that counts the elements of $\M_2(\c_-,\c_+;\s,y)$, such that on the level of cohomology this $U$-map does not depend on the choice of base point $y$.

\bigskip
The chain complex $\Cfrom_j(Y,\s)$ for $\Hfrom_j(Y,\s)$ in grading $j$ is a finitely generated free abelian group (see \cite{KM:book}*{Lemma 22.3.3}), and the cochain complex for $\Hfrom^j(Y,\s)$ in grading $j$ is then defined by $\Cfrom^j=\Hom(\Cfrom_j,\ZZ)$. If $Y$ is disconnected, then $\Cfrom_*(Y,\s)$ and $\Cfrom^*(Y,\s)$ are the tensor products of the respective (co)chain complexes of the components of $Y$. The same applies to the other flavors of monopole Floer (co)homology. There is one $U$-map for each connected component of $Y$, namely the tensor product of the $U$-map for the corresponding component with the identity map on the other factors.

%%%%%%%%%%%%%%%%%%%%%%%%%%%%%%%%%%%%%%%%%%%%%%%%%%%%%%%%%%
%%%%%%%%%%%%%%%%%%%%%%%%%%%%%%%%%%%%%%%%%%%%%%%%%%%%%%%%%%
\section{Review of Taubes' isomorphisms}
\label{Review of Taubes' isomorphisms}

\indent\indent
What follows is some background and an introduction to Taubes' relation between gauge theory and pseudoholomorphic curve theory. This is necessary review in order to mimic the story for cobordisms in Section~\ref{Equating ECH and HM cobordism counts}. Complete details are found in \cite{Taubes:SWGrBook,HutchingsTaubes:Arnold2,Taubes:ECH=SWF1,Taubes:ECH=SWF2,Taubes:ECH=SWF3,Taubes:ECH=SWF4,Taubes:ECH=SWF5}.

%%%%%%%%%%%%%%%%%%%%%%%%%%%%%%%%%%%%%%%%%%%%%%%%%%%%%%%%%%
\subsection{Vortices and orbits and curves}
\label{Vortices and orbits and curves}

\indent\indent
The isomorphism between monopole Floer homology and ECH was inspired by the equivalence of the 4-dimensional invariants, the Seiberg--Witten invariants and the Gromov invariants. But these relations were preceded (and depended on) the analogous correspondence in two dimensions, between vortices and points in the complex plane.

A pair $(A,\alpha)$ consists of a Hermitian connection on the trivial complex line bundle $\underline\CC\to\CC$ and a section of it, and $\c$ denotes its gauge-equivalence class under the gauge group $C^\infty(\CC,S^1)$. Given a nonnegative integer $n$, the \textit{n-vortex equations} for a configuration $\c$ are
$$* F_A=-i(1-|\alpha|^2),\;\;\dbar_A\alpha=0,\;\;|\alpha|\le1,\;\;\int_\CC(1-|\alpha|^2)\dv=2\pi n$$
The solutions are called \textit{n-vortices}, and their moduli space is denoted by $\C_n$. For $n=0$ this space is the single point $(0,1)$ up to gauge-equivalence, and when $n>0$ this space has the structure of a complex manifold that is biholomorphic to $\CC^n$ via the map~\eqref{eqn:biholo} below. In fact,

\begin{theorem}[\cite{JaffeTaubes, Taubes:vortex}]
Given a nonnegative integer $n$ and a collection of (not necessarily distinct) points $z_1,\ldots,z_n\in \CC$, there exists a unique solution $(A,\alpha)$ of the vortex equation (up to gauge-equivalence) having finite energy and vortex number $n$ and satisfying
$$\alpha^{-1}(0)=\bigcup_{j=1}^n\lbrace z_j\rbrace$$
Conversely, all finite energy solutions having vortex number $n\ge0$ are gauge-equivalent to a solution of this form.
\end{theorem}

The biholomorphisms $\C_n\approx\text{Sym}^n(\CC)\approx\CC^n$ used in this paper are given by
\begin{equation}
\label{eqn:biholo}
\c\mapsto\lbrace z_1,\ldots,z_n\rbrace\mapsto(\sigma_1,\ldots,\sigma_n),\indent \sigma_k=\sum_{j=1}^n z_j^k=\frac{1}{2\pi}\int_\CC z^k(1-|\alpha|^2)\dv
\end{equation}
and the origin $0\in \CC^n$ corresponds to the unique ``symmetric'' vortex $(A,\alpha)$ satisfying $\alpha^{-1}(0)=0$ (see \cite{Taubes:GrtoSW}*{\S2.b.4, \S2.b.5, Lemma 2.3} which also specifies the natural $S^1$-action on $\C_n$, under which this ``symmetric'' vortex is the unique fixed point).

Given $\mu\in C^\infty(S^1,\CC)$ and $\nu\in C^\infty(S^1,\RR)$, the function
$$\h_{\mu\nu}:\C_n\to\RR,\indent (A,\alpha)\mapsto\frac{1}{4\pi}\int_\CC(2\nu|z|^2+\mu\bar z^2+\bar\mu z^2)(1-|\alpha|^2)\dv$$
induces a time-dependent Hamiltonian vector field for a particular K{\"a}hler metric on $\C_n$, whose closed integral curves $\c(t):S^1\to\C_n$ satisfy
\begin{equation}
\label{eqn:corbit}
\frac{i}{2}\c_*(\partial_t)^{(1,0)}+\nabla^{(1,0)}\h_{\mu\nu}|_\c=0
\end{equation}
where $\nabla^{(1,0)}$ denotes the holomorphic part of the gradient. A solution is \textit{nondegenerate} if the linearization of this equation, with respect to a certain covariant derivative, at the solution has trivial kernel (see \cite{Taubes:ECH=SWF2}*{\S2.b.3}).

\begin{theorem}(\cite{Taubes:ECH=SWF2}*{\S2.b})
\label{thm:corbitL}
Let $(\mu,\nu)$ denote the pair associated with an $L$-flat nondegenerate Reeb orbit $\gamma$ of a contact 3-manifold, and $n$ a positive integer. If $\gamma$ is elliptic then there is a single (nondegenerate) solution to~\eqref{eqn:corbit}, the constant map to the unique ``symmetric'' vortex. The same result holds if $\gamma$ is hyperbolic and $n=1$. If $n>1$ and $\gamma$ is hyperbolic, then there are no solutions. 
\end{theorem}

\begin{remark}
The solutions granted by Theorem~\ref{thm:corbitL} are used to define the isomorphism between ECH and monopole Floer homology. The fact that there are no solutions when $n>1$ and $\gamma$ is hyperbolic is not a problem, because such a pair $(\gamma,n)$ does not arise in an admissible orbit set.
\end{remark}

Now we consider $J$-holomorphic curves in the completion $(\overline X,\omega,J)$ of a symplectic cobordism $(X,\omega):(Y_+,\lambda_+)\to(Y_-,\lambda_-)$, including the special case of a symplectization $\overline X=\RR\times Y$. Denote by $\mM(\Theta^+,\Theta^-)$ the moduli space of $J$-holomorphic currents in $\overline X$ asymptotic to orbit sets $\Theta^\pm$ (see \cite{Gerig:taming}*{\S2.2}), and by $\mM_I(\Theta^+,\Theta^-)$ the subset of such currents with ECH index $I\in\ZZ$. Let $\pi:N_C\to C$ be the (holomorphic) normal bundle of an immersed connected $J$-holomorphic curve $C$ in $\mM(\Theta^+,\Theta^-)$, and denote by $S_{N_C}\subset N_C$ the unit circle subbundle. Form the \textit{n-vortex bundle}
$$\C_{N_C,n}:= S_{N_C}\times_{S^1}\C_n$$
whose projection onto $C$ will also be denoted by $\pi$. With respect to a Hermitian metric and compatible connection on $N_C$, the $(1,0)$-part of its vertical tangent space is
$$T_{1,0}^\text{vert}\C_{N_C,n}=(\ker d\pi)_{1,0}=S_{N_C}\times_{S^1}T_{1,0}\C_n$$
Sections $\c\in\Gamma(\C_{N_C,n})$ can be viewed as $S^1$-invariant maps $S_{N_C}\to\C_n$, so their covariant derivative can be taken and restricted to the horizontal subspace $T^\text{hor}S_{N_C}$. This defines a ``del-bar'' operator
$$\c\mapsto \dbar\c\in\Gamma(\c^* T_{1,0}^\text{vert}\C_{N_C,n}\otimes T^{0,1}C)$$
Given the pair $(\nu_C,\mu_C)\in\Gamma(T^{0,1}C)\times\Gamma(T^{0,1}C\otimes N_C^2)$ associated with the deformation operator $D_C$ of $C$ (see \cite{Gerig:taming}*{\S2.5}), we define the section $\h_{\nu_C\mu_C}$ of $\pi^* T^{0,1}C\to\C_{N_C,n}$ pointwise by 
$$\h_{\nu_C\mu_C}\big(p,(A_p,\alpha_p)\big)=\frac{1}{4\pi}\int_\CC\left[2\nu_{C,p}|z|^2+(\mu_{C,p}\bar z^2+\bar\mu_{C,p} z^2)\right](1-|\alpha_p|^2)\dv$$
after identifying a fiber of $N_C$ with $\CC$ and identifying a point in $\C_{N_C,n}$ with $(p,(A_p,\alpha_p))\in C\times\C_n$. Denote by $\nabla^{1,0}\h_{\nu_C\mu_C}$ the corresponding section of $T_{1,0}^\text{vert}\C_{N_C,n}\otimes\pi^* T^{0,1}C$.

\bigskip
As will become evident in Section~\ref{Maps between ECH and HM}, we are interested in certain sections of $\C_{N_C,n}$ asymptotic to zero on the ends of $C$. They should be considered as objects that are ``halfway'' between $J$-holomorphic curves and SW instantons. They were originally introduced in \cite{Taubes:GrtoSW}*{\S3}, see also \cite{Taubes:ECH=SWF2}*{\S2.e}, and are defined as follows.

\begin{definition}
\label{def:csurface}
The subspace $\zZ_0\subset L^2_1(\C_{N_C,n})$ consists of those sections $\c$ satisfying
\begin{equation}
\label{eqn:csurface}
\dbar\c+\c^*\nabla^{1,0}\h_{\nu_C\mu_C}=0
\end{equation}
\end{definition}

\noindent We can spell out this subspace more concretely using the biholomorphism~\eqref{eqn:biholo}. Namely,~\eqref{eqn:biholo} induces a bundle isomorphism $\C_{N_C,n}\cong \bigoplus^n_{j=1}N_C^j$ and under this isomorphism~\eqref{eqn:csurface} takes the form
\begin{equation}
\label{eqn:csurface2}
\dbar\eta+\nu_C\aleph(\eta)+\mu_C\FF(\eta)=0
\end{equation}
for sections $\eta\in L^2_1(\bigoplus^n_{j=1}N_C^j)$. Here, $\dbar$ is the del-bar operator with respect to the Hermitian connection on $N_C$,
$$\aleph:\Gamma(\bigoplus^n_{j=1}N_C^j)\to\Gamma(\bigoplus^n_{j=1}N_C^j)$$
is the map that multiplies the $j^\text{th}$ summand by $j$, and
$$\FF:\Gamma(\bigoplus^n_{j=1}N_C^j)\to\Gamma(\bigoplus^n_{j=1}N_C^{j-2})$$
is some fiber-preserving bundle map that is not $\RR$-linear unless $n=1$. When $n=1$, this map $\FF$ is the complex conjugation operator and so~\eqref{eqn:csurface2} becomes
$$\dbar\eta+\nu_C\eta+\mu_C\bar\eta=0$$
for sections $\eta\in L^2_1(N_C)$. Thus, $\zZ_0$ is identified with $\ker(D_C)$ when $n=1$. For example, if $C$ is a Fredholm index zero curve cut out transversely then $\zZ_0$ is a point, the constant map to the unique ``symmetric'' vortex.

\begin{remark}
\label{Zregular1}
In the setting of \cite{Taubes:ECH=SWF1} where $\overline X$ is a symplectization, the only multiply covered curves were $\RR$-invariant cylinders with $\mu_C=0$. The nonlinear map $\FF$ only played a role in the setting of \cite{Taubes:SWGrBook} where $\overline X$ is a closed manifold, due to the existence of multiply covered tori with $\mu_C\ne0$. In this paper, the space of solutions to~\eqref{eqn:csurface} for $n>1$ will be of concern whenever $C$ is a special curve.
\end{remark}

In \cite{Taubes:ECH=SWF2}*{\S 2.f}, Taubes introduces the appropriate Morrey spaces as certain completions of the spaces of compactly supported sections of the bundles $\c^* T_{1,0}^\text{vert}\C_{N_C,n}$ and $\c^* T_{1,0}^\text{vert}\C_{N_C,n}\otimes T^{0,1}C$. They are denoted by $\kK_\c$ and $\lL_\c$, respectively. A nice feature of these Banach space completions is that the small normed elements can be used to define deformations of $\c$ (see \cite{Taubes:ECH=SWF2}*{Equation 2-28}).

At any given $\c\in L^2_1(\C_{N_C,n})$, define the $\RR$-linear operator $\mathbf\Delta_\c:\kK_\c\to \lL_\c$ to be the linearization of the operator in~\eqref{eqn:csurface} which cuts out $\zZ_0$. From the viewpoint of $\bigoplus^n_{j=1}N_C^j$ the corresponding linearization of the operator in~\eqref{eqn:csurface2} is
\begin{equation}
\label{eqn:Delta}
\Delta_\c:\zeta\mapsto \dbar \zeta+\nu_C\aleph(\zeta)+\mu_Cd\FF_\c(\zeta)
\end{equation}
When $\c\in\zZ_0$ the operators $\mathbf\Delta_\c$ and $\Delta_\c$ are conjugate to each other using a bundle isomorphism induced from~\eqref{eqn:biholo} (see \cite{Taubes:GrtoSW}*{Proposition 3.2(5)}). An element $\c\in \zZ_0$ is called \textit{regular} if $\coker(\mathbf\Delta_\c)=0$, or equivalently $\coker(\Delta_\c)=0$.

%%%%%%%%%%%%%%%%%%%%%%%%%%%%%%%%%%%%%%%%%%%%%%%%%%%%%%%%%%
\subsection{Maps between ECH and HM}
\label{Maps between ECH and HM}

\indent\indent
In \cite{Taubes:ECH=SWF1}, Taubes defined a canonical isomorphism of relatively graded $\ZZ/d\ZZ$-modules (with $\ZZ/2\ZZ$ coefficients)
\begin{equation}
\label{eqn:isomorphism}
ECH^L_*(Y,\lambda,\Gamma,J)\cong\Hfrom^{-*}_L(Y,\lambda,\s_\xi+\PD(\Gamma),J,r)
\end{equation}
under the assumption that $r$ is sufficiently large and $(\lambda,J)$ is a generic\footnote{\label{generic1}Genericity here is in the sense that $J$ is chosen from a residual subset of almost complex structures (endowed with the $C^\infty$ Fr\'echet space topology) for which the $L$-filtered ECH is well-defined.} $L$-flat pair, where $d$ denotes the divisibility of $c_1(\s_\xi+\PD(\Gamma))=c_1(\xi)+2\PD(\Gamma)$ in $H^2(Y;\ZZ)/\text{Torsion}$. Taking the direct limit as $L\to\infty$, Taubes' isomorphism becomes
$$ECH_*(Y,\xi,\Gamma)\cong\Hfrom^{-*}(Y,\s_\xi+\PD(\Gamma))$$
A detailed explanation can be found in \cite{HutchingsTaubes:Arnold2}, specifically the proof of \cite{HutchingsTaubes:Arnold2}*{Theorem 1.3}. Taubes' isomorphism also preserves the absolute gradings by homotopy classes of oriented 2-plane fields
$$ECH_j(Y,\xi)\cong\Hfrom^j(Y)$$
where $ECH_*(Y,\xi):=\bigoplus_{\Gamma\in H_1(Y;\ZZ)}ECH_*(Y,\xi,\Gamma)$ and $j\in J(Y)$ (see \cite{Gardiner:gradings}), and it also intertwines the respective $U$-maps (see \cite{Taubes:ECH=SWF5}*{Theorem 1.1}). We now briefly explain how this isomorphism~\eqref{eqn:isomorphism} was constructed.

\begin{theorem}[Taubes]
\label{thm:generators}
Fix $L>0$ and a generic\textsuperscript{\normalfont\ref{generic1}} $L$-flat pair $(\lambda,J)$ on the nondegenerate contact 3-manifold $(Y,\lambda)$. Then for $r$ sufficiently large and $\Gamma\in H_1(Y;\ZZ)$, there is a canonical bijection from the set of generators of $\Cfrom_L^*(Y,\lambda,\s_\xi+\PD(\Gamma),J,r)$ to the set of generators of $ECC^L_*(Y,\lambda,\Gamma,J)$.
\end{theorem}

As shown in \cite{Taubes:ECH=SWF1}*{Theorem 4.2}, the isomorphism~\eqref{eqn:isomorphism} is actually established on generators of the chain complexes, so that for any given admissible orbit set $\Theta\in ECC^L_*(Y,\lambda,\Gamma,J)$ there exists a unique irreducible monopole
$$\c_\Theta=[A_\Theta,(\alpha_\Theta,\beta_\Theta)]\in\Cfrom^{-*}_L(Y,\lambda,\s_\xi+\PD(\Gamma),J,r)$$
To construct $\c_\Theta$, an arbitrary smooth map $\c_{\Theta_i}:S^1\to\C_{m_i}$ is first assigned to each pair $(\Theta_i,m_i)\in\Theta$. From such a map, a ``potential candidate'' $(A_r,(\alpha_r,0))$ for a monopole is constructed which almost solves Taubes' perturbed Seiberg--Witten equations~\eqref{SW3} when $r$ is large; this would not be possible if $\Theta_i$ was not a Reeb orbit (see \cite{Taubes:ECH=SWF2}*{Lemma 3.4}). When $\c_{\Theta_i}$ satisfies~\eqref{eqn:corbit}, perturbation theory is then used to find an honest solution nearby this candidate. In fact, by Theorem~\ref{thm:corbitL} there is precisely one such map $\c_{\Theta_i}$.

As shown in \cite{Taubes:ECH=SWF1}*{Theorem 4.3}, the chain complex differentials also agree: Given two generators $\Theta^\pm\in ECC^L_*(Y,\lambda,\Gamma,J)$ and their corresponding generators $\c_{\Theta^\pm}\in\Cfrom^{-*}_L(Y,\lambda,\s_\xi+\PD(\Gamma),J,r)$, for $r$ sufficiently large there is an orientation-preserving diffeomorphism between $\mM_1(\Theta^+,\Theta^-)$ and $\M_1(\c_{\Theta^-},\c_{\Theta^+};\s)$ which is $\RR$-equivariant with respect to the translation actions. In analogy with the construction for chain complex generators, this diffeomorphism
\begin{equation}
\label{eq:bijection}
\Psi_r:\mM_1(\Theta^+,\Theta^-)\to\M_1(\c_{\Theta^-},\c_{\Theta^+};\s)
\end{equation}
is constructed as follows (see \cite{Taubes:ECH=SWF2}*{\S5} for more details). Given $\cC\in\mM_1(\Theta^+,\Theta^-)$ there exists a complex line bundle $E\to\RR\times Y$ and a ``potential candidate''
$$(A^*,\Psi^*)\in\text{Conn}(E)\oplus\Gamma(E\oplus K^{-1}E)$$
for a SW instanton which almost solves Taubes' perturbed Seiberg--Witten equations~\eqref{SW4} when $r$ is large. Here, $E$ has a section whose zero set (with multiplicity) is $\cC$, and away from $\cC$ the bundle $E$ is identified with the trivial bundle. Away from $\cC$ the pair $(A^*,\Psi^*)$ is close to $(A_0,(1,0))$, where $A_0$ is the flat connection on $E$ coming from the product structure. Near a component $(C,d)\in\cC$ the pair $(A^*,\Psi^*)$ is determined by the zero-section $0\in\zZ_0\subset\Gamma(\C_{N_C,d})$. Then an implicit function theorem argument may be used to perturb $(A^*,\Psi^*)$ to a SW instanton $\Psi_r(\cC)$.

\begin{remark}
\label{Zregular2}
This remark is intended to complement Remark~\ref{Zregular1}. Roughly speaking, we can generalize the construction of $(A^*,\Psi^*)$ using nontrivial sections in $\zZ_0\subset L^2_1(\C_{N_C,d})$, such that if all $\zZ_0$ were regular then sufficiently $L^\infty$-small elements of $\zZ_0$ would parametrize a neighborhood of $\Psi_r(\cC)$ in the SW moduli space, and if some $\zZ_0$ were zero-dimensional (with additional hypotheses on $(C,d)$) then there would be an associated SW instanton for each point in that $\zZ_0$. When $d=1$, the operator~\eqref{eqn:Delta} is precisely the deformation operator $D_C$ which has trivial cokernel (for generic $J$), and hence $\zZ_0$ consists of regular points. When $d>1$, there is no guarantee that any point of $\zZ_0$ is regular if $\mu_C\ne 0$. Luckily for ECH, if $d>1$ then $C$ is an $\RR$-invariant cylinder with $\mu_C=0$, in which case $\zZ_0$ consists of a single regular point (the zero-section).
\end{remark}

Given a basepoint $y\in Y$, there is an analogous bijection between point-constrained moduli spaces $\mM_2(\Theta^+,\Theta^-;y)\to\M_2(\c_{\Theta^-},\c_{\Theta^+};\s,y)$ used to relate the respective $U$-maps. It is established in \cite{Taubes:ECH=SWF5}*{Theorem 2.6} and closely follows the construction of $\Psi_r$.

We end this section by pointing the reader to Taubes' papers for the various properties of $\Psi_r$. The assertion that $\Psi_r$ maps $\mM_1(\Theta^+,\Theta^-)$ onto a union of components of $\M_1(\c_{\Theta^-},\c_{\Theta^+};\s)$ that contain solely nondegenerate SW instantons is given in \cite{Taubes:ECH=SWF3}*{\S3.a}. The proof that $\Psi_r$ is surjective consists of three main arguments spelled out in \cite{Taubes:ECH=SWF4}*{\S3--\S7}. First, certain global properties of SW instanton solutions to Taubes' perturbed Seiberg--Witten equations are established in \cite{Taubes:ECH=SWF4}*{\S3} and in \cite{Taubes:ECH=SWF4}*{Lemma 5.2, Lemma 5.3}. Second, these global results are used to assign an element in $\mM_1(\Theta^+,\Theta^-)$ to a given SW instanton in $\M_1(\c_{\Theta^-},\c_{\Theta^+};\s)$ and is done so in \cite{Taubes:ECH=SWF4}*{\S4} and in \cite{Taubes:SWtoGr}*{\S5--\S7}. Third, these assignments are given by the map $\Psi_r$ and is done so in \cite{Taubes:ECH=SWF4}*{\S6--\S7}.

%%%%%%%%%%%%%%%%%%%%%%%%%%%%%%%%%%%%%%%%%%%%%%%%%%%%%%
%%%%%%%%%%%%%%%%%%%%%%%%%%%%%%%%%%%%%%%%%%%%%%%%%%%%%%
\section{Equating ECH and HM cobordism counts}
\label{Equating ECH and HM cobordism counts}

\indent\indent
In light of Section~\ref{Exceptional spheres}, we now fix $(X,\omega,\s)$ such that $E\cdot \tau_\omega(\s)\ge-1$ for all $E\in\eE_\omega$. Subsequently, $(-\partial X_0,\lambda_\s)$ and $(\overline X_0,\omega,J)$ are also fixed. We also fix an ordering of the zero-circles of $\omega$, as well as a homology orientation of $X_0$ (hence of $X$). Denote by
$$\s_\s:=\s_\omega+\tau_\omega(\s)$$
the spin-c structure on $X_0$ (and on $\overline X_0$) that corresponds to the relative class $\tau_\omega(\s)$, identified with the spin-c structure $\s_{\xi_0}+1$ on the ends of $\overline X_0$. Denote by $E\to\overline X_0$ the complex line bundle for the spinor decomposition $\SS_+=E\oplus K^{-1}E$ associated with $\s_\s$. Now, make the following choices:

\bigskip
	$\bullet$ an integer $I\ge0$,
	
	$\bullet$ an integer $p\in\lbrace0,\ldots,I\rbrace$ such that $I-p$ is even,
	
	$\bullet$ an ordered set of $p$ disjoint oriented loops $\bar\eta:=\lbrace\eta_1,\ldots,\eta_p\rbrace\subset X_0$,
	
	$\bullet$ a set of $\frac{1}{2}(I-p)$ disjoint points $\bar z:=\lbrace z_1,\ldots,z_{(I-p)/2}\rbrace\subset X_0-\bar\eta$.

\bigskip	
\noindent
Recall from \cite{Gerig:taming} that the choice $I=d(\s)$ is used to define $Gr_{X,\omega}(\s)$. Denote by $\M_I(\c,\varnothing;\s_\s,\bar z,\bar\eta)$ the subset of SW instantons $\d=[A,\Psi=(\alpha,\beta)]\in\M_I(\c,\varnothing;\s_\s)$ for which $\alpha\in\Gamma(E)$ vanishes at each point $z_i\in\bar z$ and at some point $w_i$ along each loop $\eta_i\in\bar\eta$.

Although the upcoming Theorem~\ref{thm:moduli} is a statement using $\ZZ/2\ZZ$ coefficients, we specify the coherent orientations and the signed weights attached to SW instantons because it is expected that the theorem also holds over $\ZZ$. To define the sign $q(\d)$ attached to each $\d\in\M_I(\c,\varnothing;\s_\s,\bar z,\bar\eta)$, we build the $\RR^I$-vector space
$$V_\d:=\bigoplus^{\frac{1}{2}(I-p)}_{i=1}E_{z_i}\oplus\bigoplus_{i=1}^{p}\big(E_{w_i}/\nabla_{A}\alpha(T_{w_i}\eta_i)\big)$$
as in \cite{Taubes:Gr=SW}*{\S2.c}. For generic choices of perturbations $\mu\in\Omega^2(\overline X_0)$ that define Taubes' perturbed Seiberg--Witten equations~\eqref{SW4}, the covariant derivative of $\alpha$ along $\eta_i$ at $w_i$ is nonzero and the restriction map
$$\ker(\D_\d)\to V_\d$$
is an isomorphism. Then $q(\d)=\pm1$ depending on whether this restriction map is orientation-preserving or orientation-reversing. Here, $\ker(\D_\d)$ is oriented by the coherent orientations (see Section~\ref{Homology orientations}), and $V_\d$ is naturally oriented because the complex bundle $E$ is oriented, the loops $\gamma_i$ are oriented, and the points $w_i$ are ordered.

\begin{notation}
\label{notation}
Fix an admissible orbit set $\Theta$ with action less than $\rho_\s$ and its corresponding monopole $\c_\Theta$. Since the moduli spaces $\mM_I(\varnothing,\Theta;\tau_\omega(\s),\bar z,\bar\eta)$ and $\M_I(\c_\Theta,\varnothing;\s_\s,\bar z,\bar\eta)$ will appear often, we denote them by $\mM$ and $\M$ unless otherwise specified. We remind the reader that both $\mM$ and $\M$ depend on $(\lambda,J)$ while $\M$ also depends on $r$ (and abstract perturbations).
\end{notation}

\begin{theorem}
\label{thm:moduli}
For an admissible orbit set $\Theta$ with action less than $\rho_\s$, generic $J$ as needed to make~\ref{def:GromovCycle} well-defined, and sufficiently large $r$,
$$\sum_{\cC\in\mM}q(\cC)\equiv \sum_{\d\in\M}q(\d)\;\mod2$$
\end{theorem}

The proof of Theorem~\ref{thm:moduli} is spelled out in the following subsections. A key point is that all relevant constructions of Taubes' isomorphisms which occur in a symplectization and a closed symplectic manifold generalize to (completed) symplectic cobordisms, because the analysis takes place local to the $J$-holomorphic curves. We will construct a multi-valued bijective map\footnote{A multi-valued map is bijective if the image of the domain is the full codomain (surjectivity) and the images of any two points in the domain are disjoint (injectivity).}
$$\Psi_r:\mM\to\M$$
such that $q(\cC)$ is equal to the number of points in $\Psi_r(\cC)$ for each $\cC\in\mM$ (see Definition~\ref{def:multivalued}). The construction is similar to that in Section~\ref{Maps between ECH and HM}, but the analysis associated with $\RR$-invariant cylinders (in \cite{Taubes:ECH=SWF1}) disappears and the analysis associated with holomorphic tori (in \cite{Taubes:SWGrBook}) appears. The map $\Psi_r$ is multi-valued because, for a multiply covered special curve, the associated space $\zZ_0$ may consist of more than a single (possibly non-regular) point. Such a complication was alluded to in Remarks~\ref{Zregular1} and~\ref{Zregular2}.

Here is an outline of the following subsections. Section~\ref{SW instantons from multiply covered} preemptively studies the spaces $\zZ_0$ (given by Definition~\ref{def:csurface}) associated with special curves. The knowledge obtained about $\zZ_0$ is then crucially used in Section~\ref{Kuranishi structures} to construct the multi-valued injective map $\Psi_r:\mM\to\M$. Section~\ref{Surjectivity} then explains surjectivity of $\Psi_r$ while Section~\ref{Proof of Theorem} relates the image of $\Psi_r$ to the weighted count of points in $\mM$, completing the proof of Theorem~\ref{thm:moduli}.

To further ground the reader, similarly to viewing \cite{Taubes:ECH=SWF2}*{\S5--\S7} as a Floer-theoretic variant of the arguments for \cite{Taubes:Gr=SW}*{Proposition 4.1}, the upcoming two sections may be viewed as a Floer-theoretic variant of the arguments for \cite{Taubes:Gr=SW}*{Proposition 4.2} which generalizes \cite{Taubes:Gr=SW}*{Proposition 4.1}.

%%%%%%%%%%%%%%%%%%%%%%%%%%%%%%%%%%%%%%%%%%%%%%%%%%%%%%
\subsection{SW instantons from multiply covered tori and planes}
\label{SW instantons from multiply covered}

\indent\indent
In regards to constructing the analogous bijection $\mM_1(\Theta^+,\Theta^-)\to\M_1(\c_{\Theta^-},\c_{\Theta^+};\s)$ over a symplectization (see~\eqref{eq:bijection}), the analysis in \cite{Taubes:ECH=SWF1} to handle covers of $\RR$-invariant cylinders is simple: for an $\RR$-invariant cylinder $C=\RR\times\gamma$, the section $\mu_C$ vanishes and hence the unique ``symmetric'' vortex which solves~\eqref{eqn:corbit} over $\gamma$ extends to the unique ``symmetric'' vortex solution to~\eqref{eqn:csurface2} over $C$. If $\mu_C$ did not vanish, the bundle map $\FF$ would prevent such an extension (since $\FF(0)\ne0$). This unfortunately turns out to be the case for the special curves in the scenario at hand. The appropriate analysis to handle the special tori is given in \cite{Taubes:SWGrBook} and can be mimicked to handle the special planes. A key point to mimicking Taubes' analysis here is that the special planes limit to elliptic orbits \textit{with multiplicity one}, which allows us to bypass the analysis at the ends in this case and instead places us in the simplified setting \cite{Taubes:ECH=SWF2}*{\S5.a.1}.

\begin{prop}
\label{prop:Zcompact}
Let $(C,d)$ be a component of $\cC\in\mM$ where $C$ is a special curve and $d\ge1$. Then the subspace $\zZ_0\subset L^2_1(\bigoplus^d_{j=1}N_C^j)$ carved out by~\eqref{eqn:csurface2} is compact for generic $J$ as needed to make~\ref{def:GromovCycle} well-defined. If $C$ is furthermore a plane, then by replacing $\mu_C$ in the definition of $\zZ_0$ with $t\mu_C$ for any $t\in[0,1]$, the resulting space of sections $\zZ_0^t$ remains compact for each $t$.
\end{prop}

\begin{proof}
When $C$ is a torus, this is precisely \cite{Taubes:Gr=SW}*{Proposition 2.8} which in turn is the combination of \cite{Taubes:Gr=SW}*{Proposition 2.7} and the fact that $(C,d)$ is ``$d$-nondegenerate'' (see \cite{Taubes:counting}*{Lemma 5.4}). When $C$ is a plane, we may copy the proof of \cite{Taubes:Gr=SW}*{Proposition 2.7} in \cite{Taubes:Gr=SW}*{\S 3} for the analogous case that $C$ is a torus. Strictly speaking, \cite{Taubes:Gr=SW}*{Proposition 2.7} argues that $\zZ_0$ is compact if and only if certain Cauchy--Riemann type operators associated to multiple covers of $C$ have trivial kernel and cokernel (which is then guaranteed by $d$-nondegeneracy of $C$). As the relevant analysis is performed locally along the curve, the argument extends almost verbatim to the case that $C$ is a plane. The argument is actually shorter than that for tori, and what follows is a brief sketch.

Let $\pi:N_C\to C$ denote the normal bundle of $C$ with respect to $\overline X_0$, and let $s:N_C\to\pi^* N_C$ denote the tautological section. Define the maps
$$R:L^2_1(\bigoplus_{q=1}^d N_C^q)\to\RR,\indent y=(y_1,\ldots,y_d)\mapsto \sup_C\sup_q|y_q|^{1/q}$$
$$p:L^2_1(\bigoplus_{q=1}^d N_C^q)\to L^2_1(\pi^* N_C^d),\indent y=(y_1,\ldots,y_d)\mapsto s^d+\pi^* y_1\cdot s^{d-1}+\cdots+\pi^* y_d$$
and suppose there exists a sequence $\lbrace y^j\rbrace_{j\in\NN}\subset \zZ^t_0$ with $R_j:=R(y^j)$ increasing and unbounded. Then define the subsets
$$\Sigma_j:=\lbrace\eta\in N_C\;|\;p(y^j)(R_j\cdot\eta)=0\rbrace$$
The proof of \cite{Taubes:Gr=SW}*{Proposition 3.1} shows that a subsequence of $\lbrace\Sigma_j\rbrace_{j\in\NN}$ converges pointwise to the image of a somewhere-injective $J^t$-holomorphic map $\varphi:C'\to N_C$ which does not factor through the zero-section of $N_C$. Here, $J^t$ on $TN_C=\pi^* N\oplus \pi^* TC$ is the unique almost complex structure defined by the condition that $T^{1,0}N_C$ is locally spanned by $\pi^* dz$ and $(\nabla s+\pi^*\nu_C\cdot s+t\cdot\pi^*\mu_C\cdot\bar s)\cdot\pi^* d\bar z$, where $\nabla$ is the Levi-Civita connection induced by the metric on $\overline X_0$.

Now, the composition $\pi\circ\varphi:C'\to C$ is a $d$-fold (possibly branched\footnote{The proof of \cite{Taubes:Gr=SW}*{Proposition 2.7} (and in particular, the proof of \cite{Taubes:Gr=SW}*{Proposition 3.2}) for a torus $C$ is more involved because it is shown there that $\pi\circ\varphi:C'\to C$ is unbranched, i.e. the domain $C'$ is also a torus.}) holomorphic covering and $\varphi$ defines a nontrivial element in the kernel of the Cauchy--Riemann type operator $D^t_{\pi\circ\varphi}$ over $\varphi^*\pi^* N_C$ that is the pull-back of $D^t:=\dbar+\nu_C+t\mu_C$ over $N_C$ (see \cite{Wendl:automatic}). But this contradicts ``super-rigidity'' of $C$ \cite{Gerig:taming}*{Lemma 3.20}. Thus, $\zZ^t_0$ is compact.
\end{proof}

\begin{remark}
Since a special plane $C$ has index zero and satisfies super-rigidity, it may be possible to modify $J$ in a small neighborhood of its image so that $\mu_C$ becomes 0, following the methodology of Taubes' $L$-flat approximations in \cite{Taubes:ECH=SWF1}*{Appendix} (this was suggested by Taubes). The a priori issue is that there may be other nearby $J$-holomorphic curves in $\overline X_0$ which pop into existence at some point along the smooth path of modifications of $J$, altering the moduli space of curves. If this issue can be dealt with, as in the case of the $L$-flat approximations, then $\zZ_0$ is a single (regular) point and there is no need for the Kuranishi structures along special planes that appear in Section~\ref{Kuranishi structures}.
\end{remark}

%%%%%%%%%%%%%%%%%%%%%%%%%%%%%%%%%%%%%%%%%%%%%%%%%%%%%%
\subsection{Kuranishi structures}
\label{Kuranishi structures}

\indent\indent
In this section we give a proof of the following.

\begin{prop}
\label{prop:Psi}
Let $\Theta,\mM,\M$ be as in Notation~\ref{notation}. Given the hypotheses of Theorem~\ref{thm:moduli}, there is a multi-valued injective map $\Psi_r:\mM\to\M$. It is described in Definition~\ref{def:multivalued} and is only multi-valued at those currents in $\mM$ which contain a special curve of multiplicity greater than one.
\end{prop}

\begin{proof}
We closely follow \cite{Taubes:ECH=SWF2}*{\S5--\S7} and the analogous \cite{Taubes:GrtoSW}*{\S4--\S6}, while noting that the point/loop constraints cause no changes to the arguments (see \cite{Taubes:ECH=SWF5}*{\S4.e}).

The first main step in \cite{Taubes:ECH=SWF2}*{\S5} is to construct specific configurations $(A^*,\Psi^*)\in\text{Conn}(E)\times\Gamma(\SS_+)$ for each $\cC=\lbrace(C_k,d_k)\rbrace_k\in\mM$ (and $r\in\RR$), such that away from $\bigcup_k C_k\subset\overline X_0$ the curvature of $A^*$ is flat and the $E$-component of $\Psi^*$ is covariantly constant. The construction involves patching together local configurations on open sets in $\overline X_0$. The difference between compact and noncompact $\overline X_0$ is that, in the noncompact case, there need not be an embedding of a fixed-radius tubular neighborhood of the punctured curves. This complication occurs because the curves may have multiple ends approaching multiple covers of the same Reeb orbit, or an end approaching a multiple cover of a Reeb orbit (see Appendix~\ref{Appendix}). But by \cite{Gerig:taming}*{Proposition 3.16} this complication does not occur for the special planes -- this is important to note because we may now treat special planes and special tori on equal footing.

Fix a $\cC=\lbrace(C_k,d_k)\rbrace_k\in\mM$ and write $\Theta=\lbrace(\Theta_i,m_i)\rbrace_i$. The cover of $\overline X_0$ is prescribed by the bullet points of \cite{Taubes:ECH=SWF2}*{Equation 5-3} but augmented by a simplified version \cite{Taubes:ECH=SWF2}*{\S5.a.1 Step 1} over the special curves, and described as follows (a schematic example is depicted in Figure~\ref{fig:cover}):
\begin{itemize}
  \setlength{\itemsep}{1pt}
  \setlength{\parskip}{0pt}
  \setlength{\parsep}{0pt}
  
  \item There is an open set $U_{\Theta_i}$ for each orbit $\Theta_i$ that is not approached by a special plane, such that $U_{\Theta_i}$ lives far out on the ends of $\overline X_0$ and contained in a fixed-radius tubular neighborhood of $\RR\times\Theta_i$ in that region. These sets are pairwise disjoint.
  \item There is an open set $U_{C_k}$ for each punctured component $C_k$ that is not a special plane, such that $U_{C_k}$ lives in a compact region of $\overline X_0$ and contained in a fixed-radius tubular neighborhood of $C_k$ in that region -- this set does not contain the entire $C_k$. These sets are pairwise disjoint and only intersect $U_{\Theta_i}$ when $\Theta_i$ is approached by $C_k$.
  \item There is an open set $U_{C_k}$ for each closed component and special plane $C_k$, given by a fixed-radius tubular neighborhood of $C_k$. These sets are pairwise disjoint and disjoint from all preceding open sets, which is possible because each special plane is asymptotic to an elliptic orbit with multiplicity one.
  \item There is an open set $U_0$ to cover the remainder of $\overline X_0$.
\end{itemize}

\begin{figure}
    \centering
    \includegraphics[width=5cm]{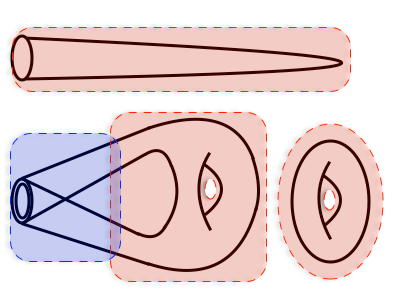}
    \caption{$C_k$ in black, $U_{C_k}$ in red, $U_{\Theta_i}$ in blue ($U_0$ not shown)}
    \label{fig:cover}
\end{figure}

\begin{notation}
Let $\Theta'$ denote the subset of $\Theta$ consisting of elliptic orbits with multiplicity one that are approached by a special plane, and denote the complement $\Theta''=\Theta-\Theta'$. Let $\cC'$ denote the subset of $\cC$ consisting of special curves, and denote the complement $\cC''=\cC-\cC'$. Use $k'$ to index the components of $\cC'$, and use $k''$ to index the (remaining) components of $\cC''$.
\end{notation}

As in \cite{Taubes:ECH=SWF2}*{\S5.a} and \cite{Taubes:GrtoSW}*{\S2.d}, the configuration $(A^*,\Psi^*)$ is defined by patching up the configurations respectively on $U_{C_k}$ for each $(C_k,d_k)\in\cC$, on $U_{\Theta_i}$ for each $(\Theta_i,m_i)\in\Theta$, and on $U_0$, defined in the following manner:
\begin{itemize}
  \setlength{\itemsep}{1pt}
  \setlength{\parskip}{0pt}
  \setlength{\parsep}{0pt}
   
  \item Over $U_0$, we trivialize $E|_{U_0}=U_0\times\CC$ and set $A^*$ to be the trivial (flat) connection and $\Psi^*=(1,0)\in\underline\CC\oplus K^{-1}|_{U_0}$.

  \item Over $U_C$ for a component $(C,1)\in\cC''$, the pair $(A^*,\Psi^*)$ is constructed from $\c=0\in L^2_1(\C_{N_C,1})$ via the equations \cite{Taubes:ECH=SWF2}*{Equation 5-2}.

  \item Over $U_{\Theta_i}$ for an orbit $(\Theta_i,m_i)\in\Theta''$, after identifying $U_{\Theta_i}$ with $(-\infty,-R)\times S^1\times\CC$ for some $R\gg1$ via \cite{Taubes:ECH=SWF2}*{Equation 4-1}\footnote{The $r$-independent lower bound on $R$ depends on $\cC$ \cite{Taubes:ECH=SWF2}*{\S4.c}.}, the pair $(A^*,\Psi^*)$ is defined by the equations \cite{Taubes:ECH=SWF2}*{Equation 5-8, Equation 5-9} and constructed from the aforementioned configurations on $U_0\cup \bigcup_{(C,1)\in\cC''}U_C$ (with cutoff functions) and also from $\bigcup_kC_k$ by viewing the latter as the graph of a map $(-\infty,-R)\times S^1\to\text{Sym}^{m_i}(\CC)\approx\C_{m_i}$ which sends $(s,t)$ to the set of $m_i$ intersections of $\bigcup_kC_k$ with $(\lbrace(s,t)\rbrace\times\CC)$ \cite{Taubes:ECH=SWF2}*{\S5.c}.

  \item  Over $U_C$ for a component $(C,d)\in\cC'$, after identifying $U_C$ with the unit disk subbundle of $\pi:N_C\to C$ using the exponential map, we set $E|_{U_C}=\pi^*N_C^d$. The pair $(A^*,\Psi^*)$ is constructed from $\c=0\in L^2_1(\C_{N_C,d})$ via the analog of equations \cite{Taubes:ECH=SWF2}*{Equation 5-2} in which the section $\s/|\s|$ (notation of \cite{Taubes:ECH=SWF2}*{\S5.a}) and connection $\theta_\s$ (defined by \cite{Taubes:ECH=SWF2}*{Equation 5-1}) over $\pi^*N_C$ are replaced by the induced $\s^d/|\s|^d$ and $d\theta_\s$ over $\pi^*N_C^d$.\footnote{This model also appears in \cite{Taubes:ECH=SWF2}*{\S5.a.2 Step 2} when $C$ is an $\RR$-invariant cylinder in a symplectization.}
\end{itemize}

An examination of Taubes' arguments for constructing SW instantons, in particular \cite{Taubes:ECH=SWF2}*{\S7.e.1} and \cite{Taubes:GrtoSW}*{\S3.b}, leads to the conclusion that the constructions in \cite{Taubes:ECH=SWF2}*{\S5} of $(A^*,\Psi^*)$ did not have to require the section $\bc=(\c_k)_k\in\bigtimes_kL^2_1(\C_{N_{C_k},d_k})$ to be the zero element nor to live in $\bigtimes_k\zZ^{(k)}_0$, where $\zZ^{(k)}_0$ denotes the $(C_k,d_k)$ version of~\eqref{def:csurface}. Instead, it can and will be assumed that $\bc$ belongs to the space
$$\yY:=\big(\bigtimes_{k'}\kK_\Lambda^{(k')}\big)\times\big(\bigtimes_{k''}\big\lbrace0\big\rbrace\big)$$
where $\kK_\Lambda^{(k')}$ is a suitably small open neighborhood of $\zZ^{(k')}_0$ in $L^2_1(\bigoplus^{d_{k'}}_{j=1}N_{C_{k'}}^j)$; this neighborhood will be specified in the upcoming Lemma~\ref{lem:5.1Gr} (similar to what is denoted by the same notation in \cite{Taubes:GrtoSW}). To clarify, recall from Remark~\ref{Zregular2} that all $\zZ^{(k)}_0$ are regular in the setting of \cite{Taubes:ECH=SWF2}. In our context, the versions of $\zZ_0$ associated to special curves are no longer regular, and the arguments in \cite{Taubes:ECH=SWF2} need to be augmented with the more general gluing arguments from \cite{Taubes:GrtoSW}*{Proposition 5.2}. The version of $(A^*,\Psi^*)$ constructed from $\bc\in\yY$ will be denoted henceforth by $(A_{\bc}^*,\Psi_{\bc}^*)$.

As in \cite{Taubes:ECH=SWF2}*{\S5.b, \S5.d}, we continue the search for a solution to the large $r$ version of Taubes' perturbed Seiberg--Witten equations~\eqref{SW4} by considering a family of deformations of $(A_{\bc}^*,\Psi_{\bc}^*)$, i.e. particular families\footnote{The corresponding notation in \cite{Taubes:ECH=SWF2} is such that $\big(A_0^\xi(\b),\Psi_0^\xi(\b)\big)$ equals \cite{Taubes:ECH=SWF2}*{Equation 5-19} with $\big(A_0^\xi(0),\Psi_0^\xi(0)\big)$ denoted by $(A^\xi,\Psi^\xi)$.} $\big(A_{\bc}^\xi(\b),\Psi_{\bc}^\xi(\b)\big)$ of configurations parametrized by
\begin{align*}
\bc\in &\;\yY\\
\xi\in &\;\kK_{\bc}\cap\kK_{\bc*}\subset\kK_{\bc}:=\big(\bigtimes_{k'}\kK_{\c_{k'}}\big)\times\kK''\\
\b\in &\;\HH(iT^*\overline X_0\oplus\SS_+)
\end{align*}
Here, each $\kK_{\c_{k'}}$ is the space mentioned at the end of Section~\ref{Vortices and orbits and curves}, while $\kK''$ is a Banach space described in the same manner as that in \cite{Taubes:ECH=SWF2}*{\S5.b Step 4} by completing a product space (using the norm defined by \cite{Taubes:ECH=SWF2}*{Equation 5-13}) whose factors are indexed by the elements $(C_{k''},d_{k''})\in\C''$ and $(\Theta_i,m_i)\in\Theta''$ and characterized by \cite{Taubes:ECH=SWF2}*{Equation 5-12}, and $\kK_{\bc*}$ is a Banach space similar to $\kK_{\bc}$ using slightly different norms (see \cite{Taubes:ECH=SWF2}*{\S5.b Step 4}). Finally, $\HH$ is the completion of the space of compactly supported sections of $iT^*\overline X_0\oplus\SS_+$ using the norm defined by \cite{Taubes:ECH=SWF2}*{Equation 6-2}. Such an element $\xi\in\kK_{\bc}\cap\kK_{\bc*}$ is used to deform the model configurations $(A_{\bc}^*,\Psi_{\bc}^*)$ on the respective open sets $U_{C_k}$ and $U_{\Theta_i}$ and patching them together: This deformation $\big(A_{\bc}^\xi(0),\Psi_{\bc}^\xi(0)\big)$ is constructed in exactly the same manner as \cite{Taubes:ECH=SWF2}*{Equation 5-15} except over the open sets $U_C$ for $(C,d)\in\cC'$ where it is constructed by the analog of \cite{Taubes:ECH=SWF2}*{Equation 5-11} in which $\pi^\ast N_C^d$ replaces the bundle $E$ there. Then an element $\b\in\HH$ defines a global perturbation $\big(A_{\bc}^\xi(\b),\Psi_{\bc}^\xi(\b)\big)$ of $\big(A_{\bc}^\xi(0),\Psi_{\bc}^\xi(0)\big)$ by \cite{Taubes:ECH=SWF2}*{Equation 5-19}.

Such a configuration $\big(A_{\bc}^\xi(\b),\Psi_{\bc}^\xi(\b)\big)$ solves the large $r$ version of Taubes' perturbed Seiberg--Witten equations when \cite{Taubes:ECH=SWF2}*{Equation 5-20} (equivalently, \cite{Taubes:ECH=SWF2}*{Equation 6-1}) is satisfied, written schematically as
\begin{equation}
\label{eqn:schematic}
\begin{split}
\D\b+r^{1/2}\b*\b-\fv&=0\\
\lim_{s\to-\infty}\b=\b_\Theta&
\end{split}
\end{equation}
where $\b_\Theta$ is a solution to \cite{Taubes:ECH=SWF2}*{Equation 3-5}) determined by $\c_\Theta$. Here, $\b\mapsto\b*\b$ denotes an $r$-independent quadratic fiber-preserving map from $iT^*\overline X_0\oplus\SS_+$ to $i\wedge^2_+T^*\overline X_0\oplus\SS_-$, $\D$ is the deformation operator associated with $\big(A_{\bc}^\xi(0),\Psi_{\bc}^\xi(0)\big)$, and $\fv$ is the remainder (it is an error term determined by the failure of $\big(A_{\bc}^\xi(0),\Psi_{\bc}^\xi(0)\big)$ to solve the Seiberg--Witten equations). The idea now is Lyapunov--Schmidt reduction, i.e. to first project~\eqref{eqn:schematic} onto a certain subspace, solve for $\b$ in terms of $\bc$ and $\xi$ (and $r$), and then use the remaining part of~\eqref{eqn:schematic} to solve for $\xi$ in terms of $\bc$.

For each $\bc\in\yY$ and $\xi\in\kK_{\bc}\cap\kK_{\bc*}$, \cite{Taubes:ECH=SWF2}*{\S6.a} introduces a map $\t_\xi:\kK_{\bc}^2\to L^2(iT^*\overline X_0\oplus\SS_+)$ (see \cite{Taubes:ECH=SWF2}*{Equation 6-9}), where $\kK_{\bc}^2$ is the version of $\kK_{\bc}$ using the $L^2$-norm. Let $\Pi_\xi$ denote the $L^2$ orthogonal projection onto the image of $\t_\xi$.  The result of \cite{Taubes:ECH=SWF2}*{\S6} is a solution to 
\begin{equation}
\label{eqn:schematicProj1}
(1-\Pi_\xi)\left(\D\b+r^{1/2}\b*\b\right)=(1-\Pi_\xi)\fv
\end{equation}
Specifically, \cite{Taubes:ECH=SWF2}*{Proposition 6.4} solves for $\b$ as a smooth function of $\lbrace\bc,\xi,r\rbrace$ (using the contraction mapping principle indirectly), given appropriate bounds on $\xi$ and further assuming bounds on a certain term
$$(1-\Pi_\xi)(\fv-\fv_\h)$$
that appears when~\eqref{eqn:schematicProj1} is rewritten as in \cite{Taubes:ECH=SWF2}*{\S6.c} ($\fv_\h$ is given by \cite{Taubes:ECH=SWF2}*{Equation 6-21}). It is shown in \cite{Taubes:ECH=SWF2}*{Lemma 6.3} that this assumed bound on $(1-\Pi_\xi)(\fv-\fv_\h)$ is guaranteed when $\c\in\bigtimes_k\zZ^{(k)}_0$, but in general we have to check whether this bound is actually satisfied.

It then remains to solve
\begin{equation}
\label{eqn:schematicProj2}
\Pi_\xi\left(\D\b+r^{1/2}\b*\b-\fv\right)=0
\end{equation}
and also to find those $\bc\in\yY$ for which the term $(1-\Pi_\xi)(\fv-\fv_\h)$ is suitably bounded.
When $\bc\notin\bigtimes_k\zZ^{(k)}_0$, the constructions in \cite{Taubes:ECH=SWF2}*{\S7} need to be augmented by the constructions in \cite{Taubes:GrtoSW}. Indeed, it is also shown in \cite{Taubes:ECH=SWF2}*{\S7.d} that the related term
$$\Pi_\xi(\fv-\fv_\h)$$
is large when $\c_k$ is far away from $\zZ^{(k)}_0$, so that~\eqref{eqn:schematicProj2} need not be solvable. What follows are the modifications to \cite{Taubes:ECH=SWF2}*{\S7} that come directly from \cite{Taubes:GrtoSW}*{\S5}. Again, it is important to note the special curves $C_{k'}$ are disjoint from the other components of $\cC$ and only approach embedded Reeb orbits which are furthermore not approached by the other components of $\cC$, so the analysis in \cite{Taubes:GrtoSW} complements (and does not interfere with) the analysis in \cite{Taubes:ECH=SWF2}.\footnote{Since the arguments in \cite{Taubes:GrtoSW} and \cite{Taubes:ECH=SWF2} follow the same principle of obstructed gluing theory, here are some similarities/differences between the notions and equations in \cite{Taubes:GrtoSW}*{\S2--\S5} and our setup (which follows \cite{Taubes:ECH=SWF2}*{\S5--\S7}). Ignoring contributions from the ends of the curves, our variables $\lbrace\bc,\xi,\b,\fv-\fv_\h\rbrace$ may be viewed as analogs of $\lbrace y,x,q',\operatorname{err}\rbrace$ in \cite{Taubes:GrtoSW}, though the pairs $(\bc,\xi)$ and $(y,x)$ play different roles in the respective gluing arguments. The respective points of expansion of a tentative Seiberg--Witten solution are $(A_{\bc}^\xi(0),\Psi_{\bc}^\xi(0))$ in our setup and the version of $(\underline a_r,(\underline\alpha_r,\underline\beta_r))$ in \cite{Taubes:GrtoSW}*{\S3.b} constructed from $y$ (and from $x$ as in \cite{Taubes:GrtoSW}*{Equation 5.10}). This configuration $(\underline a_r,(\underline\alpha_r,\underline\beta_r))$ is a modification to the configuration $(a_r,(\alpha_r,0))$ in \cite{Taubes:GrtoSW}*{\S2.d} (whose analog is our $(A_{\bc}^*,\Psi_{\bc}^*)$) using the ``correction'' term $u_r$ defined by \cite{Taubes:GrtoSW}*{Equation 3.13} (whose analog is part of our $\b$). The Seiberg--Witten equations are then given by Equation~\eqref{eqn:schematic} with linear variable $\b$ and by \cite{Taubes:GrtoSW}*{Equation 4.2} with linear variable $q'$. The argument of \cite{Taubes:ECH=SWF2}*{\S 7.d} concerning $\Pi_\xi(\fv-\fv_\h)$ is analogous to that of \cite{Taubes:GrtoSW}*{\S 5.f} concerning $\Pi\operatorname{err}$.

While we're at it, here is a parallel between the above constructions and the analogous constructions for Reeb orbits, i.e. the construction of $\Theta\mapsto\c_\Theta$ (Theorem~\ref{thm:generators}). Our equations~\eqref{eqn:schematic},\eqref{eqn:schematicProj1},\eqref{eqn:schematicProj2} are the analogs of \cite{Taubes:ECH=SWF2}*{Equation 3-6, Equation 3-16, Equation 3-35}. Solving these equations amounts to a proof of Theorem~\ref{thm:generators} for the following reason: A key lemma \cite{Taubes:ECH=SWF2}*{Lemma 3.8} guarantees a unique solution to \cite{Taubes:ECH=SWF2}*{Equation 3-35} whenever $\c\in\zZ_0$ with $\coker\Delta_\c=0$, while Theorem~\ref{thm:corbitL} guarantees that $\zZ_0=\lbrace0\rbrace$ with $\coker\Delta_0=0$; so there is a unique monopole $\c_\Theta$.}

As in \cite{Taubes:ECH=SWF2}*{\S7}, we view the left hand side of~\eqref{eqn:schematicProj2} as a bundle map over $\yY$ given fiberwise by an operator
$$\kK_{\bc}\to\lL_{\bc},\indent\xi\mapsto r^{-1/2}\tT(\bc,\xi)$$
Here, the Banach space $\lL_{\bc}$ is defined similarly to that in \cite{Taubes:ECH=SWF2}*{\S6.a.7} -- it is the analog of our $\kK_{\bc}$ in which all relevant bundles involved in its definition are tensored with the components $T^{0,1}C_k$. We then consider the Taylor expansion
$$\tT(\bc,\xi)=\tT_0(\bc)+\tT_1(\bc)\cdot\xi+\tT_2(\bc,\xi)$$
where $\tT_0(\bc):=\tT(\bc,0)$ and $\tT_1(\bc)$ is linear and $\tT_2(\bc,\cdot)$ is the remainder. In the absence of special curves, the zeros of $\tT(\bc,\xi)$ are described by \cite{Taubes:ECH=SWF2}*{Proposition 7.1}. The existence of zeros hinges on appropriate bounds on $\lbrace \tT_j\rbrace$ and requires $\tT_1(\bc)$ to satisfy additional properties; they are granted as follows:
\begin{itemize}
  \setlength{\itemsep}{1pt}
  \setlength{\parskip}{0pt}
  \setlength{\parsep}{0pt}
  
  \item[(i)] the bounds on $\tT_2$ are established in \cite{Taubes:ECH=SWF2}*{\S 7.e.3} and granted by the small norms on $\kK_{\bc}$, 

  \item[(ii)] the bounds on $\tT_0(\bc)$ are established in \cite{Taubes:ECH=SWF2}*{\S7.e.1} and granted by the constraint $\bc\in\bigtimes_k\zZ^{(k)}_0$,

  \item[(iii)] the properties of $\tT_1(\bc)$ are ultimately granted by the constraint $\bigoplus_k\coker\Delta_{\c_k}=0$, which in turn is granted by the fact that $\bigoplus_k\coker D_{C_k}=0$ for generic $J$ (see Section~\ref{Vortices and orbits and curves}).
\end{itemize}
However, for a special curve $C_{k'}$ it can be the case that $\coker\Delta_{\c_{k'}}\ne0$, so item (iii) no longer holds and the proof of \cite{Taubes:ECH=SWF2}*{Proposition 7.1} needs to be modified. The resolution is to relax the constraint $\c_{k'}\in\zZ^{(k')}_0$ (so item (ii) also needs to be modified) and appeal to Kuranishi structures as in the proofs of \cite{Taubes:GrtoSW}*{Lemma 5.1, Proposition 5.2}. We do this now.

Using the fact that $\zZ^{(k')}_0$ is compact (Proposition~\ref{prop:Zcompact}) we have the following analog of \cite{Taubes:GrtoSW}*{Lemma 5.1} whose proof is contained in \cite{Taubes:GrtoSW}*{\S5.a, \S5.g.2}.\footnote{The submanifold $\kK_\Lambda$ in \cite{Taubes:GrtoSW}*{Lemma 5.1} has dimension $d+\dim(\Lambda)$, where $d=\ind_\RR\Delta$. In our case, $\ind_\RR\Delta=\ind_\RR D_C=0$. For the reader's benefit, we point out that the integer $d$ is misstated in \cite{Taubes:GrtoSW}*{Lemma 5.1} though correctly stated as $d=2m(1-g)+m(m+1)n=2m(n+1-g)+m(m-1)n$ in \cite{Taubes:GrtoSW}*{Proposition 3.2}.}

\begin{lemma}
\label{lem:5.1Gr}
Let $(C,d)$ be a special curve with multiplicity. There exists a finite dimensional vector subspace $\Lambda\subset L^2_1(\bigoplus_{j=1}^d N^j_C\otimes T^{0,1}C)$ such that for all $\c\in\zZ_0$, the projection of $\Lambda$ onto $\coker\Delta_\c$ is surjective. Denote by $Q_\Lambda:L^2_1(\bigoplus_{j=1}^d N^j_C\otimes T^{0,1}C)\to\Lambda$ the $L^2$ orthogonal projection. There also exists a smooth $\dim(\Lambda)$-dimensional submanifold $\kK_\Lambda\subset L^2_1(\bigoplus_{j=1}^d N^j_C)$ with compact closure, satisfying the following properties:

1) If $\c\in\kK_\Lambda$ then
$$(1-Q_\Lambda)\left(\dbar\c+\nu_C\aleph(\c)+\mu_C\FF(\c)\right)=0$$

2) $\zZ_0$ embeds in $\kK_\Lambda$ as the zero set of the map
$$\psi_\Lambda:\kK_\Lambda\to\Lambda,\indent \c\mapsto Q_\Lambda\left(\dbar\c+\nu_C\aleph(\c)+\mu_C\FF(\c)\right)$$

3) For a suitably chosen $r$-independent constant $\e>0$, for each $\c\in\kK_\Lambda$ there is some element in $\zZ_0$ that is within $\e$ distance (with respect to the $L^2$ norm) to $\c$, and $(1-Q_\Lambda)\Delta_\c$ is surjective.
\end{lemma}

We now show how this lemma is used to modify the arguments in \cite{Taubes:ECH=SWF2}*{\S7}, along the lines of the analogous arguments that prove \cite{Taubes:GrtoSW}*{Proposition 5.2}. Still following \cite{Taubes:ECH=SWF2}*{\S7.b.5}, we write the components of $\tT$ as $((\tT_{C_k})_{C_k\in\cC},(\tT_{\Theta_j})_{\Theta_j\in\Theta''})$. In terms of the Taylor expansion, each component $\tT_{C_{k}}(\c_{k},\xi_{k})$ equals the sum of the $0^\text{th}$ order term
$$\tT_{0C_{k}}(\c_{k},\xi_{k})=\dbar\c_{k}+\nu_{C_{k}}\aleph(\c_{k})+\mu_{C_{k}}\FF(\c_{k})+\text{err}$$
and the linear term
$$\tT_{1C_{k}}(\c_{k},\xi_{k})=\Delta_{\c_{k}}\xi_{k}+\text{err}$$
and the remainder $\tT_{2C_{k}}(\c_{k},\xi_{k})$.\footnote{The ``err'' term in $\tT_{0C_{k}}$ is the sum of two parts: $r^{1/2}\Pi_\xi(\D\b+r^{1/2}\b*\b-\fv)-r^{1/2}\Pi_\xi(\fv-\fv_\h)$ evaluated at $\xi=0$ (which is bounded similarly to that done in \cite{Taubes:ECH=SWF2}*{\S7.b.2, \S7.b.3}), and $r^{1/2}\Pi_\xi(\fv-\fv_\h)-(\dbar\c_{k}+\nu_{C_{k}}\aleph(\c_{k})+\mu_{C_{k}}\FF(\c_{k}))$ evaluated at $\xi=0$ (which is bounded similarly to that done in \cite{Taubes:ECH=SWF2}*{\S7.d.1, \S7.d.2}). The ``err'' term in $\tT_{1C_{k}}$ is the difference between $\Delta_{\c_{k}}\xi_{k}$ and a directional derivative of $r^{1/2}\Pi_\xi(\fv-\fv_\h)$ at $\xi=0$ (similar to \cite{Taubes:ECH=SWF2}*{\S7.e.2}).} For special curves, the linear term $\tT_{1C_{k'}}$ may have nonzero (co)kernel. To deal with this, we first apply the operator $1-Q_\Lambda^{(k')}$ to the $\tT_{C_{k'}}$ components, and we denote the resulting operator by
$$\tT_Q:=\left(\left((1-Q_\Lambda^{(k')})\tT_{C_{k'}}\right)_{C_{k'}\in\cC'},(\tT_{C_{k''}})_{C_{k''}\in\cC''},(\tT_{\Theta_j})_{\Theta_j\in\Theta''}\right)$$
We then restrict the fiberwise domain of $\tT_Q$ to $(\bigtimes_{k'}\lL_{\c_{k'}}^\perp)\times\kK''$, where $\lL_{\c_{k'}}^\perp$ denotes the $L^2$ orthogonal complement of
$$\lL_{\c_{k'}}:=\ker(1-Q^{(k')}_\Lambda)\Delta_{\c_{k'}}\subset L^2_1(\bigoplus_{j=1}^{d_{k'}}N_{C_{k'}}^j)$$
We now rerun \cite{Taubes:ECH=SWF2}*{\S7}, with $\tT_Q$ replacing $\tT$, to solve the projected equation $\tT_Q(\bc,\xi)=0$. In this regard, the aforementioned item (i) is unchanged, item (ii) is now granted by the relaxed constraint $\bc\in\yY$ thanks to Lemma~\ref{lem:5.1Gr}(1), and item (iii) is now satisfied thanks to Lemma~\ref{lem:5.1Gr}(3).

The corresponding result of \cite{Taubes:ECH=SWF2}*{Proposition 7.1, Proposition 7.6}\footnote{The vector space $V_0:=\bigoplus\ker(D_C)$ is constructed more generally in \cite{Taubes:ECH=SWF2}*{Equation 7-43}, but thanks to the asymptotics of our currents $\cC$ and contact form $\lambda_\s$, there are no contributions from what is denoted $V_{\gamma\pm,k}$ there (its dimension \cite{Taubes:ECH=SWF3}*{Equation 3-44} is zero).} is an $L^\infty$-small ball
$$B\subset \bigoplus_{C\in\cC''}\ker(D_C)$$
and map $q_{\bc}:\kK_{\bc}\to \bigoplus_{C\in\cC''}\ker(D_C)$ such that the zeros of $\tT_Q(\bc,\xi)$, as a function of $(\bc,b)\in\yY\times B$, are
$$\xi(\bc,b)\in(\bigtimes_{k'}\lL_{\c_{k'}}^\perp)\times\kK''$$
given by the restriction of $q_{\bc}^{-1}(b)$ to an $L^\infty$-small ball (this ball corresponds to that specified in \cite{Taubes:ECH=SWF2}*{Proposition 6.4}). This association $(\bc,b)\mapsto\xi(\bc,b)$ induces the smooth map
$$\Psi_{\cC,r}:\yY\times B\to\bB(\c_\Theta,\varnothing;\s_\s),\indent (\bc,b)\mapsto \Big(A_{\bc}^{\xi(\bc,b)}\big(\b\big(\bc,\xi(\bc,b)\big)\big), \Psi_{\bc}^{\xi(\bc,b)}\big(\b\big(\bc,\xi(\bc,b)\big)\big)\Big)$$
and we note that the pre-image 
$$\yY_{\bar z,\bar\eta}:=\Psi_{\cC,r}^{-1}\big(\bB(\c_\Theta,\varnothing;\s_\s,\bar z,\bar\eta)\big)\subset\yY\times B$$
is diffeomorphic to $\bigtimes_{k'}\kK_\Lambda^{(k')}$ because no special curve $C_{k'}$ intersects the points $\bar z$ and loops $\bar\eta$ (see also \cite{Taubes:ECH=SWF5}*{Proposition 4.2} and \cite{Taubes:Gr=SW}*{Proposition 2.10}). It remains to find those $\bc\in\yY_{\bar z,\bar\eta}$ that will satisfy $\Psi_{\cC,r}(\bc)\in\M$. 

The evaluation of the projection $Q_\Lambda^{(k')}\tT_{C_{k'}}$ defines the smooth map\footnote{This map appears in Taubes' papers as \cite{Taubes:GrtoSW}*{Equation 5.21}.}
\begin{equation}
\label{eqn:psi}
\psi_{\cC,r}:\yY\times B\to\bigoplus_{k'}\Lambda_{k'},\indent (\bc,b)\mapsto \bigoplus_{k'}\left[\psi_\Lambda^{(k')}(\c_{k'})+Q_\Lambda^{(k')}\left(\Delta_{\c_{k'}}\xi(\c,b)_{k'}+\rR_{C_{k'}}\left(\c_{k'},\xi(\c,b)_{k'}\right)\right)\right]
\end{equation}
where $\rR_{C_{k'}}$ denotes the sum of $\tT_{2C_{k'}}$ and the ``err'' terms, and we note that by construction,
$$\psi_{\cC,r}^{-1}(0)\cap\yY_{\bar z,\bar\eta}=\Psi_{\cC,r}^{-1}(\M)$$
As explained in \cite{Taubes:GrtoSW}*{\S6} and \cite{Taubes:ECH=SWF5}*{\S4} (see also \cite{Taubes:Gr=SW}*{Property 2.22}), $\Psi_{\cC,r}$ is a smooth embedding for sufficiently large $r$, it maps $\psi_{\cC,r}^{-1}(0)\cap\yY_{\bar z,\bar\eta}$ homeomorphically onto a subset of $\M$, and the images of $\lbrace\Psi_{\cC,r}\rbrace_{\cC\in\mM}$ in $\bB(\c_\Theta,\varnothing;\s_\s)$ are disjoint. (In this regard, each image $\im\Psi_{\cC,r}$ constitutes a ``Kuranishi model'' for a subset of $\M$.) We can now state the anticipated correspondence between $\mM$ and $\M$, completing the proof of Proposition~\ref{prop:Psi}.

\begin{definition}
\label{def:multivalued}
The multi-valued injective map $\Psi_r:\mM\to\M$ is given by the composition
$$\cC\mapsto\psi_{\cC,r}^{-1}(0)\cap\yY_{\bar z,\bar\eta}\mapsto\Psi_{\cC,r}\left(\psi_{\cC,r}^{-1}(0)\cap\yY_{\bar z,\bar\eta}\right)$$
where the first map is multi-valued and the second map is injective.
\end{definition}
\end{proof}

%%%%%%%%%%%%%%%%%%%%%%%%%%%%%%%%%%%%%%%%%%%%%%%%%%%%%%
\subsection{Surjectivity of $\Psi_r$}
\label{Surjectivity}

\indent\indent
In this section we give a proof of the following.

\begin{prop}
\label{prop:PsiSurjective}
Given the hypotheses of Theorem~\ref{thm:moduli}, Proposition~\ref{prop:Psi}'s map $\Psi_r$ is surjective.
\end{prop}

Its proof uses the upcoming Lemma~\ref{lem:estimates} numerous times along with other estimates established in \cite{Taubes:SWtoGr} on closed symplectic manifolds and in \cite{Taubes:ECH=SWF4} on symplectizations. It is explained in \cite{HutchingsTaubes:Arnold2}*{\S7} how the estimates from \cite{Taubes:ECH=SWF4} on a symplectization carry over, with minor modifications, to \textit{exact} symplectic cobordisms. For everything to carry over to \textit{strong} (non-exact) symplectic cobordisms, such as $(X_0,\omega)$, we sometimes must make an additional modification. We will elaborate on these modifications when necessary.

\begin{lemma}
\label{lem:estimates}
There exists $\kappa\ge1$ such that if $r\ge\kappa$ and if $(A,\Psi=(\alpha,\beta))$ solves Taubes' perturbed Seiberg--Witten equations~\eqref{SW4} on $\overline X_0$ then
$$|\alpha|\le1+\kappa r^{-1}$$
$$|\beta|^2\le\kappa r^{-1}(1-|\alpha|^2)+\kappa^2r^{-2}$$
\end{lemma}

\begin{proof}
This was already stated in \cite{HutchingsTaubes:Arnold2}*{Lemma 7.3} for exact symplectic cobordisms, and there are no changes in general. In fact, the proof follows the arguments in \cite{Taubes:SWtoGr}*{Proposition 2.1, Proposition 2.3} on the compact region $X_0$ of $\overline X_0$. Those arguments extend over the ends $(-\infty,0]\times\partial X_0$ of $\overline X_0$, as in the case of symplectizations \cite{Taubes:ECH=SWF4}*{Lemma 3.1}, because the analogous bounds on the monopoles over $\partial X_0$ are granted by \cite{Taubes:ECH=SWF4}*{Lemma 2.3}. We now outline the proof.

From $D_\bA\Psi=0$ we know that $D_\bA^* D_\bA\Psi=0$. We rewrite this equation using the Bochner-Weitzenb\"ock formula,
\begin{equation}
\label{Weitzenbock}
\nabla_\bA^*\nabla_\bA\Psi+\frac{R_g}{4}\Psi+\frac{1}{2}\cl_+(F_\bA^+)\Psi=0
\end{equation}
where $R_g$ denotes the Ricci scalar curvature of $g$ on $\overline X_0$. Introducing the components $\Psi=(\alpha,\beta)$,
$$\alpha=\frac{1}{2}\left(1+\frac{i}{2}\cl_+(\widehat\omega)\right)\Psi,\indent \beta=\frac{1}{2}\left(1-\frac{i}{2}\cl_+(\widehat\omega)\right)\Psi$$
we take the inner products of~\eqref{Weitzenbock} with $\Psi$ and $(\alpha,0)$ and $(0,\beta)$ separately. We then apply the maximum principle (using \cite{Taubes:ECH=SWF4}*{Lemma 2.3} as $s\to-\infty$ on the ends of $\overline X_0$) to each of the resulting equations. Then we combine all of the resulting inequalities to get the asserted bounds.
\end{proof}

\begin{proof}[Proof of Proposition~\ref{prop:PsiSurjective}]
With appropriate modifications coming from \cite{HutchingsTaubes:Arnold2,Hutchings:fieldtheory}, we will follow the proof of \cite{Taubes:ECH=SWF4}*{Theorem 1.2} and the proof of \cite{Taubes:Gr=SW}*{Proposition 2.10}. We note again that the point/loop constraints do not affect the arguments (see also \cite{Taubes:ECH=SWF5}*{\S4.e}), so no further comments will be made about them.

Assume first that special curves do not exist. The proof of \cite{Taubes:ECH=SWF4}*{Theorem 1.2} can then be copied to show that $\Psi_r$ is surjective, albeit with the following modifications that are also explained in \cite{HutchingsTaubes:Arnold2}*{\S7} and \cite{Hutchings:fieldtheory}.\footnote{The reader of \cite{Taubes:ECH=SWF4,HutchingsTaubes:Arnold2} will notice the appearance of the ``Seiberg--Witten action,'' denoted by $\a$ in \cite{Taubes:ECH=SWF4}*{Equation 3-2} and \cite{HutchingsTaubes:Arnold2}*{Equation 97}. In the context of our paper, it is a gauge-invariant functional $\a:\bB(S^1\times S^2,\s_{\xi_0}+1)\to\RR$ because $\s_{\xi_0}+1$ is a torsion spin-c structure. It differs from the Chern--Simons--Dirac functional $\lL_\text{CSD}$ by an $O(r)$ constant, and its set of critical points is precisely the moduli space of monopoles $\M(S^1\times S^2,\s_{\xi_0}+1)$.} We must replace the manifold $\RR\times M$ by $\overline X_0$ and introduce a piecewise-smooth function $s_*:\overline X_0\to(-\infty,0]$ that agrees with the $(-\infty,0]$ coordinate on each cylindrical end of $\overline X_0$ and equals 0 on $X_0$. Then we replace the subsets $[s_1,s_2]\times M$ by $s_*^{-1}([s_1,s_2])$, the 2-form $ds\wedge a+\frac{1}{2}* a$ by $\widehat \omega$, the 2-form $\frac{\partial}{\partial s}A\pm B_A$ by $F^\pm_A$ (for a given SW instanton $\d=[A,(\alpha,\beta)]\in\M$), and the spectral flow $f_\d$ by $\ind(\D_\d)$. Most importantly, we must replace the first assertion of the key lemma \cite{Taubes:ECH=SWF4}*{Lemma 6.1} with the following key lemma, which asserts that SW instantons on a symplectic cobordism give rise to pseudoholomorphic curves.

\begin{lemma}
\label{lem:SWtoGr}
Fix a nondegenerate monopole $\c\in\M(-\partial X_0,\s_{\xi_0}+1)$. Let $[A,(\alpha,\beta)]$ be an element of $\M_I(\c,\varnothing;\s_\s,\bar z,\bar\eta)$ for $r$ sufficiently large. Then
$$E(\c)<2\pi\rho_\s$$
and there exists a current $\cC\in\mM_I(\varnothing,\Theta;\tau_\omega(\s),\bar z,\bar\eta)$ asymptotic to the admissible orbit set $\Theta$ determined by $\c$ in Theorem~\ref{thm:generators}. Furthermore, given a real number $\delta>0$ there exists a real number $r_\delta\ge1$ such that
\begin{equation}
\label{eqn:distance}
\sup_{z\in\cC}\dist(z,\alpha^{-1}(0))+\sup_{z\in \alpha^{-1}(0)}\dist(\cC,z)<\delta
\end{equation}
whenever $r\ge r_\delta$.
\end{lemma}

We postpone the proof of this lemma until the end of this section. The remaining modifications to the proof of \cite{Taubes:ECH=SWF4}*{Theorem 1.2} are as follows. As explained in \cite{HutchingsTaubes:Arnold2}*{\S7}, the arguments in \cite{Taubes:ECH=SWF4}*{\S3--5} over symplectizations are replaced by the corresponding arguments in \cite{HutchingsTaubes:Arnold2}*{\S7.3--7.5} over exact symplectic cobordisms. Then as explained more generally in \cite{Hutchings:fieldtheory} over strong (non-exact) symplectic cobordisms, they are further augmented over our symplectic cobordism $(X_0,\omega)$ so that:

	$\bullet$ the (broken) pseudoholomorphic curves represent the relative homology class $\tau_\omega(\s)$, and
	
	$\bullet$ the energy bounds on the SW monopoles over $\partial X_0$ are increased by the amount $2\pi\rho_\s$.

\noindent
For example, these two bulleted statements are apparent in the statement of Lemma~\ref{lem:SWtoGr}.

\bigskip
Assume now that special curves exist, and note that Lemma~\ref{lem:SWtoGr} does not depend on the existence of special curves or lack thereof. To prove surjectivity of $\Psi_r$ in general it remains to further modify the proof of \cite{Taubes:ECH=SWF4}*{Theorem 1.2} to show that \textit{each element $\d\in\M$ lies in the image of $\Psi_{\cC,r}$ for the element $\cC\in\mM$ afforded by Lemma~\ref{lem:SWtoGr}.} This is precisely the final assertion of \cite{Taubes:Gr=SW}*{Proposition 2.10} for $\overline X_0$ closed, which as explained in \cite{Taubes:Gr=SW}*{\S5}, is a special case of \cite{Taubes:Gr=SW}*{Proposition 5.1}. We now explain these modifications in our scenario.

When following the proof of \cite{Taubes:ECH=SWF4}*{Theorem 1.2} to prove surjectivity of $\Psi_r$, we make use of a refinement of Lemma~\ref{lem:SWtoGr}: \textit{The bound $\delta$ on the Hausdorff distance in~\eqref{eqn:distance} is replaced by $\delta r^{-1/2}$.} The analogous statement is \cite{Taubes:ECH=SWF4}*{Lemma 6.2} which refines \cite{Taubes:ECH=SWF4}*{Lemma 6.1}, and its proof in \cite{Taubes:ECH=SWF4}*{\S7} extends (without further modifications than the ones already listed) to our scenario in the absence of special curves. Taubes remarks at the beginning of \cite{Taubes:ECH=SWF4}*{\S7} that the arguments for \cite{Taubes:ECH=SWF4}*{Lemma 6.2} closely follow the arguments in \cite{Taubes:Gr=SW}*{\S5} for \cite{Taubes:Gr=SW}*{Proposition 5.1}, so by explaining how the proof of \cite{Taubes:Gr=SW}*{Proposition 5.1} extends to our scenario we will have simultaneously explained why this refinement of Lemma~\ref{lem:SWtoGr} holds in the presence of special curves.

\begin{remark}
The relevant arguments from \cite{Taubes:ECH=SWF4}*{\S7} and \cite{Taubes:Gr=SW}*{\S5} involve ``special sections'' of powers of $N_{C_k}$ associated with any component $(C_k,d_k)\in\cC$, denoted by $\fo$ in \cite{Taubes:ECH=SWF4}*{Equation 7-9} and by $h$ in \cite{Taubes:Gr=SW}*{Equation 5.26}; they differ by an inconsequential factor of $\frac{1}{d_k\pi}$. In this regard, Taubes also remarks at the beginning of \cite{Taubes:ECH=SWF4}*{\S7} that \cite{Taubes:Gr=SW}*{Lemma 5.5} (which is used to prove \cite{Taubes:Gr=SW}*{Proposition 5.1}) is flawed and must be replaced by \cite{Taubes:ECH=SWF4}*{Lemma 7.1}. For the convenience of the reader, here is a further dictionary: \cite{Taubes:Gr=SW}*{Lemma 5.4} corresponds to \cite{Taubes:ECH=SWF4}*{Lemma 4.10}\footnote{There is a typo in the statement of \cite{Taubes:ECH=SWF4}*{Lemma 4.10}: replace the word ``least'' with ``most'' in the last sentence.} while \cite{Taubes:Gr=SW}*{Lemma 5.8} corresponds to \cite{Taubes:ECH=SWF4}*{Lemma 7.2}.
\end{remark}

The proof of \cite{Taubes:Gr=SW}*{Proposition 5.1} can be copied for the following reason. We have already explained in Sections~\ref{SW instantons from multiply covered}+~\ref{Kuranishi structures} that special planes and special tori may be treated on equal footing, in the sense that they are endowed with an embedded fixed-radius tubular neighborhood (through their normal bundle), they are $d$-nondegenerate (for a given multiplicity $d\ge1$), and their associated space $\zZ_0$ (given by Definition~\ref{def:csurface}) is compact. Thus, those aspects of the proof of \cite{Taubes:Gr=SW}*{Proposition 5.1} referring to a multiply covered special torus may now refer to any multiply covered special curve. The remaining aspects of the proof of \cite{Taubes:Gr=SW}*{Proposition 5.1} consist of local arguments which carry over to our scenario in addition to the analogous arguments in \cite{Taubes:ECH=SWF4}*{\S7}.\footnote{The proof of \cite{Taubes:Gr=SW}*{Proposition 5.1} hinges on \cite{Taubes:Gr=SW}*{Proposition 5.3, Lemma 5.10}, particularly the assertion that at least one of the ``special sections'' $\fo$ is nontrivial. Some details to the proof of this assertion are lacking in \cite{Taubes:Gr=SW}, but are found by mimicking the proof of \cite{Taubes:ECH=SWF4}*{Lemma 7.4} (private communication with Taubes). In this regard, we warn the reader that the statement of \cite{Taubes:ECH=SWF4}*{Lemma 7.4} is technically incorrect but fixed by replacing the constraint $\rho\in[\chi \delta_nr_n^{-1/2},\kappa^{-2}]$ with $\rho=\chi \delta_nr_n^{-1/2}$.}
\end{proof}

\begin{proof}[Proof of Lemma~\ref{lem:SWtoGr}]
We start with the following lemma, which is proved more generally in \cite{Hutchings:fieldtheory} for strong symplectic cobordisms. It is the analog of \cite{HutchingsTaubes:Arnold2}*{Proposition 5.2} over exact symplectic cobordisms and of \cite{Taubes:SWtoGr}*{Theorem 1.3} over closed symplectic manifolds.

\begin{lemma}
Let $[A,(\alpha,\beta)]$ be an element of $\M_I(\c,\varnothing;\s_\s,\bar z,\bar\eta)$ for $r$ sufficiently large. Then $E(\c)<2\pi\rho_\s$, and there exists a (possibly broken) $J$-holomorphic curve $\cC$ which contains $\bar z\cup\bar\eta$, has total homology class $\tau_\omega(\s)$, and is asymptotic to the admissible orbit set $\Theta$ determined by $\c$ in Theorem~\ref{thm:generators}.
\end{lemma}

Now thanks to \cite{Gerig:taming}*{Proposition 3.13}, $\cC$ is in fact unbroken. It then also has ECH index $I$, thanks to \cite{Gardiner:gradings}*{Theorem 5.1}. To elaborate on the previous sentence, $\alpha^{-1}(0)$ is close (with respect to the Hausdorff distance) to the image of $\cC$ in $\overline X_0$ for $r$ sufficiently large, and it determines a relative homology class $Z_\alpha\in H_2(\overline X_0,\varnothing,\Theta)$ such that $[\cC]=Z_\alpha$. Thus, $I(\cC)=I(Z_\alpha)$, and the fact that $I(Z_\alpha)=I$ is the content of \cite{Gardiner:gradings}*{Theorem 5.1}.\footnote{The proof of \cite{Gardiner:gradings}*{Theorem 5.1} over symplectic cobordisms mimics that given in \cite{Taubes:ECH=SWF3}*{\S2.b, \S2.c} of the analogous statement over symplectizations.}

The proof of the remaining part of Lemma~\ref{lem:SWtoGr} follows that of \cite{HutchingsTaubes:Arnold2}*{Proposition 7.1}.
\end{proof}

%%%%%%%%%%%%%%%%%%%%%%%%%%%%%%%%%%%%%%%%%%%%%%%%%%%%%%
\subsection{Proof of Theorem~\ref{thm:moduli}}
\label{Proof of Theorem}

\indent\indent
We now combine the results of the previous sections. Especially in this section, we use $\ZZ/2\ZZ$ coefficients and hence ignore orientations.

\begin{proof}[Proof of Theorem~\ref{thm:moduli}]
As a warm-up case, suppose that $\mM$ does not contain any elements with special curve components. Propositions~\ref{prop:Psi} and~\ref{prop:PsiSurjective} then imply that $\Psi_r$ is an honest bijection\footnote{This bijection is the analog of \cite{Taubes:ECH=SWF1}*{Theorem 4.3} over symplectizations and of \cite{Taubes:Gr=SW}*{Proposition 2.6} over closed symplectic manifolds.}
$$\mM\longleftrightarrow\M$$
and the desired result follows immediately.

As a second warm-up case, make the following assumptions: $I=0$, and each space $\zZ_0^{(k')}$ associated with any special curve $(C_{k'},d_{k'})\in\cC\in\mM$ consists of only regular points. Since each $\zZ_0^{(k'')}$ associated with any non-special curve $(C_{k''},1)\in\cC\in\mM$ is identified with $\ker(D_{C_{k''}})$ (see Section~\ref{Vortices and orbits and curves}), the former assumption implies that $\zZ_0^{(k'')}$ is a single point for generic $J$ as needed to make~\ref{def:GromovCycle} well-defined. Since each $\zZ_0^{(k')}$ associated with a special curve is compact by Proposition~\ref{prop:Zcompact}, the latter assumption implies that $\zZ_0^{(k')}$ is a finite set of points. Propositions~\ref{prop:Psi} and~\ref{prop:PsiSurjective} then imply that $\Psi_r$ induces a bijection\footnote{This bijection is the analog of \cite{Taubes:Gr=SW}*{Proposition 2.9} over closed symplectic manifolds.}
$$\bigcup_{\lbrace(C_k,d_k)\rbrace\in\mM}\bigtimes_k\zZ_0^{(k)}\longleftrightarrow\M$$
In order to finish the proof of Theorem~\ref{thm:moduli} in this case, we need to show that the number of points in each $\zZ_0^{(k)}$ is equivalent modulo 2 to $r(C_k,d_k)$, where $r(C,d)$ is the integer weight attached to each component $(C,d)\in\cC$ such that $q(\cC)=\pm\prod_{(C,d)\in\cC}r(C,d)$ (see \cite{Gerig:taming}*{\S3.5}). The proof of such equality will be subsumed in the proof of the general scenario, where $\zZ_0^{(k)}$ may also contain non-regular points.

Consider now the general scenario. Propositions~\ref{prop:Psi} and~\ref{prop:PsiSurjective} then reduce Theorem~\ref{thm:moduli} to the following claim:
\begin{equation}
\label{claim}
\prod_{k'} r(C_{k'},d_{k'})\equiv\left|\psi_{\cC,r}^{-1}(0)\cap\yY_{\bar z,\bar\eta}\right|\;\mod2
\end{equation}
where $k'$ indexes over the set of special curve components $\cC'$. We establish this claim in two steps.

\textit{Step 1.} As explained in \cite{Taubes:Gr=SW}*{\S2.g.2}, there is a well-defined \textit{weight}
$$r'(C_{k'},d_{k'})\in\ZZ/2\ZZ$$
given by the (modulo 2) count of zeros of any smooth perturbation $w:\kK^{(k')}_\Lambda\to\Lambda_{k'}$ of the map $\psi^{(k')}_\Lambda:\kK^{(k')}_\Lambda\to\Lambda_{k'}$ from Lemma~\ref{lem:5.1Gr}, such that $|w|<|\psi^{(k')}_\Lambda|$ on the complement of a compact neighborhood of $\zZ_0^{(k')}$ and such that $\psi^{(k')}_\Lambda+w$ has only nondegenerate zeros. In particular, if $\zZ_0^{(k')}$ is a finite set of regular points then $r'(C_{k'},d_{k'})$ is its cardinality.

As explained in \cite{Taubes:GrtoSW}*{\S5.i.1}, the map $\psi_{\cC,r}$ in~\eqref{eqn:psi} satisfies
$$\left|\psi_{\cC,r}-\bigtimes_{k'}\psi_\Lambda^{(k')}\right|\le \zeta r^{-1/2}$$
for some $r$-independent constant $\zeta$, where the norm on $\bigoplus_{k'}\Lambda_{(k')}$ is the product of the $L^2$-norms on each of the factors. Here, we view $\psi_{\cC,r}$ as a function restricted to $\bigtimes_{k'}\kK_\Lambda^{(k')}$ (which is diffeomorphic to $\yY_{\bar z,\bar\eta}$). Thus, it follows that
$$\prod_{k'} r'(C_{k'},d_{k'})\equiv\left|\psi_{\cC,r}^{-1}(0)\cap\yY_{\bar z,\bar\eta}\right|\;\mod2$$
for $r$ sufficiently large.

\textit{Step 2.} Given the result of Step 1, it remains to show that $r'(C_{k'},d_{k'})=r(C_{k'},d_{k'})$ in order to establish~\eqref{claim} and hence complete the general proof of Theorem~\ref{thm:moduli}.

\begin{prop}
For generic $J$ as needed to make~\ref{def:GromovCycle} well-defined, the weights $r'(C,d)$ and $r(C,d)$ are equal for any $J$-holomorphic special curve $C$ with multiplicity $d$. In particular, their value is 1 when $C$ is a special plane.
\end{prop}

\begin{proof}
If $C$ is a special torus then this is precisely the result of \cite{Taubes:Gr=SW}*{Proposition 2.12, Proposition 2.15}. Similarly to the proof of \cite{Taubes:Gr=SW}*{Proposition 2.12}, if $C$ a special plane then $r'(C,d)$ is unaffected by replacing $\mu_C$ in the definition of $\zZ_0$ with $t\mu_C$ for any $t\in[0,1]$, noting that the resulting space $\zZ_0^t$ of sections is still compact by Proposition~\ref{prop:Zcompact}. In particular, $r'(C,d)$ is equal to the count of points of $\zZ^0_0$. Similarly to the end of the proof of \cite{Taubes:Gr=SW}*{Proposition 2.15}, $\zZ^0_0$ consists of a single regular point (namely, the constant map to the unique ``symmetric'' vortex) because the complex linear operator $\dbar+\nu_C\aleph$ has trivial kernel and cokernel. Thus, $r'(C,d)=1$.
\end{proof}

\end{proof}

%%%%%%%%%%%%%%%%%%%%%%%%%%%%%%%%%%%%%%%%%%%%%%%%%%%%%%
%%%%%%%%%%%%%%%%%%%%%%%%%%%%%%%%%%%%%%%%%%%%%%%%%%%%%%
\section{Relation of Seiberg--Witten counts}
\label{Relation of Seiberg--Witten counts}

\indent\indent
The goal of this section is to finally prove Theorem~\ref{thm:SWGr2} (and thus Theorem~\ref{thm:SWGr1}). It follows from Theorem~\ref{thm:generators} and Theorem~\ref{thm:moduli} using $I=d(\s)$ that, for $r$ sufficiently large, the Gromov cycle $\Phi_{Gr}$ (given by Definition~\ref{def:GromovCycle}) is chain-isomorphic over $\ZZ/2\ZZ$ to the \textit{Seiberg--Witten cocycle}
$$\Phi_{SW}:=\sum_{\Theta}\M_{\c_\Theta}\c_\Theta\in \Cfrom^{g(\s)}(-\partial X_0,\s_{\xi_0}+1)$$
where the sum is over admissible orbit sets $\Theta$ in the grading $g(\s)$ with $[\Theta]=-\partial\tau_\omega(\s)$ and
\begin{equation}
\label{eq:SWcoeff}
\M_{\c_\Theta}:=\sum_{\d\in\M_{d(\s)}(\c_\Theta,\varnothing;\s_\s,\bar z,\bar\eta)}q(\d)\in\ZZ/2\ZZ
\end{equation}
The integer $Gr_{X,\omega}(\s)\big([\eta_1]\wedge\cdots\wedge[\eta_p]\big)$ is thus equal (modulo 2) to the coefficient of the class
$$[\Phi_{SW}]\in\bigotimes_{k=1}^N\Hfrom^{[\xi_*]}(S^1\times S^2,\s_{\xi_0}+1)\cong\ZZ$$
as a multiple of the generator $\1\in \bigotimes^N_{k=1} \Hfrom^{[\xi_*]}(S^1\times S^2,\s_{\xi_0}+1)$, where again we have anticipated $g(\s)=N[\xi_*]$. In fact, this coefficient is the corresponding Seiberg--Witten invariant:

\begin{theorem}
\label{thm:stretchNeck}
Fix $(X,\omega)$ and assume $\s\in\Spinc(X)$ is such that $E\cdot\tau_\omega(\s)\ge-1$ for all $E\in\eE_\omega$. Fix an integer $p$ such that $0\le p\le d(\s)$ and $d(\s)-p$ is even, and fix an ordered set of homology classes $[\bar\eta]:=\big\lbrace[\eta_i],\ldots,[\eta_p]\big\rbrace\subset H_1(X;\ZZ)/\operatorname{Torsion}$. Then
$$Gr_{X,\omega}(\s)\big([\eta_1]\wedge\cdots\wedge[\eta_p]\big)\equiv SW_X(\s)\big([\eta_1]\wedge\cdots\wedge[\eta_p]\big)\;\mod2$$
and $g(\s)=N[\xi_*]$ as an absolute grading of the $N$-fold tensor product of $\Hfrom^*(S^1\times S^2)$.
\end{theorem}

Before we prove this theorem we make the following remarks. The work of Kronheimer--Mrowka (specifically, \cite{KM:book}*{Proposition 27.4.1}) recovers the Seiberg--Witten invariants of $X$ using their monopole Floer (co)homologies, explicitly by removing two copies of the 4-ball $B^4$ from $X$ and counting certain\footnote{They build a ``mixed'' map $\overrightarrow{\mathit{HM}}:\Hfrom_*(S^3)\to\Hto^*(S^3)$ and pair the image of a generator $\1\in\Hfrom_*(S^3)$ with a generator $\check\1\in\Hto^*(S^3)$. A homology orientation of the cobordism is identified with a homology orientation of $X$.} SW instantons on the resulting cobordism $S^3\to S^3$. Although not provided in \cite{KM:book}, the same result could have been obtained by removing two copies of $S^1\times B^3$, because both spaces ($B^4$ and $S^1\times B^3$) have positive scalar curvature and a unique spin-c structure extending the fixed torsion spin-c structure on their boundary (see \cite{KM:book}*{Proposition 22.7.1}).\footnote{It is important here to note that every self-diffeomorphism of $S^1\times S^2$ (and $S^3$) extends to $S^1\times B^3$ (and $B^4$).} There would necessarily be more work to do when using $S^1\times B^3$ because there exists a circle's worth of reducible monopoles on $S^1\times S^2$ (to the unperturbed Seiberg--Witten equations), compared with a single reducible monopole on $S^3$.\footnote{We must also choose a homology orientation of $S^1\times S^2$, i.e. an orientation of the vector space $H^1(S^1\times S^2;\RR)\cong\RR$, in order to identify a homology orientation of $X$ with that on the cobordism.}

In the proof of Theorem~\ref{thm:stretchNeck} we will remove $N$ copies of $S^1\times B^3$ from $X$, namely, the tubular neighborhoods of the zero-circles. This is a ``neck stretching'' argument along the contact hypersurfaces $(S^1\times S^2,\lambda_\s)$ in $X$, and we analyze the Seiberg--Witten equations under this deformation. It is important to note that on $X_0$ we can use Taubes' large perturbations to the Seiberg--Witten equations, for which there are no reducible solutions, but on each $S^1\times B^3$ we cannot do this because there is no symplectic form (or said another way, the near-symplectic form $\omega$ degenerates somewhere inside $S^1\times B^3$). We thus interpolate, on the ``neck region'' of $X$, between Taubes' large perturbations on $X_0$ and very small perturbations on each $S^1\times B^3$, where the small perturbations are chosen in such a way that we can understand the SW instantons on each $S^1\times B^3$ completely.

A final remark is that in this setup, we do not run into the usual difficulties that Kronheimer--Mrowka have when defining monopole Floer cobordism maps for cobordisms with disconnected and empty ends. These difficulties are ultimately due to the (stratified) space of reducible monopoles on the (positive and negative) boundary components of the cobordism, and are avoided in our setup thanks to Taubes' large perturbations.

\begin{proof}[Proof of Theorem~\ref{thm:stretchNeck}]
We closely follow the arguments in \cite{KM:book}*{\S26, \S27.4, \S36.1} that recover the Seiberg--Witten invariant and establish the composition law for monopole Floer (co)homology. These arguments involve judicious choices of Riemannian metrics and abstract perturbations on $X$ to ``stretch the neck'' and compare the resulting moduli spaces of SW instantons.

Let $\TT$ denote the circle $H^1(S^1\times S^2;i\RR)/H^1(S^1\times S^2;2\pi i\ZZ)\cong S^1$ which parametrizes reducible monopoles to the unperturbed Seiberg--Witten equations over $S^1\times S^2$. In fact, all monopoles are reducible because $S^1\times S^2$ has a metric of positive scalar curvature (see \cite{KM:book}*{Proposition 22.7.1}). After fixing a reference connection $\bA_0$ on $\det\SS$ so that any other Hermitian connection can be written as $\bA=\bA_0+2a$ for some $a\in\Omega^1(Y;i\RR)$, there is a retraction map
$$p:\bB(S^1\times S^2,\s_{\xi_0}+1)\to\TT$$
sending $[\bA_0+2a,\Psi]$ to the equivalence class of the harmonic part $a_\text{harm}$ of $a$ (see \cite{KM:book}*{\S11.1}).

Let $f$ be the ``height'' Morse function on $\TT$ with two critical points, and let $f_1=f\circ p$ be the corresponding function on $\bB(S^1\times S^2,\s_{\xi_0}+1)$. The gradient of $f_1$ is an abstract perturbation $\q_f:=\grad f_1$ (assumed small by re-scaling $f$), and the reducible critical points of $\grad\lL_\text{CSD}+\q_f$ are the maximum and minimum critical points $\lbrace\alpha,\beta\rbrace$ of $f$ on $\TT$. The perturbed Dirac operators associated with $\alpha$ and $\beta$ in the blow-up $\bB^\sigma(S^1\times S^2,\s_{\xi_0}+1)$ do not have kernel, but to guarantee that their spectrums are simple we add a further small perturbation to $\q_f$ (still denoted $\q_f$) which vanishes on $\bB^\text{red}(S^1\times S^2,\s_{\xi_0}+1)$. Label the corresponding critical points in $\bB^\sigma(S^1\times S^2,\s_{\xi_0}+1)$ as $\a_i$ and $\b_i$ in increasing order of the index, where $\a_0$ and $\b_0$ correspond to the first positive eigenvalues of the perturbed Dirac operator at $\alpha$ and $\beta$ (the critical points are boundary-stable for $i\ge0$ and boundary-unstable for $i<0$).

Consider a component $\nN_k\approx S^1\times B^3$ of $\nN=\bigsqcup_{k=1}^N\nN_k$, equipped with a metric having positive scalar curvature and containing a collar region of its boundary in which the metric is cylindrical. Choose a small perturbation $\p_\nN$ on $\overline{\nN_k}$ equal to $\q_f$ on the end, so that the corresponding moduli spaces $M(\varnothing,\nN_k,\a_i;\s)$ and $M(\varnothing,\nN_k,\b_i;\s)$ are regular (this is possible by \cite{KM:book}*{Proposition 24.4.7}). Here, $\s$ on $\nN_k$ is the unique spin-c structure which extends $\s_{\xi_0}+1$ on its boundary $S^1\times S^2$.

\bigskip
Let $X(T)$ be the closed manifold (diffeomorphic to $X$) obtained by attaching, for each end of $X_0$, two copies of the cylinder $[0,T]\times S^1\times S^2$ and one copy of another cylinder $[0,1]\times S^1\times S^2$ and the $k$th component $\nN_k$ of $\nN$ (see Figure~\ref{fig:stretch}):
\begin{figure}
    \centering
    \includegraphics[width=7cm]{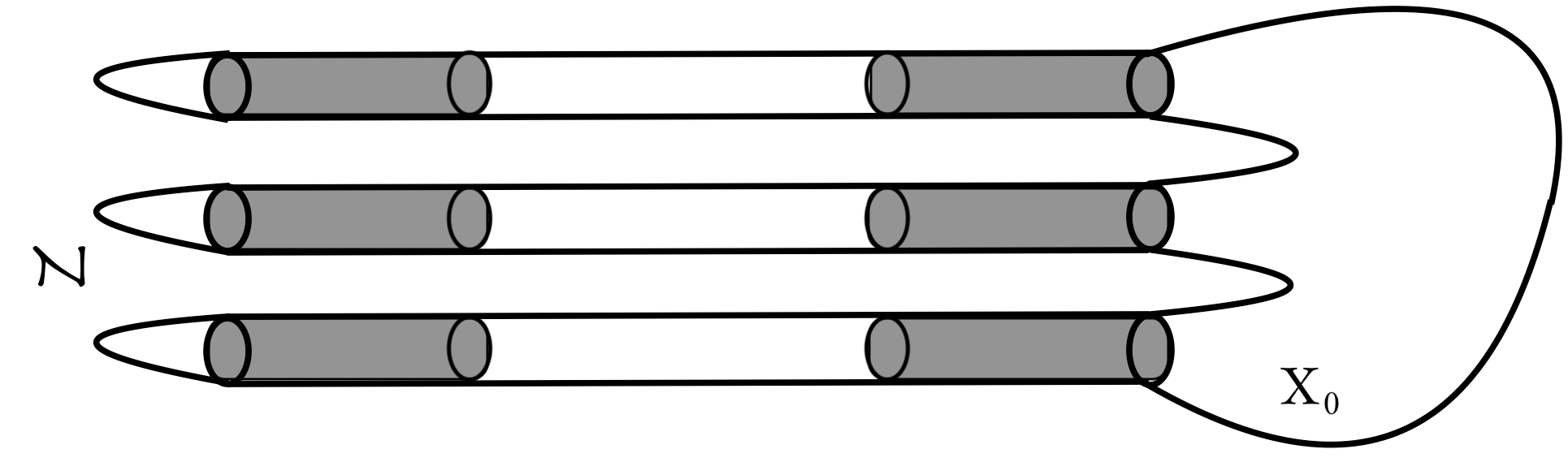}
    \caption{``Stretching the necks'' of $X$, depicted in gray}
    \label{fig:stretch}
\end{figure}
$$X(T):=\nN\cup\bigcup_{k=1}^N\Big(\left([0,T]\times S^1\times S^2\right)\cup\left([0,1]\times S^1\times S^2\right)\cup\left([0,T]\times S^1\times S^2\right)\Big)\cup X_0$$
The perturbed Seiberg--Witten equations on $X(T)$ carry the following perturbations:
\begin{itemize}
  \setlength{\itemsep}{1pt}
  \setlength{\parskip}{0pt}
  \setlength{\parsep}{0pt}
  
  \item Taubes' perturbation $\p_\omega$ on $X_0$ which extends over the adjacent copies of $[0,T]\times S^1\times S^2$ using Taubes' perturbation $\q_\lambda$ on $\partial X_0$,
 
   \item the perturbation $\p_\nN$ on each $\nN_k$ which extends over the adjacent copies of $[0,T]\times S^1\times S^2$ using the perturbation $\q_f$ on $\partial\nN_k$,
	
   \item an ``interpolating'' perturbation $\p_\text{cyl}$ on each copy of $[0,1]\times S^1\times S^2$ which agrees with $\q_\lambda$ near $\lbrace 0\rbrace\times S^1\times S^2$ and with $\q_f$ near $\lbrace 1\rbrace\times S^1\times S^2$.
\end{itemize}
To simplify notation we write $\iI_k$ for the $k$th copy of $[0,1]\times S^1\times S^2$ and $\iI:=\bigsqcup_{k=1}^N\iI_k$. Consider the moduli space $\M(X(T),\s)$ for the manifold equipped with this perturbation. As $T\in[0,\infty)$ varies, these
form a parametrized moduli space
$$\mM(X,\s):=\bigcup_{T\in[0,\infty)}\lbrace T\rbrace\times\M(X(T),\s)$$
This has a compactification
$$\mM^+(X,\s):=\bigcup_{T\in[0,\infty]}\lbrace T\rbrace\times\M(X(T),\s)$$
formed by attaching a fiber at $T=\infty$, where $\M(X(\infty),\s)$ is defined to be the set of quintuples $(\d_0,\breve\d_1,\d_2,\breve\d_3,\d_4)$ such that
\begin{align*}
\d_0 &\in M(\varnothing,\nN,\c_{i_1};\s)\\
\breve\d_1 &\in\breve M^+(\c_{i_1},\c_{i_2})\\
\d_2 &\in M(\c_{i_2},\mathcal{I},\c_{i_3};\s_{\xi_0}+1)\\
\breve\d_3 &\in\breve M^+(\c_{i_3},\c_{i_4})\\
\d_4 &\in \M(\c_{i_4},X_0,\varnothing;\s_\s)
\end{align*}
The space $\mM^+(X,\s)$ is stratified by manifolds, its codimesion-1 strata consisting of the fiber $\M(X,\s)$ over $T=0$ and those strata over $T=\infty$ with $\c_{i_1}=\c_{i_2}$ and $\c_{i_3}=\c_{i_4}$ (so that $\breve\d_1$ and $\breve\d_3$ belong to point moduli spaces). The latter strata are of the form
$$M(\varnothing,\nN,\c_1;\s)\times M(\c_1,\iI,\c_2;\s_{\xi_0}+1)\times \M(\c_2,X_0,\varnothing;\s_\s)$$
where each of $\c_1=\lbrace\c_1^k\rbrace_{1\le k\le N}$ and $\c_2=\lbrace\c_2^k\rbrace_{1\le k\le N}$ is a critical point over $S^1\times S^2$ associated with the perturbations $\q_f$ and $\q_\lambda$, respectively.

\bigskip
We now incorporate the point and loop constraints that are used to define the closed Seiberg--Witten invariant. Since the constraints $\bar z\cup\bar\eta$ sit inside $X_0\subset X$ and $X$ may be written as the composition of cobordisms\footnote{There is no ambiguity in the resulting 4-manifold because every self-diffeomorphism of $S^1\times S^2$ extends to $S^1\times B^3$.} $X_0\circ\iI\circ\nN$, we may decompose the element
$$u:=U^{\frac{1}{2}(d(\s)-p)}[\eta_1]\wedge\cdots\wedge[\eta_p]\in\AA(X)$$
as the product
$$u=R_\nN^*(1)\smile R_\iI^*(1)\smile R_{X_0}^*(u_0)$$
where each $R_{\square}:\bB^\sigma(X)\dashrightarrow\bB^\sigma(\square)$ is the partially-defined restriction map and the class $u_0\in H^{d(\s)}(\bB^\sigma(X_0);\ZZ)$ is induced by $\bar z\cup\bar\eta$. Here, the product operation is defined in \cite{KM:book}*{\S23.2, \S26.2} and given in terms of {\v C}ech cocycle representatives as follows. Fix suitable\footnote{The word ``suitable'' in this context means that the open covers are transverse (in the sense of \cite{KM:book}*{\S21.2}) to all strata in all compactified moduli spaces under consideration.} open covers $\lbrace\uU_0, \uU_\iI, \uU_\nN\rbrace$ of $\lbrace\bB^\sigma(X_0), \bB^\sigma(\iI), \bB^\sigma(\nN)\rbrace$, fix a suitable refinement $\vV$ of the open cover $[0,\infty]\times\uU_\nN\times\uU_\iI\times\uU_0$ of $[0,\infty]\times\bB^\sigma(\nN)\times\bB^\sigma(\iI)\times\bB^\sigma(X_0)$, let $\bB^\sigma(X)^o$ be the domain on which each $R_{\square}$ is defined, and let $\uU$ be the pull-back of $\vV$ under the map $\bB^\sigma(X)^o\to\lbrace0\rbrace\times\bB^\sigma(\nN)\times\bB^\sigma(\iI)\times\bB^\sigma(X_0)$. After fixing a cocycle $\fu_0(\bar z,\bar\eta)\in \check C^{d(\s)}(\uU_0;\ZZ/2\ZZ)$ which represents the class $u_0$, the pull-back of the cocycle $1\times1\times\fu_0(\bar z,\bar\eta)\in\check C^{d(\s)}(\vV;\ZZ/2\ZZ)$ is a cocycle $\fu\in\check C^{d(\s)}(\uU;\ZZ/2\ZZ)$ which represents the class $u$. We will specify the choice of $\fu_0(\bar z,\bar\eta)$ momentarily.

There is a continuous map defined in \cite{KM:book}*{\S26.1},
$$r:\mM^+(X,\s)\to[0,\infty]\times\bB^\sigma(\nN,\s)\times\bB^\sigma(\iI,\s_{\xi_0}+1)\times\bB^\sigma(X_0,\s_\s)$$
which is given by $(T,\d)\mapsto(T,\d|_\nN,\d|_\iI,\d|_{X_0})$ for $T<\infty$, in particular. The image $r\left(\mM^+(X,\s)\right)$ is also stratified by manifolds, and the only relevant strata which pair nontrivially with the cocycle $1\times1\times\fu_0(\bar z,\bar\eta)$ are determined by
\begin{align}
\label{align:N}
\dim M(\varnothing,\nN_k,\c_1^k;\s) &=0\\
\label{align:I}
\dim M(\c_1^k,\iI_k,\c_2^k;\s_{\xi_0}+1) &=0\\
\label{align:X}
\dim\M(\c_2,X_0,\varnothing;\s_\s) &=d(\s)
\end{align}
for each $k$. The dimensions of $M(\varnothing,\nN_k,\c_1^k;\s)$ are computed in Lemma~\ref{lem:Nmoduli} below, from which it follows that~\eqref{align:N} forces $\c_1^k=\a_{-1}$. Then~\eqref{align:I} forces $|\c_2^k|=|\c_1^k|$, which follows immediately from the definition of the grading (see \cite{KM:book}*{\S22.3}), or less directly from the fact that such a product cobordism induces an isomorphism on all monopole Floer (co)homologies. As explained in \cite{KM:book}*{\S36}, $\Hfrom^{[\xi_*]}(S^1\times S^2,\s_{\xi_0}+1)$ is isomorphic to $\Hto_{[\xi_*]}(S^1\times S^2,\s_{\xi_0}+1)$, with the former generated by $\a_{-1}$ and the latter generated by $\b_0$.\footnote{Both $\a_{-1}$ and $\b_0$ belong to the same absolute grading in $J(S^1\times S^2,\s_{\xi_0}+1)$ because their $\ZZ$-grading difference is $\gr[\a_{-1},\b_0]=0$ (see \cite{KM:book}*{Equation 16.9, Equation 36.1}).} Thus, each $\c_2^k$ must be one of the finitely many irreducible generators $\c\in\Cfrom^{[\xi_*]}(S^1\times S^2,\s_{\xi_0}+1,\q_\lambda)$ satisfying~\eqref{align:X}, and hence $g(\s)=N[\xi_*]$.

\begin{notation}
Here and in what follows, we abuse notation by letting $\c$ and $\a_{-1}$ denote either the respective monopoles $\c^k$ and $\a_{-1}$ on a single component of $\partial X_0$ or the respective collections $\lbrace\c^k\rbrace_{1\le k\le N}$ and $\lbrace\a_{-1}\rbrace_{1\le k\le N}$. Also, we refer to $\c$ as both a monopole and a cochain in $\Cfrom^*$ while $\hat\c$ denotes the corresponding chain in $\Cfrom_*$ which pairs nontrivially with the cochain $\c$, i.e. $\c(\hat\c)=1$.
\end{notation}

As explained in \cite{KM:book}*{\S21}, there is a $\ZZ/2\ZZ$-pairing (denoted by $\langle\cdot,\cdot\rangle$) between {\v C}ech cochains and our moduli spaces. The version of Stokes' theorem in \cite{KM:book}*{\S21}, applied to the $\ZZ/2\ZZ$-pairing of $r\left(\mM^+(X,\s)\right)$ with $\delta(1\times1\times\fu_0(\bar z,\bar\eta))=0$, implies
\begin{equation}
\label{eqn:SW}
\begin{split}
SW_X(\s)\big([\eta_1]\wedge\cdots\wedge[\eta_p]\big) &= \big\langle u,[\M(\s)]\big\rangle = \big\langle \fu,\M(\s)\big\rangle\\
&= \sum_{\c\in N[\xi_*]}\big\langle 1,M(\a_{-1},\iI,\c;\s_{\xi_0}+1)\big\rangle\cdot\big\langle \fu_0(\bar z,\bar\eta),\M(\c,X_0,\varnothing;\s_\s)\big\rangle\\
&=: \sum_{\c\in N[\xi_*]}\MM_\c\M_\c
\end{split}
\end{equation}
where we use the fact (Lemma~\ref{lem:Nmoduli} below) that the moduli space $M(\varnothing,\nN_k,\a_{-1};\s)$ is a single point. Here, $\MM_\c$ is the count of points in the 0-dimensional moduli space $M(\a_{-1},\iI,\c;\s_{\xi_0}+1)$, while $\M_\c$ is computed by specifying the choice of $\fu_0(\bar z,\bar\eta)$. As explained in \cite{HFHM5}*{\S2.5.4}, there is a choice such that $\M_\c$ is the count of points in $\M_{d(\s)}(\c,X_0,\varnothing;\s_\s,\bar z,\bar\eta)$, i.e. the number defined by~\eqref{eq:SWcoeff} (note that by Lemma~\ref{lem:SWtoGr}, $\c$ satisfies $E(\c)<2\pi\rho_\s$ and hence corresponds to an admissible orbit set).

\bigskip
We now reinterpret the sum in~\eqref{eqn:SW} to see that it equals the near-symplectic Gromov invariant. The $k$th monopole Floer chain complex in grading $[\xi_*]$ is $\Cfrom_{[\xi_*]}(S^1\times S^2,\s_{\xi_0}+1,\q_f)=\ZZ/2\ZZ\langle\hat\a_{-1}\rangle$. The cobordism $\iI_k$ induces a chain map
$$\hat m_k:\Cfrom_*(S^1\times S^2,\s_{\xi_0}+1,\q_f)\to\Cfrom_*(S^1\times S^2,\s_{\xi_0}+1,\q_\lambda)$$
which is a quasi-isomorphism \cite{KM:book}*{Corollary 23.1.6}, and in grading $[\xi_*]$ it is given by 
$$\hat\a_{-1}\mapsto \sum_{\c^k\in[\xi_*]}\MM_{\c^k}\hat\c^k$$
where $\MM_{\c^k}$ denotes the count of points in the 0-dimensional moduli space $M(\a_{-1},\iI_k,\c^k;\s_{\xi_0}+1)$. Thus, $\iI$ induces the chain map
$$\hat m=\bigotimes^N_{k=1}\hat m_k:\bigotimes^N_{k=1}\Cfrom_*(S^1\times S^2,\s_{\xi_0}+1,\q_f)\to\bigotimes^N_{k=1}\Cfrom_*(S^1\times S^2,\s_{\xi_0}+1,\q_\lambda)$$
which is also a quasi-isomorphism, and in grading $N[\xi_*]$ it sends $\hat\a_{-1}$ to $\sum_{\c\in N[\xi_*]}\MM_\c\hat\c$ because $\MM_\c=\prod_{k=1}^N\MM_{\c^k}$. Since $[\hat m(\hat\a_{-1})]$ is the generator of $\bigotimes^N_{k=1}\Hfrom_{[\xi_*]}(S^1\times S^2,\s_{\xi_0}+1)\cong\ZZ/2\ZZ$, the near-symplectic Gromov invariant is the evaluation of the Seiberg--Witten cocycle $\Phi_{SW}\in\bigotimes^N_{k=1}\Cfrom^{[\xi_*]}(S^1\times S^2,\s_{\xi_0}+1,\q_\lambda)$ at the cycle $\hat m(\hat\a_{-1})$,
$$Gr_{X,\omega}(\s)\big([\eta_1]\wedge\cdots\wedge[\eta_p]\big)\equiv_{(2)}\Phi_{SW}\left(\sum_{\c\in N[\xi_*]}\MM_\c\hat\c\right)=\sum_{\c\in N[\xi_*]}\MM_\c\M_\c$$
This number is precisely that in~\eqref{eqn:SW}, so the proof is complete.
\end{proof}

\begin{lemma}
\label{lem:Nmoduli}
For sufficiently small perturbations $\p_\nN$ and $\q_f$ specified in the proof of Theorem~\ref{thm:stretchNeck}, the moduli spaces $M(\varnothing,S^1\times B^3,\a_i;\s)$ and $M(\varnothing,S^1\times B^3,\b_i;\s)$ are empty for $i\ge0$. The moduli space $M(\varnothing,S^1\times B^3,\a_{-i};\s)$ has dimension $2i-2$ for $i\ge1$, such that $M(\varnothing,S^1\times B^3,\a_{-1};\s)$ is a point, and the moduli space $M(\varnothing,S^1\times B^3,\b_{-i};\s)$ has dimension $2i-1$ for $i\ge1$.
\end{lemma}

\begin{proof}
We mimic the analogous proof for a 4-ball $B^4$ in \cite{KM:book}*{Lemma 27.4.2}. The key point is that both $S^1\times B^3$ and $B^4$ have metrics of positive scalar curvature and have trivial 2nd (co)homology. We have already discussed the $\q_f^\sigma$-perturbed SW monopoles over $(S^1\times S^2,\s_{\xi_0}+1)$ at the beginning of Theorem~\ref{thm:stretchNeck}, and we continue to use that notation.

With respect to the unperturbed Seiberg--Witten equations (without blowing up) over $\overline{S^1\times B^3}$, there are no solutions with nonzero spinor $\Phi$ that decay to zero on the cylindrical end $[0,\infty)\times S^1\times S^2$ because the scalar curvature is positive (see the integration-by-parts trick of \cite{KM:book}*{Proposition 4.6.1}). For $\p_\nN$ sufficiently small the reducible solutions persist and are asymptotic to $\alpha=[\bA_\alpha,0]$ or $\beta=[\bA_\beta,0]$.

Now, the restriction of a reducible solution $\d\in M(\varnothing,S^1\times B^3,\a_i;\s)$ to the cylindrical end is a path
$$\check\d(t)=[\bA_\a,0,\psi(t)]$$
with $\psi(t)$ approaching the $S^1$-orbit of $\psi_i$ as $t\to\infty$, where $\a_i=[\bA_\alpha,0,\psi_i]$. Following the argument of \cite{KM:book}*{Proposition 14.6.1}, we would then obtain a nonzero solution to the perturbed Dirac equation on $\overline{S^1\times B^3}$ with asymptotics $Ce^{-\lambda_it}\psi_i$ (nonzero constant $C$) as $t\to\infty$ on the cylindrical end. If $\lambda_i>0$ then we just argued that such spinors cannot exist, so $M(\varnothing,S^1\times B^3,\a_i;\s)$ is empty. If $\lambda_{-i}<0$ then, as in \cite{KM:book}*{Proposition 14.6.1}, such spinors with growth bound $Ce^{-\lambda_{-i}t}$ have the form
$$\sum_{k=-i}^{-1}c_ke^{-\lambda_kt}\psi_k$$
on the cylindrical end. Thus, $\dim M(\varnothing,S^1\times B^3,\a_{-i};\s)=2(i-1)$ and $M(\varnothing,S^1\times B^3,\a_{-1};\s)$ is a single point. 

On the other hand, the restriction of a reducible solution $\d\in M(\varnothing,S^1\times B^3,\b_i;\s)$ to the cylindrical end is a path
$$\check\d(t)=[\bA(t),0,\psi(t)]$$
with $\bA(t)$ a trajectory lying over a Morse flowline of $f$ on $\TT\subset\bB(S^1\times S^2,\s_{\xi_0}+1)$ and asymptotic to $\bA_\beta$. We compute $\dim M(\varnothing,S^1\times B^3,\b_{-i};\s)$ for $i\ge1$ indirectly, thanks to the formal dimension formula
$$\dim M(\varnothing,S^1\times B^3,\b_{-i};\s)=\dim M(\varnothing,S^1\times B^3,\a_{-i};\s)+\gr[\a_{-i},\b_{-i}]$$
given by \cite{KM:book}*{Proposition 24.4.6}. Since $\gr[\a_{-i},\b_{-i}]=\ind_f(\alpha)-\ind_f(\beta)=1$, the dimension of $M(\varnothing,S^1\times B^3,\b_{-i};\s)$ must be $2(i-1)+1$. Finally, by \cite{KM:book}*{Proposition 24.4.3} we see that $M(\varnothing,S^1\times B^3,\b_i;\s)$ must be empty for $i\ge0$ because $\b_i$ is boundary-stable and there are no irreducible $\p_\nN$-perturbed SW instantons over $S^1\times B^3$.
\end{proof}

\begin{remark}
The following heuristic was suggested by Mrowka and can be made precise. The holonomy map
$$\operatorname{hol}:\M(X,\s)\to \prod_{k=1}^NU(1)$$
along the $N$ zero-circles of $\omega$ is cobordant (in some sense) to the restriction map
$$\operatorname{res}:M(X_0,\s_\s)\to\prod_{k=1}^N \bB^\text{red}(S^1\times S^2,\s_{\xi_0}+1)$$
where $M(X_0,\s_\s)$ is the moduli space of Seiberg--Witten solutions to the unperturbed equations on $X_0$ without blowing up. Here, $\bB^\text{red}(S^1\times S^2,\s_{\xi_0}+1)$ is identified with $U(1)$ by taking the holonomy of a flat connection along the $S^1$-factor of $S^1\times S^2$.
\end{remark}

%%%%%%%%%%%%%%%%%%%%%%%%%%%%%%%%%%%%%%%%%%%%%%%%%%%%%%%%%%
\subsection*{Acknowledgements}

\indent\indent
The author is indebted to his advisor Michael Hutchings as well as Clifford Taubes -- their roles in this work are evident, and their guidance pivotal. The author also thanks Tomasz Mrowka for helping to understand certain aspects of monopole Floer homology, and the anonymous referee for suggesting many clarifications/corrections to the original draft of this paper. This paper forms part of the author's Ph.D. thesis. The author was partially supported by NSF grants DMS-1406312, DMS-1344991, DMS-0943745, and DMS-1708899.

%%%%%%%%%%%%%%%%%%%%%%%%%%%%%%%%%%%%%%%%%%%%%%%%%%%%%%%%%%
%%%%%%%%%%%%%%%%%%%%%%%%%%%%%%%%%%%%%%%%%%%%%%%%%%%%%%%%%%
\appendix
\section{Appendix: some explanation to Taubes' analysis}
\label{Appendix}

\indent\indent
This paper can be viewed as a sort of amalgam of \cite{Taubes:ECH=SWF1} and \cite{Taubes:SWGrBook}. There was a complicated feature of \cite{Taubes:SWGrBook} that did not arise in \cite{Taubes:ECH=SWF1}, and a complicated feature of \cite{Taubes:ECH=SWF1} that did not arise in \cite{Taubes:SWGrBook}, and both features appeared in this paper. Namely, the multiply covered tori in \cite{Taubes:SWGrBook} had to be delicately counted and the map to the Seiberg--Witten moduli space required the use of Kuranishi structures, and $\RR$-invariance in \cite{Taubes:ECH=SWF1} played a complicating role for the analysis associated with non-$\RR$-invariant holomorphic curves. This latter complication came from the existence of multiple ends of a curve hitting the same orbit, or a single end hitting an orbit with multiplicity. An elaboration is given in \cite{Taubes:ECH=SWF1}*{\S 5.c.1}, and a slightly different elaboration is given below.

In the compact scenario, in order to build a Seiberg--Witten solution from a pseudoholomorphic curve $C$ with multiplicity $n$, an $n$-vortex solution is ``grafted'' into the normal bundle $N_C$, and a disk-subbundle of $N_C$ is embedded into the ambient 4-manifold. In the noncompact scenario, if we were to mimic the previous sentence for a pseudoholomorphic curve $C$ with $n>1$ ends approaching a single orbit $\gamma$, then a 1-vortex solution would be ``grafted'' into $N_C$ for each end of $C$. But to ensure embeddedness of a disk-subbundle the radii of its disk-fibers would need to shrink as $\gamma$ is approached. Subsequently, the Dirac operator evaluated at the corresponding spinor would involve derivatives of the radial coordinate of the fibers, and this would ultimate prevent us from getting the appropriate bounds on the spinor (as needed to obtain a nearby Seiberg--Witten solution).

To resolve this issue, a 1-vortex solution is not ``grafted'' into $N_C$ for each of its $n$ ends. Instead, consider the normal bundle $N_{\RR\times\gamma}$ of the cylinder $\RR\times\gamma$. Then $N_C$ and $N_{\RR\times\gamma}$ are ``nearby'' to each other along the ends of $C$ and $\RR\times\gamma$, and objects defined on them can be compared using cutoff-functions and a change of variables. The (ends of the) curve $C$ intersects any given fiber of $N_{\RR\times\gamma}$ in $n$ points, and an $n$-vortex solution is ``grafted'' into each fiber of $N_{\RR\times\gamma}$ whose zeros are those $n$ points. This solution is then compared to the would-be solution from the original approach, and is seen to be approximately the same except for the worry of varying radial coordinates.

The methodology of the previous paragraph is inspired by the following fact in vortex theory: Given two 1-vortices spaced far apart in $\RR^2$ and one 2-vortex in $\RR^2$ whose zeros are located at the two 1-vortices, the difference between the pair of 1-vortices and the single 2-vortex is exponentially small with respect to the distance between the two 1-vortices. See Figure~\ref{fig:vortex} for a qualitative visual.

\begin{figure}
    \centering
    \includegraphics[width=3cm]{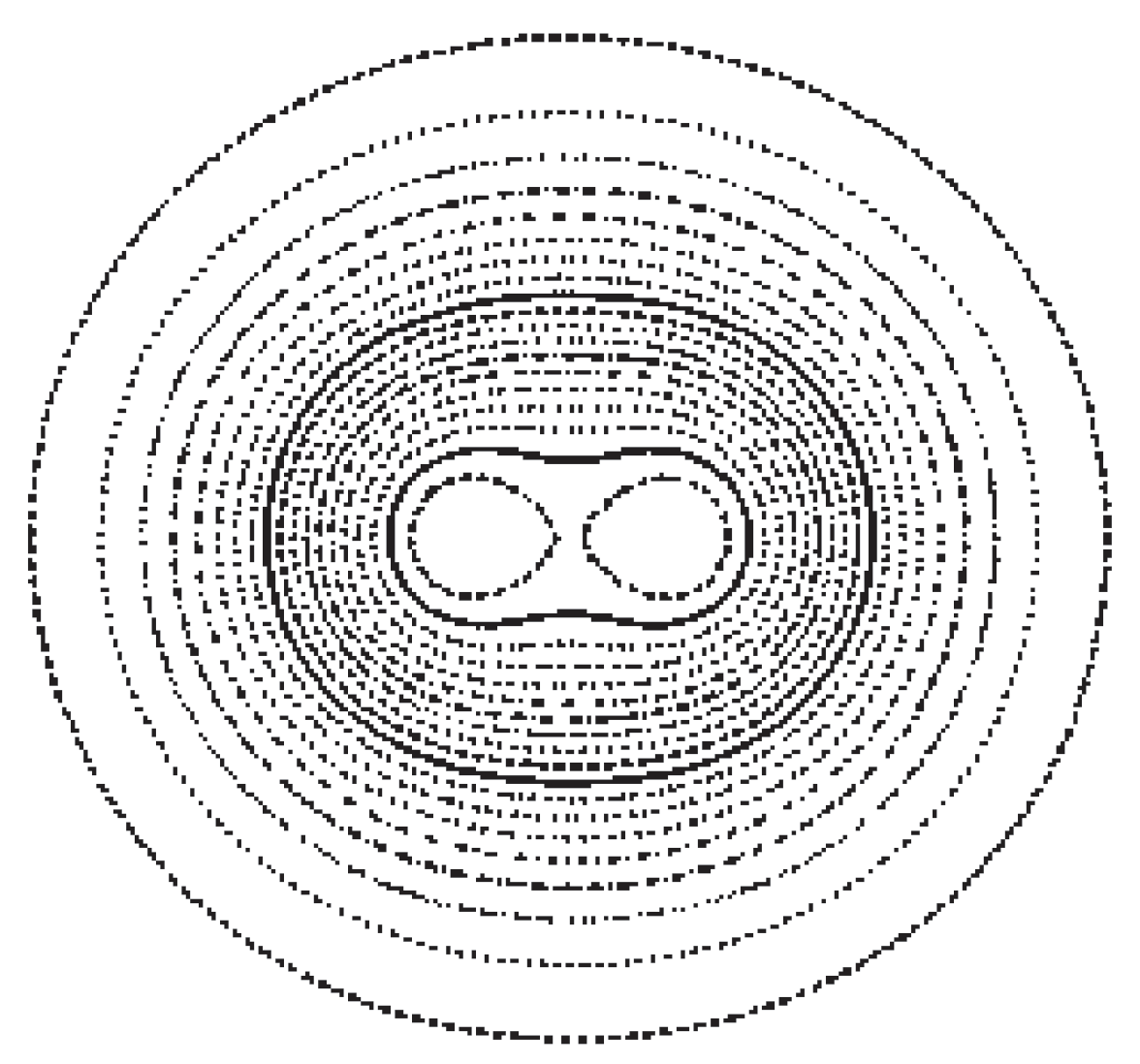}\includegraphics[width=3.5cm]{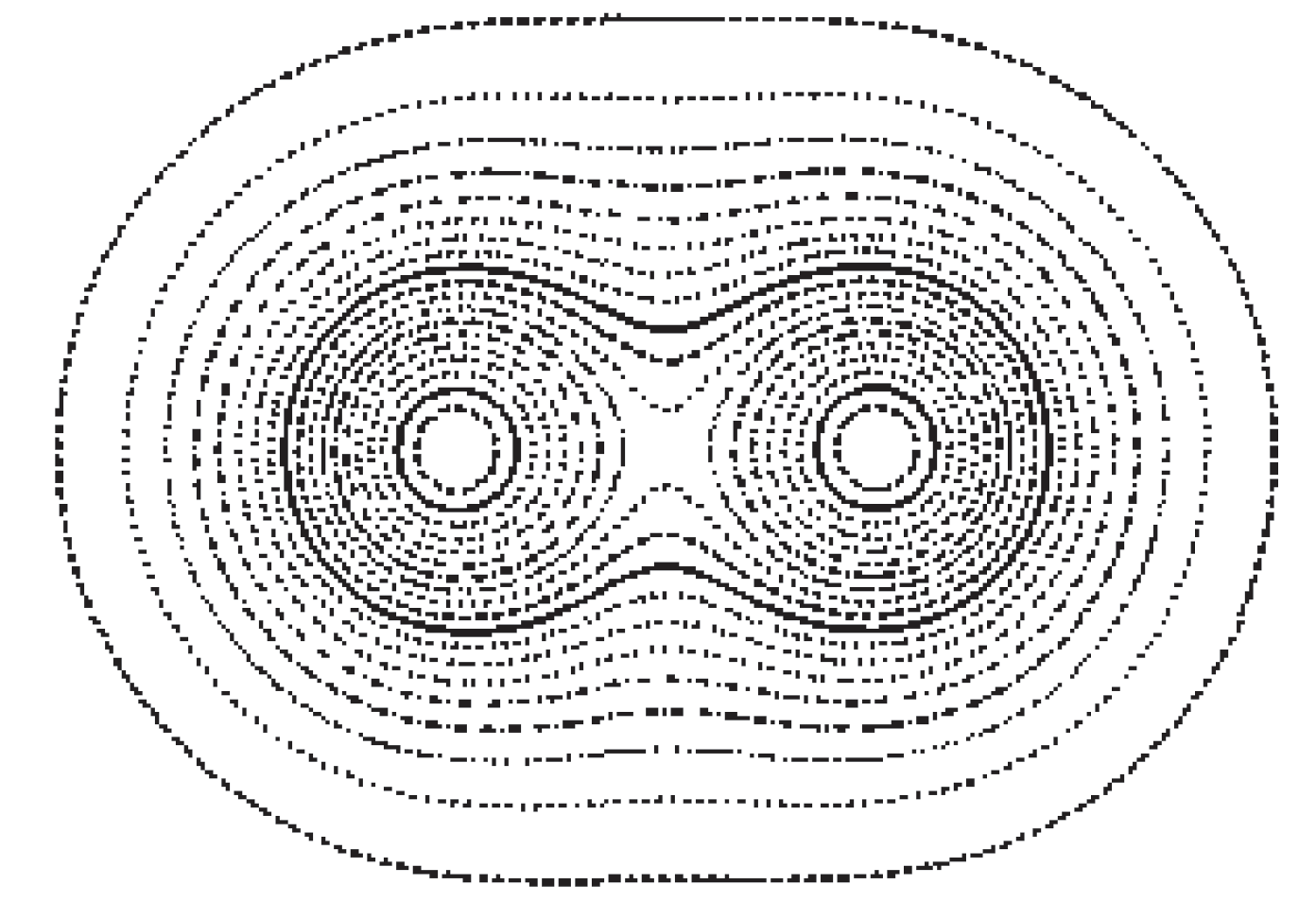}
    \caption{Zeros of a single 2-vortex with different separations (Adapted from \cite{MantonSutcliffe})}
    \label{fig:vortex}
\end{figure}

%%%%%%%%%%%%%%%%%%%%%%%%%%%%%%%%%%%%%%%%%%%%%%%%%%%%%%%%%%
%%%%%%%%%%%%%%%%%%%%%%%%%%%%%%%%%%%%%%%%%%%%%%%%%%%%%%%%%%

\end{document}